\documentclass{amsart}

\usepackage{amsmath,amssymb,graphicx,stmaryrd}

\usepackage{./boldfonts}
\usepackage{./macro}

\begin{document}

\title[The Laplace-Beltrami Operator]{Finite Element Methods for the \\ Laplace-Beltrami Operator}
\author{Andrea Bonito}
\address{Department of Mathematics, Texas A\&M University, College Station, TX 77843, USA; bonito@math.tamu.edu}%
\thanks{AB is partially supported by NSF grant DMS-1817691.}
\author{Alan Demlow}%
\address{Department of Mathematics, Texas A\&M University, College Station, TX 77843, USA; demlow@math.tamu.edu}%
\thanks{AD is partially supported by NSF grant DMS-1720369.}
\author{Ricardo H. Nochetto}%
\address {University of Maryland, College Park, MD 20742, USA; rhn@math.umd.edu}%
\thanks{RHN is partially supported by NSF grant DMS--1411808.}

\begin{abstract}
Partial differential equations posed on surfaces arise in a number of applications.  In this survey we describe three popular finite element methods for approximating solutions to the Laplace-Beltrami problem posed on an $n$-dimensional surface $\gamma$ embedded in $\mathbb{R}^{n+1}$:  the {\it parametric}, {\it trace}, and {\it narrow band} methods.  The parametric method entails constructing an approximating polyhedral surface $\Gamma$ whose faces comprise the finite element triangulation.  The finite element method is then posed over the approximate surface $\Gamma$ in a manner very similar to standard FEM on Euclidean domains.  In the trace method it is assumed that the given surface $\gamma$ is embedded in an $n+1$-dimensional domain $\Omega$ which has itself been triangulated.    An $n$-dimensional approximate surface $\Gamma$ is then constructed roughly speaking by interpolating $\gamma$ over the triangulation of $\Omega$, and the finite element space over $\Gamma$ consists of the trace (restriction) of a standard finite element space on $\Omega$ to $\Gamma$.  In the narrow band method the PDE posed on the surface is extended to a triangulated $n+1$-dimensional band about $\gamma$ whose width is proportional to the diameter of elements in the triangulation.  In all cases we provide optimal a priori error estimates for the lowest-order finite element methods, and we also present a posteriori error estimates for the parametric and trace methods.  Our presentation focuses especially on the relationship between the regularity of the surface $\gamma$, which is never assumed better than of class $C^2$, the manner in which $\gamma$ is represented in theory and practice, and the properties of the resulting methods.  
\end{abstract}

\keywords{
Surface partial differential equations, Laplace-Beltrami operator, surface finite element methods, parametric finite element methods, trace finite element methods, narrow band finite element methods, a-priori and a-posteriori error estimates.
}

\date{\today}

\maketitle


\newpage

\begingroup
\parindent=0em
\etocsettocstyle{\rule{\linewidth}{\tocrulewidth}\vskip0.5\baselineskip}{\rule{\linewidth}{\tocrulewidth}}
\localtableofcontents 
\endgroup

%

\section{Introduction}\label{sec:introduction}

Partial differential equations (PDEs) posed on surfaces play an important role in many domains of pure and applied mathematics, including geometry, modeling of materials, fluid flow, and image and shape processing.  The numerical approximation of such surface PDEs is both practically important and the source of many mathematically rich problems.  

We consider a closed, compact and orientable surface $\gamma$ in $\mathbb{R}^{n+1}$ of co-dimension $1$.  The Laplace-Beltrami operator $-\Delta_\gamma$, which acts as a generalization of the standard Euclidean Laplace operator, plays a central role in both static and time-dependent surface PDE models arising in a wide range of applications.  Because of this a wide variety of numerical methods have been developed for the Laplace-Beltrami equation %
\[
-\Delta_\gamma \widetilde u=\widetilde f,
\]
where $\widetilde f$ is a given forcing function satisfying $\int_\gamma \widetilde f =0$.  In this article we first lay out some important notions from differential geometry.  We then describe three important classes of finite element methods (FEMs) for the Laplace-Beltrami problem:  the parametric method, the trace method, and the narrow band method.   In all three cases we focus on the simplest case of piecewise linear finite element spaces and give an in-depth discussion of the effects of geometry on error behavior.  


The {\it parametric finite element method} was introduced in 1988 by Dziuk \cite{Dz88}, with some important related techniques appearing in earlier works on boundary element methods \cite{Ne76, Ben84}.  This method is the simplest of the many FEM that have been developed for solving the Laplace-Beltrami problem.  The given PDE is first written in weak form as:  Find $\widetilde u \in H^1(\gamma)$ such that $\int_\gamma \widetilde u =0$ and  
\[
a(\widetilde u,\widetilde \tv):=\int_\gamma \nabla_\gamma \widetilde u \cdot \nabla_\gamma \widetilde \tv = \int_\gamma \widetilde f \widetilde \tv \qquad \forall \widetilde \tv \in H^1(\gamma).  
\]
Here $H^1(\gamma)$ is the set of functions $\widetilde \tv$ in $L_2(\Omega)$ whose tangential gradient $\nabla_\gamma \widetilde \tv \in [L_2(\gamma)]^{n+1}$.  The continuous surface $\gamma$ is approximated by a polyhedral surface $\Gamma$ whose faces serve as a finite element mesh, and the finite element space $\V$ is made of continuous piecewise linear functions over $\Gamma$.  The finite element method then consists of finding $U \in \V$ such that
\[ 
A(U,V)=\int_\Gamma \nabla_\Gamma U \cdot \nabla_\Gamma V=\int_\gamma F V \qquad \forall V \in \V,
\]
where $F$ is a suitable approximation (lift) of $f$ defined on $\Gamma$.  In its conception and implementation, the resulting method is very similar to canonical FEM for solving Poisson's problem on Euclidean domains.    To quote Dziuk, ``...the numerical scheme is just the same as in a plane-two dimensional problem.  The only difference is that in our case the computer has to memorize three-dimensional nodes instead of two-dimensional ones.'' \cite[p. 143]{Dz88}.    The strategy underlying parametric surface finite element methods --direct translation of FEM on Euclidean spaces to triangulated surfaces-- has subsequently been applied to a variety of methods.  These include higher-order standard Lagrange methods \cite{De09}, various types of discontinuous Galerkin methods \cite{ADMSSV15, DMS13, CD16}, and mixed methods in  classical, hybridized, and finite element exterior calculus formulations \cite{Ben84, HoSt12, CD16, FFF16}.   A posteriori error estimation and adaptivity have been studied in \cite{DemlowDziuk:07, WCH10, BCMN:Magenes, DM16, BCMMN16, BD:18}.  Finally, we refer to the survey article \cite{DE13}. 

In many applications in which surface PDEs are to be solved, a background volume (bulk) mesh is already present. A paradigm example is two-phase fluid flow, in which effects on the interface between the two phases such as surface tension are coupled with standard equations of fluid dynamics on the bulk. In these cases it is advantageous to utilize the background volume mesh to solve surface PDEs instead of independently meshing $\gamma$. This is especially the case when $\gamma$ is evolving, since the meshes needed for the parametric method typically distort as $\gamma$ changes and periodic remeshing is thus necessary.  The trace and narrow band methods both employ background bulk meshes in order to solve surface PDEs.  

Trace (or cut) FEMs for the Laplace-Beltrami problem were first introduced in \cite{ORG09}.  In this method an approximating surface is constructed as in the parametric method, but using a different approach.  An {\it implicit} representation of $\gamma$ as the level set of some function $\phi$ is used, that is, it is assumed that
\[
\gamma=\big\{x \in \mathbb{R}^{n+1}~ : ~\phi(x)=0\big\}.
\]
A discrete surface $\Gamma$ is then defined as the zero level set of an interpolant of $\phi$ on the background mesh, and the finite element space is taken to be the trace of the bulk finite element space on $\Gamma$.  The FEM is posed and solved on $\Gamma$ as in the parametric method.   Note that the finite element space in the trace method consists of continuous piecewise linear functions over the faces of $\Gamma$.  However, because the faces of $\Gamma$ are arbitrary intersections of $n$-dimensional hyperplanes with $n+1$-simplices, they are not shape regular, and in particular may either fail to satisfy a minimum angle condition or be much smaller than the bulk simplices from which they are derived.  Counter to natural intuition about the quality of a finite element method posed on such a mesh, the trace method satisfies optimal error bounds and works well in practice.  In addition to the basic analysis of piecewise linear methods that we present below, the literature on trace methods for the Laplace-Beltrami problem includes study of matrix properties \cite{OR10}, adaptive versions \cite{DO12, CO15}, and extensions to higher-order \cite{Re15, GR16, GLR18}, stabilized \cite{BHL15, BHLMZ16}, and discontinuous Galerkin \cite{BHLM17} methods.  We refer to the recent survey article \cite{OR17}.

Narrow band methods also employ a bulk mesh in order to approximate surface PDEs, but extend a surface PDE to the bulk instead of restricting a bulk finite element space to a surface. This idea appeared first in \cite{MR1868103} and is based on an extension of the PDE into a tubular neighborhood $\mathcal N(\delta)$ of width $2 \delta$ about $\gamma$ that reads
\[
L(u_\delta)=-\textrm{div} \big( (I-\nabla d \otimes \nabla d) \nabla u_\delta \big)+u_\delta
= f_\delta.
\]
Here $f_\delta$ is an extension of $\wf$ from $\gamma$ to $\mathcal{N}(\delta)$ and $d$ is the distance function $\gamma$. The latter is chosen for simplicity over a generic level set function $\phi$ to represent $\gamma$ throughout this article. Because $\nabla d$ is the unit outward normal to $\gamma$, the coefficient matrix $I-\nabla d \otimes \nabla d$ is degenerate in the direction normal to $\gamma$, and the operator $L$ is thus elliptic but degenerate. We emphasize that in contrast to most previous literature on narrow band FEM we do not include a zero order term in our presentation, thereby adding extra difficulty due to the need to account for the non-trivial kernel of $L$ on closed surfaces. In narrow band FEMs, the Galerkin approximations to $u_\delta$ are posed over a discrete approximation $\mathcal{N}_h(\delta)$ to the narrow band $\mathcal{N}(\delta)$. Related methods that involve extending surface PDEs to bulk domains include the closest point method \cite{RM08}. 

Narrow-band unfitted finite element methods have been proposed and analyzed by different authors.
In \cite{MR2485787}, the aforementioned degenerate extension is shown to be well posed and error  bounds in the weighted bulk energy norm are derived. Subsequently, error estimates in the $H^1(\gamma)$ norm are obtained in \cite{MR2608464} for the lower order method.
An alternate nondegenerate extension $L(u_\delta) = -\Delta  u_\delta+u_\delta$ is then proposed in \cite{MR3249369} leading to optimal  $H^1(\gamma)$ and also $L^2(\gamma)$ error estimates for the lower order method when $f_\delta$ is  (or is close to) the constant normal extension of $\widetilde f$.
Independently, higher order methods are proposed and analyzed in \cite{MR3471100} using the extension
\[
L(u_\delta) = -\textrm{div} \big( \mu (I-dD^2 d)^{-2} \nabla u_\delta \big) + u_\delta,
\]
with $\mu := \det \big(I-dD^2 d\big)$ and $f_\delta$ the constant normal extension of $\widetilde f$.
Note also that the associated FEM requires a sufficiently accurate approximation of $D^2d$ (if not known explicitly).
For the case of lowest order (piecewise linear) finite element spaces, it is enough to approximate $D^2d$ with zero and thereby retrieve the discrete formulation in \cite{MR3249369}.


In the construction of all three FEMs above, we incur on variational crimes (consistency errors) due to the approximation of geometry.  In the parametric and trace methods, these errors arise because the finite element method is posed over a discrete approximation $\Gamma$ to $\gamma$, thereby leading to different bilinear forms ($a$ and $A$) used to compute the continuous and finite element solutions ($\wu$ and $U$).  In the narrow band method the finite element equations are posed over a discrete narrow band $\mathcal{N}_h(\delta)$ instead of over the domain $\mathcal{N}(\delta)$ on which the extended solution $u_\delta$ is defined.  This again entails the use of different bilinear forms in the definitions of the continuous and discrete solutions.  A core problem in surface FEMs is understanding and controlling these errors, which are typically called {\it geometric consistency errors} or {\it geometric errors}.   In order to analyze these errors, it is necessary to define a map $\bP:\Gamma \rightarrow \gamma$ and then compare $a(\widetilde \tv, \widetilde w)$ with $A(\widetilde \tv \circ \bP, \widetilde w \circ \bP)$ for given functions $\widetilde \tv, \widetilde w \in H^1(\gamma)$.  This is done via a change of variables argument for the map $\bP$.  There may be several competing demands of both theoretical and practical nature that come into play when choosing the map $\bP$, and a main focus of this article is to elucidate how this choice affects analysis and implementation of surface FEMs.  

The canonical choice of the map $\bP$ is defined via the so-called {\it signed distance function} $d:\mathcal N \rightarrow \gamma$.  The distance function is defined on a tubular neighborhood
$\mathcal N$ of $\gamma$ and is of the same regularity class as $\gamma$ provided that $\gamma$ is at least $C^2$ and $\mathcal N$ is sufficiently narrow. In such a case, the map (also called distance-lift or orthogonal closest point projection) 
\begin{equation*}
\bP_{d}(\bx) := \bx - d(\bx) \nabla d(\bx)  \quad
\forall~\bx \in \mathcal N
\end{equation*}
is well defined and is of class $C^1$. The maps $d$ and $\bP_{d}$
play a crucial role in analyzing and in some cases defining the numerical
algorithms presented below.  In particular, the distance function is a critical tool in proving error estimates that are of optimal order with respect to geometric consistency errors.  When a generic map $\bP:\Gamma \rightarrow \gamma$ is instead used to analyze surface FEMs, the predicted behavior of geometric errors is of one order less than is seen in practice and also than may be proved using the closest point projection.  More precisely, when quasi-uniform meshes of size $h$ are used with affine surface approximations in the parametric and trace methods, arguments which use special properties of the closest point projection predict an $O(h^2)$ geometric errors, and these are in fact observed in practice.  On the other hand, standard proofs employing a generic map $\bP$ instead of the distance function map $\bP_d$ predict only order $h$ geometric errors.  This increase in convergence order due to the properties of the closest point projection may be viewed as a {\it superconvergence} effect.  

Reliance on $\bP_{d}$ may however also constitute a serious drawback for several reasons. First, $\bP_d$ has a closed form expression only for the sphere and torus, so it is in general not directly available to the user.  We thus discuss how to use the distance function only as a theoretical tool for the parametric FEM and yet retain the superconvergence properties of $\bP_{d}$.  On a practical level, the user is still free to choose from a much more general class of lifts to implement an algorithm.  Our presentation includes optimal a priori and a posteriori estimates in $H^1$ and optimal a priori estimates in $L_2$ for an algorithm whose implementation only requires access to a generic lift $\bP$; the latter appear to be new in the literature even for smooth surfaces.
Second, if $\gamma$ is merely $C^{1,\alpha}$ for $\alpha<1$, then the closest point projection $\bP_d$ is not uniquely defined in any neighborhood of $\gamma$.  We thus also provide an analysis of parametric FEMs for $\gamma$ of class $C^{1, \alpha}$ that instead makes use of a generic parametric map.  The price we pay is a possible order reduction of the method due to the loss of superconvergence properties of $\bP_{d}$.  
Finally, previous proofs of optimal-order error estimates employing $\bP_{d}$ have required that $\bP_{d}$ is of class $C^2$ and thus $\gamma$ of class $C^3$; cf. \cite{Dz88}.  However, the solution $u$ to the Laplace-Beltrami problem already possesses the $H^2$ regularity needed to ensure optimal convergence of piecewise linear finite element methods when $\gamma$ is of class $C^2$.  In this survey we bridge this gap by giving a novel error analysis for the three FEMs which is based exclusively on $C^2$ regularity of $d$ and $\gamma$, but which also preserves the superconvergence property in the geometric error.   In the case of the trace and narrow band methods we achieve this by a regularization argument.

This article is organized as follows. In section \ref{sec:preliminaries} we
introduce surface gradient, divergence and
Laplace-Beltrami operators along with the signed distance function and its most
relevant properties. In section \ref{S:perturbation} we quantify the geometric effects
of perturbing surfaces $\gamma$ of class $C^{1,\alpha}$ and $C^2$. We also present
$H^2$ extensions to a tubular neigborhood
$\Nd\subset\mathcal{N}$ of width $\delta$
\begin{equation*}
\|u\|_{H^2(\Nd)} \lesssim \delta^{\frac12} \|\wu\|_{H^2(\gamma)}
\end{equation*}
of functions $\wu\in H^2(\gamma)$ provided $\gamma$ is of class $C^2$.
This turns out to be essential for our later
error analysis of the trace and narrow band methods for $C^2$ surfaces.
In section \ref{sec:parametric} we give a
selfcontained exposition of parametric FEMs for surfaces of class $C^{1,\alpha}$ and
$C^2$, including a priori and a posteriori error analyses. In section \ref{sec:trace}
we describe the trace method and conclude in section \ref{sec:narrow} with the
narrow band method. Both discussions assume $C^2$ regularity of $\gamma$.


\section{Calculus on Surfaces}\label{sec:preliminaries}

In this section we discuss basic concepts of differential geometry.
We start in section \ref{S:repres-surface} by describing
the paramatric representaton of $\gamma$ via charts.
This classical point of view
is critical to introduce the first fundamental form $\bg$, the area element $q$, and
the unit normal $\bnu$ of $\gamma$. We present in section \ref{S:operators} the
tangential operators (gradient $\nabla_\gamma$, divergence div$_\gamma$, and
Laplace-Beltrami $\Delta_\gamma$) as well as
the Weingarten map; we also discuss $H^2$-regularity for $\Delta_\gamma$
on surfaces $\gamma$ of class $C^2$.
We introduce the distance function $d$ in section
\ref{S:distance-function} and derive several important properties of it;
this intrinsic approach avoids
parametrizations and allows for implicit representions of $\gamma$. We devote section
\ref{S:curvatures} to
the second fundamental form of $\gamma$ and its principal curvatures using
both parametric and intrinsic approaches.

\subsection{Parametric Surfaces}\label{S:repres-surface}
%
We assume that $\gamma$ is a closed, compact, orientable manifold of class $C^{1,\alpha}$, $0<\alpha\leq1$, 
and co-dimension $1$ in $\mathbb{R}^{n+1}$. It can be represented parametrically
by an atlas $\lbrace ({\mathcal V}_i,{\mathcal U}_i,\bchi_i)\rbrace_{i\in I}$,
where the individual
charts $\bchi_i:{\mathcal V}_i \to {\mathcal U}_i\cap\gamma\subset\mathbb{R}^{n+1}$
are isomorphisms of class $C^{1,\alpha}$ compatible with the orientation of $\gamma$; the open connected sets
${\mathcal V}_i \subset \mathbb{R}^n$ are the parametric domains.
Unless stated otherwise, it will be often sufficient to consider a single chart and resort to a partition of the unity.
We thus drop the index $i$ for convenience.
For $\bx \in {\mathcal U} \cap \gamma$, we set $\by := \bchi^{-1}(\bx) \in {\mathcal V}$.

Let $\partial_j\bchi(\by)$ be the column vector of $j$-th partial derivatives of
$\bchi(\by)$ for $1\le j\le n$ at $\by\in {\mathcal V}$.
By definition, the rank of $D \bchi(\by) = \big(\partial_j\bchi(\by)\big)_{j=1}^n
\in  {\mathbb R}^{(n+1)\times n}$ is  $n$ (full rank).
This implies that $\{\partial_j\bchi(\by)\}_{j=1}^n$ are linearly independent and span
the tangent hyperplane to $\gamma$ at $\bx$.

The \emph{first fundamental form} is the symmetric and positive definite matrix
$\bg\in {\mathbb R}^{n\times n}$ defined by
\begin{equation}\label{dg:d:first}
\bg(\by) := D \bchi(\by)^t D \bchi(\by) \quad\forall \by\in\mathcal V.
\end{equation}
If $\bg = (g_{ij})_{i,j=1}^n$, then the components $g_{ij}$ read
$$
g_{ij} = \partial_i\bchi^t\partial_j\bchi = \partial_i\bchi\cdot\partial_j\bchi,
$$
which depends on the choice of parametrization.
A normal vector $\bN(\by)$ to $\gamma$ at $\bx$ can be written as
$\bN(\by) = \sum_{j=1}^{n+1} A_j(\by) \be_j$, where
$A_j:= \textrm{det}(\be_j, D \bchi)$ and
$\{ \be_j \}_{j=1}^{n+1}$ is the canonical basis of $\mathbb R^{n+1}$. In fact,
since
\[
\bN \cdot \partial_i\bchi = \sum_{j=1}^{n+1} \be_j\cdot\partial_i\bchi
\det(\be_j,D\bchi) = \det \big( \sum_{j=1}^{n+1} (\be_j\cdot\partial_i\bchi)
\be_j,D\bchi\big) = \det(\partial_i\bchi,D\bchi) = 0,
\]
and $A_j\ne0$ for at least one $j$ because $D\bchi$ has rank $n$, we deduce that
\begin{equation}\label{unit-normal}
\bnu (\by) := \frac{\bN(\by)}{|\bN(\by)|} \quad\forall \by\in\mathcal V
\end{equation}
is a well-defined unit normal vector to $\gamma$. Therefore, the matrix
\[
\bT (\by) := [D\bchi(\by),\bnu(\by)] \in \mathbb{R}^{(n+1)\times(n+1)}
\quad\forall \by\in\mathcal V
\]
has rank $n+1$ and so is invertible. We write its inverse as
\[
\bT^{-1} = \begin{bmatrix} \bB \\ \bv^t \end{bmatrix},
\quad
\bB \in \mathbb{R}^{n\times(n+1)}, \bv\in\mathbb{R}^n,
\]
and note that
\[
\bI_{(n+1)\times(n+1)} = \bT^{-1}\bT =
\begin{bmatrix}
\bB \, D\bchi & \bB \bv \\ \bv^t\ D\bchi & \bv^t\bnu,
\end{bmatrix}  
\]
whence
\[
\bB \, D\bchi = \bI_{n\times n},
\quad
\bv^t D\bchi = 0,
\quad
\bv^t \, \bnu = 1.
\]
The last two equalities imply $\bv=\bnu$. Reversing the order of multiplication yields
\[
\bI_{(n+1)\times(n+1)} = \bT \, \bT^{-1} = D\bchi \, \bB + \bnu \bnu^t,
\]
whence the {\it projection matrix} $\Pi \in \mathbb{R}^{(n+1)\times(n+1)}$ on the tangent
hyperplane to $\gamma$ has the form
\begin{equation}\label{projection}
\Pi := \bI - \bnu\otimes\bnu = D\bchi \, \bB .
\end{equation}
To obtain an explicit expression for $\bB$ note that
\[
D\bchi = (\bI - \bnu\otimes\bnu)^t D\bchi = \bB^t D\bchi^t D\bchi = \bB^t \bg
\quad\Rightarrow\quad
\bB = \bg^{-1} D\bchi^t.
\]
This leads to the following useful expression of $\Pi$ defined in \eqref{projection}:
\begin{equation}\label{projection-b}
\Pi = D\bchi \, \bg^{-1} D\bchi^t.
\end{equation}

The {\it area element} $q(\by)$ is the ratio of the
infinitesimal volume at $\by\in\mathcal{V}$
and area of $\gamma$ at $\bx=\bchi(\by)$, namely the volume of the parallelotope
in the tangent plane to $\gamma$ spanned by the vectors $\{\bchi_j\}_{j=1}^n$:
\begin{equation}\label{e:area-rep}
q(\by) := \det \big([\bnu(\by), D\chi(\by)] \big)
\quad\forall \by\in\mathcal{V}.
\end{equation}
To obtain a more familiar form of $q$ we argue as follows:
\begin{equation}\label{area-prelim}
  q = \frac{1}{|N|} \det \big([\bN, D\bchi] \big) =
  \frac{1}{|N|} \det\big( [\bN,D\bchi]^t [\bN,D\bchi]  \big)^{\frac12} = \sqrt{\det \bg},
\end{equation}
because $\det\big( [\bN,D\bchi]^t [\bN,D\bchi]  \big) =
|\bN|^2 \det(D\bchi^t D\bchi) = |\bN|^2 \det \bg$. Moreover, exploiting that
$|N|^2 = \sum_{j=1}^{n+1} A_j \det([\be_j,D\bchi]) = \det([\bN,D\bchi])$, we deduce
\begin{equation}\label{q-N}
q = |\bN|.
\end{equation}  

An integrable function $\tv:\mathcal{V}\to\mathbb{R}$ induces an integrable function
$\widetilde{\tv}:\gamma\to\mathbb{R}$ by composition $\tv = \widetilde{\tv}\circ\bchi$,
or equivalently $\widetilde{\tv}(\bx) = \tv(\by)$ for all $\by\in\mathcal{V}$.
The area element allows for integration over $\gamma$ via the formula
\begin{equation}\label{int-gamma}
  \int_\gamma \widetilde \tv = \int_{\mathcal V} \tv q
  \quad\forall \tv \in L_1(\mathcal V).
\end{equation}
This definition does not depend on the parametrization: if $\bchi_1,\bchi_2$
are parametrizations of $\gamma$, then
$\bchi_1=\bchi_2 \circ (\bchi_2^{-1}\circ\bchi_1)$ and
$D\bchi_1=D\bchi_2 D(\bchi_2^{-1}\circ\bchi_1)$ whence
\[
q_1 = \big|\det\big(D(\bchi_2^{-1}\circ\bchi_1)\big)\big| q_2
\quad\Rightarrow\quad
\int_{\mathcal{V}_1} \tv q_1 = \int_{\mathcal{V}_2} \tv q_2. 
\]

\subsection{Differential Operators}\label{S:operators}
%
If a function $\tv:\mathcal{V}\to\mathbb{R}$ is of class $C^1$, we can define the
{\it tangential (or surface) gradient} of the corresponding function
$\widetilde \tv: \gamma\to\mathbb{R}$ as
a vector tangent to $\gamma$ that satisfies the chain rule
\begin{equation}\label{def-tang-grad}
\nabla \tv(\by) = D\bchi(\by)^t \nabla_\gamma \widetilde \tv(\bx)
\quad\forall \by\in\mathcal V.
\end{equation}
Since $\nabla_\gamma\widetilde \tv$ is spanned by $\{\partial_j\bchi\}_{j=1}^n$, we get
$\nabla_\gamma\widetilde \tv = D\bchi w$ for some $w \in \mathbb{R}^n$ whence $w = \bg^{-1} \nabla \tv$ and
\begin{equation}\label{e:exact_grad}
  \nabla_\gamma \widetilde \tv (\bx) = D\bchi(\by) \bg(\by)^{-1} \nabla \tv(\by)
  \quad\forall \by\in \mathcal V.
\end{equation}
If $\widetilde\bv=(\widetilde\tv_i)_{i=1}^{n+1}:\gamma\to\mathbb{R}^{n+1}$
is a vector field of class $C^1$, we define its tangential differential
$D_\gamma\widetilde\bv \in \mathbb{R}^{(n+1)\times(n+1)}$ as a matrix whose $i$-th row is
$(\nabla_\gamma \widetilde\tv_i)^t$. If $\gamma$ is of class $C^2$, then the unit normal
vector $\bnu$ is of class $C^1$ and its differential
\begin{equation}\label{weingarten}
  \bW(\bx) = D_\gamma \bnu(\bx)
  \quad\forall \bx\in\gamma
\end{equation}
is called the {\it Weingarten} (or shape) map of $\gamma$. In addition, the
{\it tangential divergence} of $\widetilde\bv$ is the trace of $D_\gamma\widetilde\bv$
\begin{equation}\label{tang-diver}
  \textrm{div}_\gamma \widetilde\bv(\bx)
  = \textrm{trace}\big( D_\gamma\widetilde\bv(\bx)  \big)
  = \sum_{i,j=1}^n g^{ij}(\by) \, \partial_i\bchi(\by)\cdot\partial_j \bv(\by)
  \quad\forall \by\in\mathcal{V},
\end{equation}  
provided $\bg^{-1} = (g^{ij})_{i,j=1}^n$.
If both $\gamma$ and $\tv:\gamma\to\mathbb{R}$ are of class $C^2$, then
the {\it Laplace-Beltrami (or surface Laplace) operator} is now defined to be
\begin{equation}\label{lap-bel-def}
  \Delta_\gamma \widetilde \tv =
  \frac{1}{q(\by)} \div { q(\by) \bg(\by)^{-1}\nabla \tv(\by) }
  \quad\forall \by\in\mathcal V.
\end{equation}  
The following lemma shows that \eqref{lap-bel-def} is designed
to allow integration by parts on $\gamma$, exactly as it happens in flat
domains with the Laplace operator $\Delta$.

\begin{lemma}[weak form of the Laplace-Beltrami operator]
If $\widetilde\varphi$ is of class $C^1$ with compact support in $\gamma$, then
  \begin{equation}\label{lap-bel-weak}
    \int_\gamma \widetilde\varphi \, \Delta_\gamma \widetilde\tv =
    - \int_\gamma \nabla_\gamma \widetilde\varphi \cdot \nabla_\gamma \widetilde\tv.
  \end{equation}  
\end{lemma}
\begin{proof}
In view of \eqref{int-gamma}, which allows us to switch from $\gamma$ to $\mathcal{V}$
back and forth, we can write
\begin{align*}
  \int_\gamma \widetilde\varphi \Delta_\gamma \widetilde\tv
  &= \int_{\mathcal V} \varphi \, \div{q\bg^{-1}\nabla\tv}
  \\
  &= - \int_{\mathcal V} \nabla\varphi \cdot \bg^{-1} \nabla \tv \, q
  \\
  &= -\int_{\mathcal V} D\bchi \bg^{-1} \nabla\varphi \cdot D\bchi \bg^{-1} \nabla\tv \, q
  \\
  &= -\int_\gamma \nabla_\gamma\widetilde\varphi \cdot \nabla_\gamma \widetilde\tv.
\end{align*}
This proves \eqref{lap-bel-weak} as desired.
\end{proof}

In view of \eqref{lap-bel-weak}, we are now in the position to introduce the
weak formulation for the Laplace-Beltrami operator. We first define the space
of square integrable functions on $\gamma$ with vanishing mean value by
\[
 L_{2,\#}(\gamma) := \Big\{\widetilde \tv \in L_2(\gamma) \;\big|\; \int_\gamma \widetilde \tv = 0\Big\}
\]
and its subspace $H^1_\#(\gamma)$ containing square integrable weak derivatives defined as for example in Section 3 of \cite{JK95} by
\begin{equation*}
  H^1_\#(\gamma) := H^1(\gamma) \cap L_{2,\#}(\gamma),
  \quad
  H^1(\gamma) := \Big\{\widetilde \tv\in L_2(\gamma) \;\big|\; 
          \nabla(\widetilde \tv \circ  \bchi) \in [L_2(\mathcal V)]^{n}\Big\}.
\end{equation*}
Our next result shows that the natural norm $\| \nabla_\gamma \widetilde v \|_{L_2(\gamma)}+\|\widetilde v \|_{L_2(\gamma)}$ in $H^1_\#(\gamma)$ is equivalent to the semi-norm $\| \nabla_\gamma \widetilde v \|_{L_2(\gamma)}$. The proof essentially hinges on the Peetre-Tartar Lemma \cite{MR0221282,MR532371}, but we proceed with a slightly more direct proof as in \cite[Section 5.8.1]{Ev98}.


\begin{lemma}[Poincar\'e-Friedrichs inequality]\label{L:Poincare}
Let $\gamma$ be a compact Lipschitz surface. There exists a constant $C$ only depending on $\gamma$ such that
\begin{equation}\label{poincare}
\|\widetilde \tv\|_{L_2(\gamma)} \le C \|\nabla_\gamma \widetilde \tv\|_{L_2(\gamma)}
    \quad\forall \, \widetilde \tv\in H^1_\#(\gamma).
\end{equation}
\end{lemma}  
\begin{proof}
We prove by contradiction the more general estimate
\begin{equation}\label{more-general}
\|\wv\|_{L_2(\gamma)} \le C \Big( \|\nabla_\gamma \wv\|_{L_2(\gamma)}
+ \Big| \int_\gamma \wv \, \Big| \Big)
    \quad\forall \, \wv\in H^1(\gamma).
\end{equation}
Suppose that there is a sequence $\wv_k\in H^1(\gamma)$ such that
\[
\|\wv_k\|_{L_2(\gamma)}=1,\quad
\|\nabla_\gamma \wv_k\|_{L_2(\gamma)} + \Big| \int_\gamma \wv_k \, \Big| \to 0
\]
as $k\to\infty$. We deduce that $\{\wv_k\}_k$ is uniformly bounded in $H^1(\gamma)$.
Since the embedding $H^1(\gamma) \subset L_2(\gamma)$ is compact (because $H^1(\mathcal V) \subset L_2(\mathcal V)$ is compact, see the proof of~\cite[Theorem~2.34]{MR681859}), there is a Cauchy subsequence (with abuse of notation not relabeled) of $\{\wv_k\}_k$ in $L_2(\gamma)$. This, together with $\|\nabla_\gamma \wv_k\|_{L_2(\gamma)}\to0$, implies that $\{\wv_k\}_k$ is a Cauchy sequence in $H^1(\gamma)$. Let $\wv\in H^1(\gamma)$ be the limit of $\wv_k$ in $H^1(\gamma)$, which yields $\nabla_\gamma\wv = \lim_{k\to\infty}\nabla_\gamma\wv_k = 0$ whence $\wv$ is constant on $\gamma$. Moreover, $\int_\gamma\wv=\lim_{k\to\infty}\int_\gamma \wv_k = 0$ whence $\tv=0$. This gives rise to the contradiction $0=\|\wv\|_{L_2(\gamma)}=\lim_{k\to\infty}\|\wv_k\|_{L_2(\gamma)}=1$, and finishes the proof.
\end{proof}


We emphasize that the Poincar\'e-Friedrichs constant depends on the surface $\gamma$. 
Later we shall consider perturbations $\Gamma$ of $\gamma$ and derive Poincar\'e-Friedrichs type estimates on $\Gamma$ where the constant depends on $\gamma$ provided the geometry of $\gamma$ is minimally approximated by $\Gamma$. This is proved in Lemma~\ref{L:Poincare-unif} for Lipschitz surfaces and only requires that the $L_2$ and $H^1$ norms on $\gamma$ and $\Gamma$ are equivalent.

We will not deal explicitly with functionals in the dual space $H^{-1}_\#(\gamma)$
of $H^1_\#(\gamma)$, but occasionally need its norm for $\widetilde f\in L_{2,\#}(\gamma)$
\begin{equation}\label{dual-norm}
\|\widetilde f\|_{H^{-1}_\#(\gamma)} = \sup_{\widetilde \tv\in H^1_\#(\gamma)}
\frac{\int_\gamma \widetilde f \widetilde \tv}{\|\nabla_\gamma \widetilde \tv\|_{L_2(\gamma)}}.
\end{equation}
Lemma \ref{L:Poincare} (Poincar\'e-Friedrichs inequality) implies that $\|\widetilde f\|_{H^{-1}_\#(\gamma)} \le C \|\widetilde f\|_{L_{2,\#}(\gamma)}$.
The weak formulation of $-\Delta_\gamma \widetilde u = \widetilde f$ reads:
for $\widetilde f \in L_{2,\#}(\gamma)$, seek $\widetilde u \in H^1_\#(\gamma)$ so that
\begin{equation}\label{e:weak}
 \int_{\gamma} \nabla_{\gamma}\widetilde u \cdot \nabla_{\gamma}\widetilde \tv 
=  \int_\gamma \widetilde f \,  \widetilde\tv 
\quad \forall \,  \widetilde\tv \in H^1_\#(\gamma).  
\end{equation}
Since the Dirichlet bilinear form in \eqref{e:weak} is coercive, according
to Lemma \ref{L:Poincare}, 
existence and uniqueness of a solution $\widetilde u\in H^1_\#(\gamma)$ is a consequence of
the Lax-Milgram theorem.
 We observe that thanks to the property $\widetilde f \in L_{2,\#}(\gamma)$, the solution $\widetilde u \in H^1_\#(\gamma)$ satisfies
\begin{equation}\label{e:weak_relax}
 \int_{\gamma} \nabla_{\gamma} \widetilde u \cdot \nabla_{\gamma} \widetilde \tv 
=  \int_\gamma \widetilde f \, \widetilde \tv 
\quad \forall \, \widetilde \tv \in H^1(\gamma).
\end{equation}
It turns out that $\widetilde u$ exhibits the usual regularity pick-up provided $\gamma$ is
of class $C^2$.
\begin{lemma}[regularity]\label{L:regularity}
  If $\gamma$ is of class $C^2$, then there is a constant $C$ only depending on
  $\gamma$ such that
\begin{equation}\label{regularity}
\|\widetilde u\|_{H^2(\gamma)} \le C \|\widetilde f\|_{L_2(\gamma)}.  
\end{equation}
\end{lemma}
\begin{proof}
We use a localization argument to the parametric domain. We assume, without loss
of generality, that the atlas $\{(\mathcal{V}_i,\mathcal{U}_i,\bchi_i)\}_{i=1}^I$
satisfies the
following property: there exist domains $\mathcal{W}_i$ such that
$\overline{\mathcal{W}}_i \subset \mathcal{U}_i$ and $\{\mathcal{W}_i\}_{i=1}^I$ is
still a covering of $\gamma$.
Let now $\{\widetilde \psi_i\}_{i=1}^I$ be a $C^2$ partition of unity associated with the
covering $\{\mathcal{W}_i\}_{i=1}^I$. The functions $u_i=u\psi_i$ satisfy 
\[
\Delta_\gamma \widetilde u_i = \widetilde \psi_i \widetilde f + 2\nabla_\gamma \widetilde u \cdot \nabla_\gamma \widetilde \psi_i + \wu \Delta_\gamma \widetilde \psi_i=:  \widetilde g_i.
\]
In light of the estimate $\| \nabla_\gamma \widetilde u\|_{L_2(\gamma)} \leq \| \widetilde f \|_{H^{-1}_\#(\gamma)}$ and \eqref{poincare} we deduce that $\|\widetilde g_i\|_{L_2(\gamma)} \le C \|\widetilde f\|_{L^2(\gamma\cap\mathcal{U}_i)}$. 
Recalling \eqref{lap-bel-def} we can rewrite $\Delta_\gamma u_i$ in the parametric domain $\mathcal{V}_i$ as
\[
\div{q_i(\by) \bg_i(\by)^{-1} \nabla u(\by)} = q_i(y) \widetilde g_i(\bchi(\by))
\quad\forall \, \by\in\mathcal{V}_i,
\]
and observe that this is a uniformly elliptic problem with $C^1$ coefficients.
Applying interior regularity theory \cite{Ev98}, we deduce
\[
\|u_i\|_{H^2(\bchi^{-1}(\mathcal{W}_i))} \le C \|g_i\|_{L_2(\mathcal{U}_i)}.
\]
Therefore, adding over $i$ and using the finite overlap property of the
sets $\mathcal{U}_i$, we end up with
\[
\|\widetilde u\|_{H^2(\gamma)} \le \sum_{i=1}^I \| \widetilde u_i\|_{H^2(\mathcal{W}_i)}
\le C \sum_{i=1}^I \| \widetilde g_i\|_{L_2(\mathcal{U}_i)} \le
 C \|\widetilde f\|_{L_2(\gamma)} ,
\]
as asserted.
\end{proof}  

In view of our discussion below of surfaces of class $C^{1,\alpha}$ with $0<\alpha\le1$, it is natural to ask whether the regularity estimate \eqref{regularity} is still valid in
this more general context. We now show that this is indeed the case provided
the surface $\gamma$ is of class $W^2_p$ with $p>n$, or equivalently the
parametrizations $\{\bchi_i\}_{i=1}^I$ and partitions of unity $\{\widetilde \psi_i\}_{i=1}^I$ subordinate to the covering $\{\mathcal{W}_i\}_{i=1}^I$ of $\gamma$
are of class $W^2_p$. In this case a Sobolev embedding implies $\gamma$ is of class
$C^{1,\alpha}$ with $0< \alpha=1-\frac{n}{p} \le 1$.

\begin{lemma}[regularity for $W^2_p$ surfaces]\label{L:regularity-W2p}
If $\gamma$ is of class $W^2_p$ with $n<p\le\infty$, then there is a constant
$C>0$ depending on $\gamma, p$ and $n$ such that
\begin{equation}\label{regularity-W2p}
\|\widetilde u\|_{H^2(\gamma)} \le C \|\widetilde f\|_{L_2(\gamma)}.  
\end{equation}
\end{lemma}    
\begin{proof}
We argue with one chart $(\mathcal{V},\mathcal{U},\bchi)$ and thus suppress the index $i$ in $g$, $\bchi$, etc. Since $\wf\in L_2(\gamma)$ and
$\wu\in H^1(\gamma)$, the right-hand side $g=\widetilde g\circ\bchi$ 
in the proof of Lemma \ref{L:regularity} (regularity) satisfies
\[
g \in L_{r_0}(\mathcal{V}) \qquad
\frac{1}{r_0} = \frac{1}{2} + \frac{1}{p}.
\]
On the other hand, the definitions \eqref{dg:d:first} and \eqref{e:area-rep} of
the first fundamental form $\bg$ and area element $q$ imply that they are bounded
in $L_\infty(\mathcal{V})$ as well as
\[
\bg, q \in W^1_p(\mathcal{V})
\qquad\Rightarrow\qquad
\bA := q \bg^{-1} \in W^1_p(\mathcal{V}).
\]
Therefore, the Laplace-Beltrami equation in the parametric domain $\mathcal{V}$
can be written in non-divergence form as follows:
\begin{equation}\label{non-div}
\bA : D^2 u = q g - \div \bA \cdot \nabla u = \ell \in L_{r_0}(\mathcal{V}).
\end{equation}
Since $\bA$ is uniformly continuous, the Calder\'on-Zygmund regularity theory
applies (cf. \cite[Theorem 9.15 and Lemma 9.17]{GT98}) and gives the interior regularity $u\in W^2_{r_0}(\mathcal{Z})$ with
\[
\|u\|_{W^2_{r_0}(\mathcal{Z})} \lesssim \|\ell\|_{L_{r_0}(\mathcal{V})}
\]
where $\mathcal{Z}:=\chi^{-1}(\mathcal{W})$ and $\overline{\mathcal{W}}\subset \mathcal U$ as in the proof Lemma~\ref{L:regularity} (regularity). Invoking Sobolev embedding again,
we deduce
\[
u \in W^1_{t_1}(\mathcal{Z}), \qquad
\frac{1}{t_1} = \frac{1}{r_0} - \frac{1}{n},
\]
and $\wu \in W^1_{t_1}(\gamma)$ upon pasting these local estimates together over $\gamma$;
hence $u\in W^1_{t_1}(\mathcal{V})$.
We now iterate this argument and prove a recurrence relation by induction. Suppose that
a sequence of real numbers $\{r_k,t_k\}$ is governed by the relations $t_0=2$ and
\[
\frac{1}{r_k} = \frac{1}{p} + \frac{1}{t_k},
\qquad
\frac{1}{t_{k+1}} = \frac{1}{r_k} - \frac{1}{n},
\]
and the right hand side of \eqref{non-div} satisfies $\ell \in L_{r_k}(\mathcal{V})$;
note that this is the case for $k=0$.
Calder\'on-Zygmund theory  thus implies $u \in W^2_{r_k}(\mathcal{Z})$ with
\[
\|u\|_{W^2_{r_k}(\mathcal{Z})} \lesssim \|\ell\|_{L_{r_k}(\mathcal{Z})}.
\]
Sobolev embedding in turn yields $u \in W^1_{t_{k+1}}(\mathcal{V})$ whence
$\ell \in L_{r_{k+1}}(\mathcal{V})$, which proves the recurrence relation.
Iterating these relations we see that for $k\ge0$
\[
\frac{1}{r_k} = \frac{1}{r_{k-1}} + \frac{1}{p} - \frac{1}{n}
= \frac{1}{2} + \frac{1}{p} + k \Big( \frac{1}{p} - \frac{1}{n} \Big),
\]
and that every step increases the value of $r_k$, because $\frac{1}{p}-\frac{1}{n}<0$.
Since $\wf\in L_2(\gamma)$, the iteration stops once $r_k\ge2$ or equivalently
\[
k = \Big\lceil \frac{n}{p-n} \Big\rceil.
\]
This concludes the proof.
\end{proof}
  
\subsection{Signed Distance Function}\label{S:distance-function}
%
We now take advantage of the ambient space $\mathbb R^{n+1}$ and use
standard calculus in a suitable tubular neighborhood $\mathcal N$ of $\gamma$
to derive useful expressions of geometric quantities; we postpone
momentarily the precise definition of $\mathcal N$.  The surface $\gamma$ splits
$\mathbb{R}^{n+1}$ into two disjoint sets, the interior and exterior of $\gamma$.
The \emph{signed distance function} $d:\mathcal N \rightarrow \gamma$ is
defined for every $\bx\in\mathcal N$ to be the distance of $\bx$ to $\gamma$,
$\dist(\bx,\gamma)$, if
$\bx$ belongs to the exterior of $\gamma$ and $-\dist(\bx,\gamma)$ if $\bx$ belongs
to the interior of $\gamma$, whence
\[
|d(\bx)| = \dist(\bx,\gamma)
\quad\forall \, \bx\in\mathcal N.
\]
It turns out that $d$ belongs to the same regularity class as $\gamma$ so long as $\gamma$ is at least $C^2$, which we henceforth assume in our discussion of $d$.  While the distance function exists for surfaces of regularity less than $C^{1,1}$, as we explain in Section \ref{S:d_for_C1} below its properties are drastically different and it is not immediately useful for purposes of defining and analyzing surface FEM.  Returning to the setting of $C^2$ surfaces, $\nabla d(\bx)$ is well defined for all
$\bx\in\mathcal N$ and computed on $\gamma$ gives the unit normal $\bnu(\bx)$
pointing outwards:
\[
\bnu(\bx) = \nabla d(\bx)
\quad\forall \, \bx\in\gamma.
\]
Since $\nabla d$ is defined in $\mathcal N$ it provides a natural extension of
$\bnu$ to $\mathcal N$. This neighborhood $\mathcal N$ is sufficiently small
that for every $\bx\in\mathcal N$ there is a unique {\it closest point projection}
$\bP_d(\bx) \in\gamma$ defined by
\begin{equation}\label{e:lift_dist}
\bP_d(\bx) = \bx - d(\bx) \nabla d(\bx)  \qquad
\forall \, \bx \in \mathcal N.
\end{equation}
An important property is that $\nabla d$ coincides at $\bx\in\mathcal N$ and
$\bP_d(\bx)\in\gamma$:
\begin{equation}\label{property-d}
\nabla d(\bx) = \nabla d\big(\bP_d(\bx) \big)
= \nabla d \big(\bx - d(\bx) \nabla d(\bx) \big)
\quad\forall \, \bx\in\mathcal{N}.
\end{equation}
Since $|\nabla d(\bx)|^2 = 1$, we deduce that the Hessian $D^2 d(\bx)$ satisfies
\begin{equation}\label{zero-eig}
D^2 d(\bx) \, \nabla d(\bx) = 0
\quad\forall \, \bx \in \mathcal N.
\end{equation}
This implies that $D^2 d(\bx)$ can be regarded as an operator acting on the tangent
hyperplane to $\gamma$ at $\bx\in\gamma$ and thus gives
an alternative representation to the Weingarten map \eqref{weingarten}:
\[
\bW(\bx) = D^2 d(\bx)
\quad\forall \, \bx \in\gamma.
\]
This has two important consequences. First it provides a natural extension of $\bW$
to $\mathcal N$ and second shows that $\bW$ is symmetric, which is not apparent
from \eqref{weingarten}.

Given a generic function $\widetilde\tv:\gamma\to\mathbb{R}$,
the distance function $d$ provides a natural way to extend it
to the neighborhood $\mathcal{N}$ upon writing
\begin{equation}\label{normal-extension}
\tv(\bx) = \widetilde\tv\big(\bP_d(\bx)  \big)
= \widetilde\tv (\bx - d(\bx) \nabla d(\bx))
\quad\forall \, \bx\in\mathcal{N}.
\end{equation}
Differentiating and using the definition \eqref{projection} of orthogonal
projection, we obtain
\begin{equation}\label{e:rel_dist_grad}
\begin{aligned}
\nabla \tv(\bx) &= \big(\bI - \nabla d(\bx)\otimes\nabla d(\bx) - d(\bx) D^2 d(\bx) \big)
\nabla_\gamma \widetilde \tv\big(\bP_d(\bx)\big)
\\
& = \big(\Pi(\bx) - d(\bx) D^2 d(\bx)\big) \nabla_\gamma \widetilde \tv\big(\bP_d(\bx)\big)
\\
&= \big(\bI - d(\bx) D^2 d(\bx)\big) \Pi(\bx)\nabla_\gamma \widetilde \tv\big(\bP_d(\bx)\big)
\end{aligned}
\end{equation}
where the last equality hinges on \eqref{property-d}, which implies
$\Pi(\bx) = \Pi(\bP_d(\bx))$ and $D^2 d(\bx)=D^2 d(\bx) \Pi(\bx)$. In particular,
$\nabla \tv(\bx) = \nabla_\gamma \widetilde \tv\big(\bP_d(\bx)\big)$
for $\bx\in\gamma$ because
\eqref{normal-extension} provides a normal extension of $\widetilde \tv$.

Suppose now that $\tv$ is an extension of $\widetilde \tv$ to $\mathcal{N}$,
but not necessarily in the normal direction. An intrinsic definition of
tangential gradient of $\widetilde \tv$ is the orthogonal projection of
$\nabla \tv$ to the tangent hyperplane of $\gamma$:
\begin{equation}\label{grad-extension}
\nabla_\gamma \widetilde \tv(\bx) = \big(\bI - \bnu(\bx)\otimes\bnu(\bx) \big) \nabla\tv(\bx)
= \Pi(\bx) \nabla\tv(\bx)
\quad\forall \, \bx\in\gamma.
\end{equation}
This definition is consistent with \eqref{e:exact_grad}: $\nabla_\gamma \widetilde \tv(\bx)$
is orthogonal to $\bnu(\bx)$ and
\[
\nabla_\gamma \widetilde \tv(\bx) \cdot \partial_i\bchi(\by) =
\nabla \tv(\bx) \cdot \partial_i\bchi(\by) =
\partial_i \widetilde \tv(\bchi(\by))
\]
obeys the chain rule, whence it must coincide with our previous definition
based on these two properties. An important consequence of this property follows.
\begin{remark}[parametric independence]\label{R:param-indep}
  The definition \eqref{grad-extension} is independent of the extension:
  if $\tv_1,\tv_2$ are two extensions of $\widetilde \tv$ then $\tv_1-\tv_2 = 0$ on $\gamma$
  and the only non-vanishing component of $\nabla(\tv_1-\tv_2)$ is in the normal
  direction $\bnu$. Since definitions \eqref{grad-extension} and \eqref{e:exact_grad}
  agree, we deduce that the tangential gradient $\nabla_\gamma \widetilde \tv$ is independent
  of the parametrization $\bchi$ chosen to described $\gamma$. The same happens
  with the tangential divergence \eqref{tang-diver} as well as the
  Laplace-Beltrami operator \eqref{lap-bel-def}, the latter
  because of \eqref{lap-bel-weak}
  and the fact that \eqref{int-gamma} is independent of $\bchi$.
\end{remark}  

Given a vector field $\widetilde\bv:\gamma\to\mathbb{R}^{n+1}$ and corresponding
extension to $\mathcal{N}$, the tangential divergence can be written as
\[
\divg{\widetilde \bv(\bx)} = \trace{\nabla_\gamma \widetilde \bv(\bx)}
= \div{\bv(\bx)} - \bnu(\bx)^t \, \nabla \bv(\bx) \bnu(\bx)
\quad\forall \, \bx\in\gamma,
\]
and gives an alternative expression to \eqref{tang-diver}. Likewise, the
Laplace-Beltrami operator $\Delta_\gamma \widetilde\tv = \divg{\nabla_\gamma \widetilde \tv}$,
written parametrically in \eqref{lap-bel-def},
can be equivalently written in terms of the extension $\tv$ as follows
\[
\Delta_\gamma \widetilde \tv = \trace{(\bI-\bnu\otimes\bnu)D^2\tv} - (\nabla\tv\cdot\bnu)
\, \divg{\bnu},
\]
because $\nabla_\gamma(\nabla\bnu\cdot\bnu) \cdot \bnu = 0$. This implies the
expression
\[
\Delta_\gamma \widetilde \tv(\bx) = \Delta \tv(\bx) - \bnu(\bx)^t D^2\tv(\bx) \bnu(\bx)
- (\nabla\tv\cdot\bnu)(\bx) \, \divg{\bnu(\bx)}
\quad\forall \, \bx\in\gamma.
\]

\subsection{Curvatures}\label{S:curvatures}
%
We again assume that $\gamma$ is of class $C^2$.
In view of \eqref{weingarten}, the Weingarten map is symmetric and its
$n+1$ eigenvalues are real.
Except for the zero eigenvalue corresponding to the eigenvector $\bnu(\bx)$,
according to \eqref{zero-eig}, they are called the {\it principal curvatures}
of $\gamma$ at $\bx$ and are denoted by $\kappa_i(\bx)$ for $1\le i \le n$.
The eigenvectors of $\bW$ corresponding to the principal curvatures are called the {\it principal directions}.

We stress that $\kappa_i(\bx)$ is well defined for all $\bx \in \mathcal N$
because so is $\bW(\bx)$.
This allows us to make the definition of $\mathcal N$ precise. Given $\delta>0$, first let
\begin{equation}\label{e:delta-tube}
\mathcal N(\delta) := \{ \bx \in \mathbb R^{n+1} \ : \ |d(\bx)| < \delta \}.
\end{equation}
 Let also
\begin{equation}\label{K:def}
  K(\bx) := \max_{1 \le i \le n} |\kappa_i(\bx)| \quad\forall \, \bx \in \gamma;
  \qquad
  K_\infty := \|K\|_{L_\infty(\gamma)} . 
\end{equation}
We may now set
\begin{equation}\label{N:def}
  \mathcal{N} := \Big\{ \bx \in \mathbb{R}^{n+1}:
  \dist (\bx, \gamma) < \frac{1}{2 K_\infty} \Big\}=\mathcal{N}\left(K_\infty/2\right).
\end{equation}
Note that the distance function, closest point projection, and related properties are defined and hold on the larger set $\mathcal{N}(1/K_\infty)$.  We adopt the more limited definition of $\mathcal{N}$ in order to avoid degeneration of some quantities such as curvature of parallel surfaces (see below) that occurs near the boundary of the larger set.

Given $\varepsilon$ small so that $|\varepsilon| \le \frac{1}{2K_\infty}$,
we define the parallel surface $\gamma_\varepsilon$ to be
\[
\gamma_\varepsilon := \big\{\bx\in\mathcal N: d(\bx) = \varepsilon \big\}.
\]
The following statement relates the principal curvatures of $\gamma_\varepsilon$ with
those of $\gamma$.

\begin{lemma}[curvatures of parallel surface]\label{L:curv-parallel}
If $\gamma$ is of class $C^2$  so are all parallel surfaces $\gamma_\varepsilon$ and
their principal curvatures satisfy
\begin{equation}\label{kappas}
  \kappa_i(\bx) = \frac{\kappa_i(\bP_d(\bx))}{1+\varepsilon \, \kappa_i(\bP_d(\bx))}
  \quad\forall \, \bx\in\gamma_\varepsilon, 
\end{equation}
whereas the principal directions at $\bx$ and $\bP_d(\bx)$ coincide.
\end{lemma}
\begin{proof}  
Differentiate \eqref{property-d} to get
\[
D^2 d(\bx) = D^2d \big( \bP_d(\bx) \big) \big(\bI - \nabla d(\bx) \otimes \nabla d(\bx)
- d(\bx) D^2 d(\bx) \big),
\]
whence, since $\nabla d(\bx) = \nabla d \big(\bP_d(\bx) \big)$ again from
\eqref{property-d},
\[
\big( \bI + d(\bx) D^2 d\big(\bP_d(\bx) \big) \big) D^2 d(\bx)
= D^2 d\big(\bP_d(\bx) \big) \big( \bI - \nabla d(\bx) \otimes \nabla d(\bx) \big)
= D^2 d\big(\bP_d(\bx) \big).
\]
Therefore, for $\bx\in\gamma_\varepsilon$ we see that the eigenvalues of
$\big( \bI + \varepsilon D^2 d\big(\bP_d(\bx) \big) \big)$ are
\[
\kappa_i\big( \bI + \varepsilon D^2 d\big(\bP_d(\bx) \big) \big)=
1+\varepsilon \kappa_i \big(\bP_d(\bx) \big) \ge \frac12,
\]
according to \eqref{N:def}.
This implies that $\bI + \varepsilon D^2 d\big(\bP_d(\bx) \big)$ is nonsingular and
the previous relation reads as follows in terms of the Weingarten map:
\[
\bW(\bx) = \big( \bI + \varepsilon \bW\big(\bP_d(\bx) \big) \big)^{-1}
\bW \big(\bP_d(\bx) \big).
\]
This shows that the eigenvectors of $\bW(\bx)$ and $\bW \big(\bP_d(\bx) \big)$
coincide and the eigenvalues are related via \eqref{kappas}.
\end{proof}

The {\it second fundamental form}
$\bh = (h_{ij})_{i,j=1}^n$ of $\gamma$ is defined by
\[
h_{ij}(\by) := -\partial_i \bnu(\by)\cdot\partial_i\bchi(\by)
= \bnu(\by) \cdot \partial_{ij} \bchi(\by)
\quad\forall \, \by\in\mathcal{V},
\]
where the last equality relies on the fact that $\bnu$ and $\partial_j\bchi$
are orthogonal for $1\le j\le n$. The next result connects $\bh$ with the
Weingarten map \eqref{weingarten}.

\begin{lemma}[second fundamental form]\label{L:second-form}
The symmetric matrix $\bW=D_\gamma\bnu$ defines a selfadjoint operator on
the tangent hyperplane to $\gamma$ that can be represented in the basis
$\{\partial_j\bchi\}_{j=1}^n$ by the generally non-symmetric matrix
\[
\bs = -\bh \bg^{-1}.
\]
The eigenvalues of $\bs$ are the principal curvatures of $\gamma$.
\end{lemma}  
\begin{proof}
Since $D_\gamma\bnu \, \bnu = 0$, we can regard $D_\gamma\bnu$ as an operator acting on the
tangent plane to $\gamma$ and represent its image in terms
of $\{\partial_k\bchi\}_{k=1}^n$ as follows
\[
\partial_i\bnu (\by) =
D_\gamma\bnu(\bx) \, \partial_i\bchi(\by)
= \sum_{k=1}^n s_{ik}(\by) \, \partial_k\bchi(\by)
\quad\forall \, \by\in\mathcal{V}.
\]
Let $\bs(\by) := (s_{ij}(\by))_{ij, = 1}^n$ and multiply both sides by
$\partial_j\bchi(\by)$ to see that
\[
h_{ij}(\by) = - \partial_i\bnu(\by) \cdot \partial_j\bchi(\by)
= \sum_{k=1}^n s_{ik} \partial_k\bchi(\by)\cdot\partial_j\bchi(\by)
= \sum_{k=1}^n s_{ik} g_{kj}.
\]
This implies $\bh=-\bs\bg$ and thus the assertion.
\end{proof}

\subsection{Surface regularity and properties of the distance function} \label{S:d_for_C1}

In the previous two sections we have seen that when $\gamma$ is of class $C^2$, the closest point projection is uniquely defined in a tubular neighborhood of $\gamma$ whose width is related to the principal curvatures of the surface.  We shall see below that the closest point projection plays a pivotal role in analyzing finite element methods on $C^2$ surfaces.  On the other hand, some applications may require solving PDE on surfaces that are less regular than $C^2$.  Thus it is natural to ask which properties of the distance function and closest point projection carry over to less regular surfaces.  It turns out that the properties of these maps change drastically and fundamentally when crossing the threshold from $C^2$ to less regular ($C^{1,\alpha}$ with $\alpha<1$) surfaces.

In order to make this statement precise, we begin by restating for comparison from \cite[Lemma 14.16]{GT98} some fundamental properties of the distance function for $C^k$ surfaces ($k \ge 2$).
\begin{lemma}[properties of distance functions for $C^k$ surfaces] Let $\gamma$ be a $C^k$ surface, $k \ge 2$.  Then there exists a positive constant $\delta$ depending on $\gamma$ such that $d \in C^k(\mathcal{N}(\delta))$.  In addition, the closest point projection $\bP_d(\bx)=\bx-d(\bx) \nabla d(\bx)$ is defined and of class $C^{k-1}$ on $\mathcal{N}(\delta)$ with $\delta < \frac{1}{K_\infty}$. 
\end{lemma}
We now ask whether a similar statement holds for $k<2$, and in particular for $k=1$.  Note first that the distance function $d$ to {\it any} closed set $\gamma \subset \mathbb{R}^{n+1}$ is defined and Lipschitz continuous \cite[Theorem 4.8.1]{Fed59}, so the first question at hand is whether distance functions for $C^{1,\alpha}$ surfaces ($0 \le \alpha <1)$ are more than Lipshitz continuous.  


In order to understand the relationship between surface regularity and the distance function map, we first define the {\it reach} of a surface $\gamma$:  
\[
{\rm reach}(\gamma) := \sup \big\{\delta \ge 0:  \hbox{all } \bx\in \mathcal{N}(\delta) \hbox{ have a unique closest point } \bP_d(\bx) \in \gamma \big\}.
\]
For a $C^2$ surface $\gamma$, we have already seen that ${\rm reach}(\gamma)=1/K_\infty$.   We now explore the connection between the reach and properties of the distance function for less regular surfaces.  We first define
\[ U(\gamma):=\{\bx \in \mathbb{R}^{n+1} :~\bx \hbox{ has a unique closest point in } \gamma\}.
\]
The following result may be found in \cite[Theorem 4.8.3]{Fed59}.
\begin{lemma}[properties of differentiable distance functions]  \label{lem:d_properties}
If $\gamma$ is a $C^1$ surface, $\bx \in \mathbb{R}^{n+1} \setminus \gamma$, and $d$ is differentiable at $\bx$, then $\bx \in U(\gamma)$.  In particular, if $d$ is differentiable in a neighborhood of $\gamma$, then ${\rm reach}(\gamma)>0$.  
\end{lemma}

Next we define several constants from the technical report \cite{Luc57}.  Given $\bx \in \gamma$ and $\rho\ge 0$, we first define the closed normal segment
\[
S(\bx, \rho):= [\bx-\rho \bnu(\bx), \bx + \rho \bnu(\bx)].
\]
Let
$B_\rho(\by)$ denotes the ball in $\mathbb R^{n+1}$ of center $\by$ and radius $\rho>0$,
and
\begin{align*}
\begin{aligned}
\frac{1}{r_0} &:=\sup_{\bx,\by \in \gamma, \bx \neq \by} \frac{|\bnu(\bx)-\bnu(\by)|}{|\bx-\by|},
\\ \frac{1}{r_0{'}}&:=\sup \{\rho \ge 0: ~S(\bx, \rho) \cap S(\by, \rho)=\emptyset ~ \forall \bx, \by \in \gamma, ~\bx \neq \by\},
\\ \frac{1}{r_0{''}}&:=\sup \{ \rho \ge 0: ~ \overline{B_\rho (\bx \pm \rho \bnu(\bx))} \hbox{ contain respectively no points}  \\ &~~~~~~\hbox{ interior or exterior to } \gamma ~
\hbox{ for all } \bx \in \gamma\}, 
\\ \frac{1}{r_0{'''}}&:=\sup_{\bx,\by \in \gamma, \bx \neq \by} \frac{|2(\by-\bx) \cdot \bnu(\bx)|}{|\by-\bx|^2}.
\end{aligned}
\end{align*}

Combining \cite[Theorem 1]{Luc57} and noting that $r_0$ bounds the Lipschitz constant of $\gamma$ (cf. the comment on p. 15 of \cite{Luc57}), we have the following.
\begin{lemma}[further properties of $C^1$ surfaces] \label{lem:C1_properties} 
If the surface $\gamma$ is of class $C^1$, then the constants $r_0$, $r_0{'}$, $r_0{''}$, and $r_0{'''}$ are all equal.  In addition, if $r_0>0$ then  $\gamma$ is of class $C^{1,1}$.
\end{lemma}

%
%


Combining the previous lemmas with the statement in \cite[Theorem 4.18]{Fed59} that $r_0{'''}={\rm reach}(\gamma)$ yields the following result.
\begin{theorem}[$C^1$ distance function implies $C^{1,1}$ surface] \label{t:C1_implies_C2} If the distance function $d$ associated to a $C^1$ surface $\gamma$ is continuously differentiable in a tubular neighborhood $\mathcal{N}(\delta)$ of $\gamma$ for some $\delta>0$, then $\gamma$ is of class $C^{1,1}$.  In addition, any $C^1$ surface with positive reach is of class $C^{1,1}$.  
\end{theorem}
%

The preceding results establish that the properties of the distance function and the associated closest point projection for $C^2$ surfaces that we previously discussed are inherently connected with surfaces of bounded curvature.  This can be seen both in Theorem \ref{t:C1_implies_C2} (since the curvatures are defined and bounded almost everywhere on a $C^{1,1}$ surface) and in the definition of the constant $r_0{''}$ (since for $\bx \in \gamma$, the supremum over the radii $\rho$ for which $\overline{B_\rho (\bx \pm \rho \nu(\bx))} \cap \gamma = \emptyset$ is the inverse of the maximum principal curvature at $\bx$).  

For our purposes, Theorem \ref{t:C1_implies_C2} is essentially a negative result in that it establishes that the distance function and closest point projection are of limited immediate use for surfaces that are less regular than $C^{1,1}$.  In particular, in this case the closest point projection is not uniquely defined on any tubular neighborhood of $\gamma$.  In addition, the regularity of the distance function does {\it not} vary continuously with that of $\gamma$, since for a $C^{1,\alpha}$ surface with $\alpha<1$ Theorem \ref{t:C1_implies_C2}  establishes that $d$ is only Lipschitz.  Thus we must use different tools when considering surface finite element methods on less regular surfaces than $C^2$.  

\subsection{Divergence Theorem on Surfaces}\label{S:diver-thm}
%
We conclude this section with an application of calculus in $\mathbb{R}^{n+1}$
to derive an integration by parts formula on not necessarily closed surfaces.

\begin{proposition}[divergence theorem]\label{P:diver-thm}
Let $\gamma$ be a compact, oriented surface of class $C^2$ with Lipschitz boundary $\partial\gamma$.
Let $H = \sum_{i=1}^n \kappa_i$ be the total curvature of $\gamma$ and
$\bmu$ be the unit outward normal to $\partial\gamma$ lying in the tangent hyperplane to
$\gamma$. If $\widetilde \tv:\gamma\to\mathbb{R} \in H^1(\gamma)$, then
\[
\int_\gamma \nabla_\gamma \widetilde \tv = \int_\gamma \widetilde \tv H \bnu + \int_{\partial\gamma}\widetilde  \tv \bmu.
\]
\end{proposition}  
\begin{proof}
Given $\varepsilon < \frac{1}{2 K_\infty}$ we define the tubular set
\[
\Omega_\varepsilon := \big\{ \bz =  \bx + \rho \bnu(\bx):  \quad \bx\in\gamma, |\rho|<\varepsilon  \big\};
\]
note that $\bP_d(\bz)=\bx$ for all $\bz\in\Omega_\varepsilon$.
We decompose the boundary $\partial\Omega_\varepsilon$ of $\Omega_\varepsilon$ into
\[
\gamma_{\pm\varepsilon} := \big\{ \bx \pm \varepsilon \bnu(\bx): ~ \bx\in\gamma  \big\},
\quad
\lambda_\varepsilon := \partial \Omega_\varepsilon \setminus (\gamma_\varepsilon\cup\gamma_{-\varepsilon}).
\]
The sets $\gamma_{\pm\varepsilon}$ are parallel surfaces to $\gamma$ whereas
$\lambda_\varepsilon$ is the lateral boundary of size $2\varepsilon$. We first
assume that $\widetilde \tv$ is of class $C^1$, let $\tv$
be an extension of $\widetilde \tv$ to $\Omega_\varepsilon$ of class
$C^1(\overline{\Omega_\varepsilon})$,
and apply the divergence theorem in $\Omega_\varepsilon$ to obtain
\[
\int_{\Omega_\varepsilon} \nabla \tv = \int_{\partial\Omega_\varepsilon} \tv \bnu_\varepsilon
= \int_{\gamma_\varepsilon} \tv \, \bnu\circ\bP_d- \int_{\gamma_{-\varepsilon}} \tv \, \bnu\circ\bP_d
+ \int_{\lambda_\varepsilon} \tv \, \bmu\circ\bP_d,
\]
where $\bnu_\varepsilon$ is the unit outward normal of $\partial\Omega_\varepsilon$.
We divide both sides of this equality by $2\varepsilon$, the thickness of $\Omega_\varepsilon$
and compute the limits as $\varepsilon\to0$. According to \eqref{e:rel_dist_grad}
we first see that
\[
\frac{1}{2\varepsilon} \int_{\Omega_\varepsilon} \nabla\tv
= \frac{1}{2\varepsilon} \int_{\Omega_\varepsilon} \big(\bI - d(\bx) D^2 d(\bx) \big)
\nabla_\gamma \ttv(\bP_d(\bx)) d\bx
\mathop{\longrightarrow}_{\varepsilon \to 0} \int_\gamma \nabla_\gamma \ttv.
\]
Likewise
\[
\frac{1}{2\varepsilon} \int_{\lambda_\varepsilon} \tv \, \bmu\circ\bP_d
\mathop{\longrightarrow}_{\varepsilon \to 0} \int_{\partial\gamma} \ttv \, \bmu.
\]
Moreover, since $\bnu\circ\bP_d = \nabla d$, we infer that
\begin{align*}
\mathop{\lim}_{\varepsilon\to0} \frac{1}{2\varepsilon}
\Big( \int_{\gamma_\varepsilon} \tv \, \bnu\circ\bP_d
& - \int_{\gamma_{-\varepsilon}} \tv \, \bnu\circ\bP_d  \Big)
= \frac{d}{d\rho} \int_{\gamma_\rho} \tv \, \nabla d ~\Big|_{\rho=0}
\\
& = \frac{d}{d\rho} \int_{\mathcal{V}} \tv(\bx) \, \nabla d\big(\bx+\rho\nabla d(\bx)\big)
\, q_\rho(\by) d\by ~\Big|_{\rho=0}
\end{align*}
with $\bx = \bchi(\by)\in\gamma$ and $q_\rho(\by)$ denotes the infinitesimal area associated with the surface $\gamma_\rho:=\{ \bz = \bx + \rho \bnu(\bx):  \bx \in \gamma\}$. Since
$\frac{d}{d\rho}\nabla d\big(\bx+\rho\nabla d(\bx)\big)
= D^2 d\big(\bx+\rho\nabla d(\bx)\big) \nabla d(\bx) = 0$, it remains to evaluate
$\frac{d}{d\rho} q_\rho$. We resort to \eqref{e:area_ratio_distance} (shown below)
with $\Gamma=\gamma_\rho$ and use that $\bnu_\rho\cdot\bnu=1$
as well as \eqref{kappas} to write
\[
\frac{q_\rho(\by)}{q(\by)} = \frac{1}{\det\Big( \bI - \rho D^2 d(\bx) \Big)}
= \frac{1}{\mathop{\prod}_{i=1}^n \big(1-\rho\kappa_i(\bx) \big)}
= \mathop{\prod}_{i=1}^n \Big(1+\rho\kappa_i(\bP_d(\bx)) \Big).
\]
We finally observe that
\[
\frac{d}{d\rho} q_\rho(\by) \Big|_{\rho=0}
= q(\by) \sum_{i=1}^n \kappa_i(\bP_d(\bx)) = q(\by) H(\bP_d(\bx))
\]
to conclude the proof for $\widetilde tv$ of class $C^1$. The assertion for $\ttv\in H^1(\gamma)$
follows by density of $C^1(\overline{\gamma})$ in $\ttv\in H^1(\gamma)$.
\end{proof}

Applying Proposition \ref{P:diver-thm} (divergence theorem) to a vector field $\widetilde\bv:\gamma\to\mathbb{R}^{n+1}$
and computing the trace yields the more familiar expression
\begin{equation*}
\int_\gamma \textrm{div}_\gamma \widetilde \bv = \int_\gamma H \widetilde \bv\cdot\bnu + \int_{\partial\gamma} \widetilde\bv\cdot\bmu.
\end{equation*}
\begin{corollary}[integration by parts]\label{C:int-parts}
Let $\gamma$ be a surface of class $C^2$ with Lipschitz boundary $\partial\gamma$.
If $\ttv, \widetilde w:\gamma\to\mathbb{R}$ satisfy $\ttv\in H^2(\gamma)$ and $\widetilde w\in H^1(\gamma)$,
then
\[
\int_\gamma \widetilde w \, \Delta_\gamma \ttv + \nabla_\gamma \widetilde w \cdot \nabla_\gamma \widetilde w
= \int_{\partial\gamma} \widetilde w \, \nabla_\gamma \ttv\cdot\bmu.
\]
\end{corollary}
\begin{proof}
Apply the previous equality to $\widetilde \bv = \widetilde w \, \nabla_\gamma \ttv$.
\end{proof}

\section{Perturbation Theory}\label{S:perturbation}
%
In most surface finite element methods, the approximate problem is not posed on the continuous surface $\gamma$.  This may occur either for convenience, or because $\gamma$ is not known precisely. Examples of only incomplete information being present in simulations include free boundary problems such as two-phase flow and cases where $\gamma$ is reconstructed from some sort of imaging data. 

The purpose of this section is to investigate how geometric quantities change
under perturbation of the surface $\gamma$. To this end, suppose that
$\Gamma$ is a closed Lipschitz surface (not necessarily $C^2$). We use a subscript $\Gamma$
to denote geometric quantities associated with $\Gamma$:
$\bchi_\Gamma$ (parametrization), $\bg_\Gamma$
(first fundamental form), $q_\Gamma$ (area element), $\bnu_\Gamma$ (unit normal),
$\nabla_\Gamma$ (tangential gradient), and $\Pi_\Gamma$ (orthogonal projection
onto $\Gamma$).

Let $\widetilde u\in H^1_\#(\gamma)$ solve \eqref{e:weak_relax} and $u_\Gamma\in H^1_\#(\Gamma)$ solve
\begin{equation}\label{Gamma:LBproblem}
  \int_\Gamma \nabla_\Gamma u_\Gamma \cdot \nabla_\Gamma \tv = \int_\Gamma f_\Gamma \tv
  \quad\forall \, \tv\in H^1_\#(\Gamma),
\end{equation}  
for a given forcing $f_\Gamma\in L_{2,\#}(\Gamma)$.
To examine the error between $u$ and $u_\Gamma$, we first have to study how the bilinear
forms in \eqref{e:weak_relax} and \eqref{Gamma:LBproblem} change
when changing $\gamma$. This amounts to
deriving expressions for the {\it error matrices}
$\bE,\bE_\Gamma \in\mathbb{R}^{(n+1)\times(n+1)}$ in the error equations
\begin{equation}\label{consistency}
  \int_\Gamma \nabla_\Gamma \tv \cdot \nabla_\Gamma w
  - \int_\gamma \nabla_\gamma \widetilde{\tv} \cdot \nabla_\gamma \widetilde{w}
  = \int_\gamma \nabla_\gamma \widetilde{\tv} \cdot \bE \, \nabla_\gamma \widetilde{w}
  = \int_\Gamma \nabla_\Gamma \tv \cdot \bE_\Gamma \, \nabla_\Gamma w,
\end{equation}
valid for all $\tv,w \in H^1(\Gamma)$ and 
$\widetilde{\tv}, \widetilde{w} \in H^1(\gamma)$ the corresponding lifts.
We carry out this program below within two scenarios depending on the regularity
of $\gamma$. We alert the reader about the following abuse of notation:
the matrix $\bE$ (resp. $\bE_\Gamma$) is defined in $\gamma$ (resp. $\Gamma$),
but we will often write them in the parametric domain $\mathcal{V}$ thereby
identifying $\bE$ (resp. $\bE_\Gamma$) with $\bE\circ\bchi$
(resp. $\bE_\Gamma\circ\bchi_\Gamma$).

\subsection{Perturbation Theory for $C^{1,\alpha}$ Surfaces}\label{S:perturb-C1alpha}
%
Let $\gamma$ be of class $C^{1,\alpha}$ and
$\bchi$ and $\bchi_\Gamma$ be the parametrizations of $\gamma$ and $\Gamma$.
They dictate the relation between $\widetilde{\tv}$ and $\tv$, the former
defined on $\gamma$ and the latter on $\Gamma$,
\[
\tv = \wv \circ \bchi \circ \bchi_\Gamma^{-1}.
\]
In the sequel, we first establish a relation between $\nabla_\gamma\widetilde{\tv}$ and
$\nabla_\Gamma\tv$ and next use it to characterize $\bE$ and $\bE_\Gamma$.

\begin{lemma}[relation between tangential gradients]\label{L:tan-grads-lip}
If $\wv:\gamma\to\mathbb{R}$ is of class $H^1$, then the tangential
gradients $\nabla_\gamma\widetilde{\tv}$ and $\nabla_\Gamma\tv$ satisfy
\begin{equation}\label{tan-grads-lip}
  \nabla_\Gamma \tv = D\bchi_\Gamma \, \bg_\Gamma^{-1} \, D\bchi^t \,
  \nabla_\gamma \widetilde{\tv},
  \qquad
  \nabla_\gamma \widetilde{\tv} = D\bchi \, \bg^{-1} \, D\bchi_\Gamma^t \,
  \nabla_\Gamma \tv.
\end{equation}
\end{lemma}
\begin{proof}
We concatenate \eqref{e:exact_grad} and \eqref{def-tang-grad} to write
\[
\nabla_\Gamma \tv = D\bchi_\Gamma \, \bg_\Gamma^{-1} \, \nabla (\tv\circ\bchi_\Gamma)
= D\bchi_\Gamma \, \bg_\Gamma^{-1} \, \nabla (\wv\circ\bchi)
= D\bchi_\Gamma \, \bg_\Gamma^{-1} \, D\bchi^t \, \nabla_\gamma\widetilde{\tv},
\]
which is the first asserted expression provided $\wu$ is of class $C^1$. Using
the density of $C^1(\gamma)$ in $H^1(\gamma)$ for a surface $\gamma$ of class
$C^{1,\alpha}$, the first assertion follows. The second one follows similarly.
\end{proof}

\begin{lemma}[geometric consistency]\label{L:geom_consist}
The error matrices $\bE$ and $\bE_\Gamma$ read on $\mathcal{V}$
\begin{gather}
\label{error-matrix-g}
\bE = D\bchi \Big(\frac{q_\Gamma}{q} \bg_\Gamma^{-1} - \bg^{-1} \Big) D\bchi^t,
\\
\label{error-matrix-G}
\bE_\Gamma = D\bchi_\Gamma \Big(\bg_\Gamma^{-1} - \frac{q}{q_\Gamma}\bg^{-1}
\Big) D\bchi_\Gamma^t.
\end{gather}
\end{lemma}
\begin{proof}
Using \eqref{tan-grads-lip}, together with the definition \eqref{dg:d:first} of
$\bg_\Gamma=D\bchi_\Gamma^t D\bchi_\Gamma$, yields
\[
\int_\Gamma \nabla_\Gamma \tv\cdot\nabla_\Gamma w =
\int_\gamma \nabla_\gamma \ttv \cdot
\frac{q_\Gamma}{q}\big( D\bchi\bg_\Gamma^{-1} D\bchi^t \big) \nabla_\gamma \widetilde w.
\]
Since $\nabla_\gamma \widetilde \tv=\Pi\nabla_\gamma \widetilde \tv=D\bchi \, \bg^{-1} \, D\bchi^t\nabla_\gamma \widetilde \tv$, according to
\eqref{projection-b}, the first equality in \eqref{consistency} follows
immediately. The proof of the second equality is similar.
\end{proof}

Our task now is to relate  $\bg-\bg_\Gamma$ and $q-q_\Gamma$
with $D(\bchi-\bchi_\Gamma)$. We accomplish this next but first we introduce some
additional concepts. For any $\by\in\mathcal{V}$, we denote by $|D\bchi(\by)|$
(resp. $|D^-\bchi(\by)|$) the largest (resp. smaller) singular value of $D\bchi(\by)$.
Given the relation $\bg = D\bchi^t \, D\bchi$, these quantities are the square
roots of the largest and smallest eigenvalues of $\bg$. We define the
stability constant
\begin{equation}\label{stab-const}
S_\bchi := \sup_{\by\in\mathcal{V}} ~
\frac{\max\big\{|D\bchi(y)|,|D\bchi_\Gamma(y)|\big\}}{\min\big\{|D^-\bchi(y)|,|D^-\bchi_\Gamma(y)|\big\}}
\end{equation}
and point out that it is a measure of non-degeneracy of $D\bchi$ and $D\bchi_\Gamma$.
We further define the following relative measure of geometric accuracy
\begin{equation}\label{geo-est}
  \lambda_\infty := \sup_{\by\in\mathcal{V}} ~
  \frac{|D(\bchi-\bchi_\Gamma)(\by)|}{\min\big\{|D^-\bchi(\by)|,|D^-\bchi_\Gamma(\by)|\big\}}.
\end{equation}  

\begin{lemma}[error estimates for $\bg$ and $q$]\label{L:error-est}
The following error estimates are valid
\begin{gather}\label{error-est-g}
  \|\bI-\bg_\Gamma\bg^{-1}\|_{L_\infty(\mathcal{V})}, \,
  \|\bI-\bg_\Gamma^{-1}\bg\|_{L_\infty(\mathcal{V})}
  \lesssim S_\bchi \, \lambda_\infty,
\\ \label{error-est-q}
  \|1-q^{-1}q_\Gamma\|_{L_\infty(\mathcal{V})}, \, \|1-q_\Gamma^{-1}q\|_{L_\infty(\mathcal{V})}
  \lesssim S_\bchi^n \, \lambda_\infty.
\end{gather}
\end{lemma}
\begin{proof}
  Since $|D\bchi|=|D\bchi^t|$, $|\bg^{-1}| \le |D^-\bchi|^{-2}$ and
\[
(\bg - \bg_\Gamma)(\by) =
D\bchi(\by)^t D(\bchi-\bchi_\Gamma)(\by)
+ D(\bchi-\bchi_\Gamma)(\by)^t D\bchi_\Gamma(\by)
\quad\forall \, \by\in\mathcal{V},
\]
the first assertion in \eqref{error-est-g} follows; the second one is similar.
To prove \eqref{error-est-q}, we write
\[
q(\by) - q_\Gamma(\by) = \frac{\det\bg(\by) - \det\bg_\Gamma(\by)}{q(\by)+q_\Gamma(\by)}
\quad\forall \, \by\in\mathcal{V},
\]
and note that $q = \sqrt{\det\bg}=\sqrt{\prod_{i=1}^n \lambda_i(\bg)}$ where
$\{\lambda_i(\bg)\}_{i=1}^n$ are the eigenvalues of $\bg$. Utilizing the definitions
of $|D\bchi|$ and $|D^-\bchi|$ we end up with
\begin{equation}\label{q:nondegen}
|D^-\bchi(\by)|^n \le q(\by) \le |D\bchi(\by)|^n
\quad\forall \, \by\in\mathcal{V}.
\end{equation}
Since $\det\bg - \det\bg_\Gamma$ is the sum of terms of the form
$\partial_i\bchi\cdot\partial_j\bchi - \partial_i\bchi_\Gamma\cdot\partial_j\bchi_\Gamma$
multiplied by $n-1$ factors bounded by $|D\bchi|$, we deduce
\[
|q(\by)^{-1}(q-q_\Gamma)(\by)| \lesssim
|D^-\bchi(\by)|^{-n} \, |D(\bchi-\bchi_\Gamma)(\by)| \, |D\bchi(\by)|^{n-1}
\quad\forall \, \by\in\mathcal{V}.
\]
This is the first assertion in \eqref{error-est-q} in disguise. The second one
is similar.
\end{proof}

\begin{lemma}[norm equivalence]\label{L:norm-equiv}
Let $\gamma$ and $\Gamma$ be Lipschitz surfaces which are related via a bi-Lipschitz map $\bP=\bchi \circ \bchi^{-1}_\Gamma :\Gamma \rightarrow \gamma$. Then there is a constant $C \ge 1$, depending on the stability constant $S_\bchi$ in \eqref{stab-const}, such that
\begin{gather}
\label{L2:equiv}
C^{-1} \|\tv\|_{L_2(\Gamma)} \le \|\ttv\|_{L_2(\gamma)}
\le C \|\tv\|_{L_2(\Gamma)}
\quad\forall \, \ttv \in L_2(\gamma),
\\ \label{H1:equiv}
C^{-1} \|\nabla_\Gamma \tv\|_{L_2(\Gamma)} \le \|\nabla_\gamma \ttv\|_{L_2(\gamma)}
\le C \|\nabla_\Gamma \tv\|_{L_2(\Gamma)}
\quad\forall \, \ttv \in H^1(\gamma).
\end{gather}
\end{lemma}
\begin{proof}
Use \eqref{def-tang-grad} and \eqref{e:exact_grad} in conjunction with \eqref{int-gamma}.
\end{proof}  
  
Lemma~\ref{L:Poincare} (Poincar\'e-Friedrichs inequality) holds on the perturbed surface $\Gamma$ but with a constant depending on $\Gamma$.  In order to avoid this dependence, and thus obtain a uniform constant in $\Gamma$, it is only necessary that Lemma \ref{L:norm-equiv} (norm equivalence) be valid. Before stating our result, we first define a class of surfaces.  Given a Lipschitz surface $\gamma$, we let $\mathcal{S}_{eq}$ be the class of Lipschitz surfaces $\Gamma$ such that Lemma \ref{L:norm-equiv} (norm equivalence) holds with uniform equivalence constant $C_{eq}$. Note that implicit in this definition is the existence of a bi-Lipschitz bijection $\bP:\Gamma \rightarrow \gamma$ for each $\Gamma \in \mathcal{S}_{eq}$,  for instance $\bP = \bchi\circ\bchi_\Gamma^{-1}$.

\begin{lemma}[uniform Poincar\'e-Friedrichs constant]\label{L:Poincare-unif}
Given a Lipschitz surface $\gamma$, for every $\tv \in H^1_\#(\Gamma)$ with $\Gamma \in \mathcal{S}_{eq}$ there holds that
\begin{equation}\label{poincare_unif}
\|\tv\|_{L_2(\Gamma)} \lesssim \|\nabla_\Gamma u\|_{L_2(\Gamma)}
\end{equation}
with the constant hidden in $\lesssim$ depending only on $\gamma$ and $C_{eq}$. 
\end{lemma}
\begin{proof}
We argue by contradiction the validity of
\[
\|\tv\|_{L_2(\Gamma)} \le C \|\nabla_\Gamma \tv\|_{L_2(\Gamma)}
    \quad\forall \, \tv\in H^1(\Gamma)
\]
and all $\Gamma\in\mathcal{S}_{eq}$ with uniform constant $C$. We thus assume the existence of a sequence of surfaces $\Gamma_k \in \mathcal{S}_{eq}$ and functions $\tv_k \in H^1_\#(\Gamma_k)$ such that
\[
\|\tv_k\|_{L_2(\Gamma_k)}=1
\qquad
\|\nabla_{\Gamma_k} \tv_k\|_{L_2(\Gamma_k)} \rightarrow 0
\]
as $k \rightarrow \infty$.  We denote by $\bP_k: \Gamma_k \rightarrow \gamma$ the associated bi-Lipschitz bijections and by $\widetilde{\tv}_k = \tv_k \circ \bP_k^{-1}$ the lifts of the functions $\tv_k$ to $\gamma$. Since $\Gamma_k\in\mathcal{S}_{eq}$, the estimates of Lemma \ref{L:norm-equiv} (norm equivalence) hold with uniform constant $C_{eq}$ for each $\Gamma_k$, whence $\wv_k \in H^1(\gamma)$ and
\[
\|\wv_k\|_{L_2(\gamma)} \simeq 1 ,\qquad
\|\nabla_\gamma\wv_k\|_{L_2(\gamma)} \rightarrow 0
\]
as $k \rightarrow \infty$. Proceeding as in Lemma \ref{L:Poincare}
(Poincar\'e-Friedrichs inequality), we deduce that a subsequence
of $\{\wv_k\}_k$, still denoted $\{\wv_k\}_k$, converges in $H^1(\gamma)$ to a function
$\wv\in H^1(\gamma)$ with $\nabla_\gamma\wv=0$; this implies that $\wv$ is constant. To show that $\wv=0$, let
$\epsilon>0$ be arbitrary and $k$ sufficiently large so that $\|\wv_k-\wv\|_{L_2(\gamma)} \le \epsilon$. Exploiting that $\widetilde{\tv}$ is constant and $\int_{\Gamma_k}\tv_k=0$, we use Lemma \ref{L:norm-equiv} to compute
\[
\begin{aligned}
|\widetilde{\tv}| & =|\Gamma_k|^{-1} \left | \int_{\Gamma_k} \widetilde{\tv} \, \right | = |\Gamma_k|^{-1} \left | \int_{\Gamma_k} \widetilde{\tv}-\tv_k \, \right | 
\\ & \le |\Gamma_k|^{-1/2} \|\widetilde{\tv}-\tv_k\|_{L_2(\Gamma_k)} \le C_{eq} |\Gamma_k|^{-1/2} \|\widetilde{\tv}-\widetilde{\tv}_k\|_{L_2(\gamma)} \le C_{eq} |\Gamma_k|^{-1/2} \epsilon.
\end{aligned}
\]
Applying again Lemma \ref{L:norm-equiv}, now to the function $1$, yields $|\Gamma_k|\simeq |\Gamma|$ with constant depending only on $C_{eq}$, so that $|\widetilde{\tv}| \lesssim \epsilon$. Since $\epsilon$ is arbitrary, we must thus  have $\widetilde{\tv}=0$.  This contradicts $\|\widetilde{\tv}_k\|_{L_2(\gamma)} \simeq 1$ because $\|\widetilde{\tv}_k\|_{L_2(\gamma)} \rightarrow \|\wv\|_{L_2(\gamma)}=0$. Consequently, the desired statement is proved.
\end{proof}

\begin{lemma}[perturbation error estimate for $C^{1,\alpha}$ surfaces]\label{L:perturbation_bound}
Let $\wu\in H^1_\#(\gamma)$ solve \eqref{e:weak_relax} and $u_\Gamma\in H^1_\#(\Gamma)$
solve \eqref{Gamma:LBproblem}. Then, the following error estimate for $u-u_\Gamma$ holds
\begin{equation}\label{perturbation_bound}
\|\nabla_\gamma(\wu-\wu_\Gamma)\|_{L_2(\gamma)} \lesssim \lambda_\infty \|f_\Gamma\|_{H^{-1}_\#(\Gamma)} + \| f q q_\Gamma^{-1} - f_\Gamma \|_{H^{-1}_\#(\Gamma)},
\end{equation}
where the hidden constant depends on $S_\bchi$ defined in \eqref{stab-const}.
\end{lemma}
\begin{proof}
We proceed in several steps.

\noindent
{\it Step 1: error representation}.
Let $\ttv = \wu-\wu_\Gamma$ and make use of \eqref{consistency} to write
\begin{equation*}
\|\nabla_\gamma ( \wu- \wu_\Gamma)\|_{L_2(\gamma)}^2
= \int_\gamma \nabla_\gamma  \wu \cdot \nabla_\gamma \ttv
-  \int_\Gamma \nabla_\Gamma u_\Gamma \cdot \nabla_\Gamma \tv
+ \int_\gamma \nabla_\gamma  \wu_\Gamma \cdot \bE \, \nabla_\gamma \ttv.
\end{equation*}
We next employ the equations 
\eqref{e:weak_relax} and \eqref{Gamma:LBproblem} satisfied by $\widetilde u$ and $u_\Gamma$
to obtain
\begin{equation*}
  \|\nabla_\gamma (\wu-\wu_\Gamma)\|_{L_2(\gamma)}^2  =
  \int_\Gamma \Big(f\frac{q}{q_\Gamma} - f_\Gamma \Big) \tv
+ \int_\gamma \nabla_\gamma \wu_\Gamma \cdot \bE \, \nabla_\gamma \wv,
\end{equation*}
where we have also employed \eqref{int-gamma} to switch the domain of integration
of $f$.

\smallskip\noindent
{\it Step 2: geometric error matrix}.
To derive a bound for $\|\bE \|_{L_\infty(\gamma)}$, we rewrite $\bE$
\[
\bE = D\bchi \Big( \big(q^{-1}q_\Gamma-1\big) \bg^{-1}_\Gamma
- \bg^{-1} \big(\bI - \bg\bg_\Gamma^{-1} \big) \Big) D\bchi^t .
\]
Since $|\bg^{-1}|=|D^-\bchi|^{-2}, |\bg_\Gamma^{-1}|=|D^-\bchi_\Gamma|^{-2}$,
applying \eqref{error-est-g} and \eqref{error-est-q} leads to the error estimate
\begin{equation}\label{bound-E}
\|\bE \|_{L_\infty(\gamma)} \lesssim \lambda_\infty.
\end{equation}
\smallskip\noindent
{\it Step 3: final estimates}.
The Cauchy-Schwarz inequality yields
\[
\int_\gamma \nabla \wu_\Gamma \cdot \bE \, \nabla_\gamma \ttv \le
\|\nabla_\gamma \ttv\|_{L_2(\gamma)}
\|\nabla_\gamma \wu_\Gamma\|_{L_2(\gamma)} \|\bE \|_{L_\infty(\gamma)}.
\]
To derive a bound for $\|\nabla_\gamma \wu_\Gamma\|_{L_2(\gamma)}$,
we first combine \eqref{dual-norm} with \eqref{Gamma:LBproblem} to obtain
$\|\nabla_\Gamma u_\Gamma\|_{L_2(\Gamma)} \le \|f_\Gamma\|_{H^{-1}_\#(\Gamma)}$,
and next appeal to Lemma \ref{L:norm-equiv} (norm equivalence).
On the other hand, we recall that $f\frac{q}{q_\Gamma}-f_\Gamma$ has vanishing mean-value on
$\Gamma$, let $\overline{\tv} = |\Gamma|^{-1} \int_\Gamma \tv$ be the mean-value
of $\tv$, and use \eqref{dual-norm} to arrive at
\[
\int_\Gamma \Big(f\frac{q}{q_\Gamma} - f_\Gamma \Big) \tv =
\int_\Gamma \Big(f\frac{q}{q_\Gamma} - f_\Gamma \Big) \big( \tv - \overline{\tv} \big)
\le \|f q q_\Gamma^{-1} - f_\Gamma \|_{H^{-1}_\#(\Gamma)} \|\nabla_\Gamma\tv\|_{L_2(\Gamma)}.
\]
Finally, applying Lemma \ref{L:norm-equiv} ends the proof.
\end{proof}

\subsection{Perturbation Theory for $C^2$ Surfaces}\label{S:perturb-C2}
%
Let $\gamma$ be of class $C^2$ and the tubular neighborhood $\mathcal{N}$
satisfy \eqref{N:def}, namely
\begin{equation}\label{N:def-2}
\mathcal{N} = \Big\{ \bx\in\mathbb{R}^{n+1}: ~ |d(\bx)| < \frac{1}{2K_\infty} \Big\},
\end{equation}
so that parallel surfaces to $\gamma$ within $\mathcal{N}$ are also $C^2$.
We further assume that $\Gamma\subset\mathcal{N}$ and the distance function projection
$\bP_d=\bI-d\nabla d:\Gamma \rightarrow \gamma$ is a bijection. The parametrizations
of $\gamma$ and $\Gamma$ are given by $\bchi := \bP_d \circ \bchi_\Gamma$ so that
\[
\tv = \widetilde{\tv} \circ \bP_d.
\]

\begin{lemma}[relation between tangential gradients]\label{L:tan-grads}
If $\wv:\gamma\to\mathbb{R}$ is of class $H^1$, then the tangential
gradients $\nabla_\gamma\wv$ and $\nabla_\Gamma\tv$ satisfy for all
$\bx\in\Gamma$
\begin{equation}\label{e:tang_exact_to_discrete}
  \nabla_\Gamma \tv(\bx) = \Pi_\Gamma(\bx) \, \big(\bI - d \bW\big)(\bx) \, \Pi(\bx)
  \nabla_\gamma \wv(\bP_d(\bx)),
\end{equation}
and
\begin{equation}\label{e:tang_discrete_to_exact}
\nabla_\gamma \wv(\bP_d(\bx)) = \big(\bI - d \bW \big)^{-1}(\bx)\left(\bI - \frac{\bnu_\Gamma(\bx) \otimes \bnu(\bx)}{\bnu_\Gamma(\bx) \cdot \bnu(\bx)}\right) \nabla_\Gamma \tv(\bx).
\end{equation}
\end{lemma}  
\begin{proof}
Let us assume that $\wv \in C^1(\gamma)$.
Recalling \eqref{e:rel_dist_grad} and \eqref{grad-extension}, we readily get
\[
\nabla_\Gamma \tv(\bx) = \Pi_\Gamma (\bx) \nabla\tv(\bx)
= \Pi_\Gamma (\bx) \, \big(\bI - d \bW\big)(\bx) \, \Pi(\bx)
  \nabla_\gamma \wv(\bP_d(\bx)),
\]
hence \eqref{e:tang_exact_to_discrete}. Since $\bI-d(\bx)\bW(\bx)$ is invertible
for all $\bx\in\mathcal{N}$, according to the definition \eqref{N:def} of
$\mathcal{N}$ and shown in Lemma \ref{L:curv-parallel} (curvature of parallel
surfaces), \eqref{e:rel_dist_grad} can be rewritten as
\[
\nabla_\gamma \wv(\bP_d(\bx)) = \big(I - d \bW)(\bx) \big)^{-1} \nabla \tv(\bx)
\quad\forall \, \bx \in \mathcal N.
\]
To prove \eqref{e:tang_discrete_to_exact} we must relate $\nabla \tv$ and
$\nabla_\Gamma \tv$. First note that for $\bx\in\Gamma$
\[
\nabla \tv = (\bI - \bnu_\Gamma\otimes\bnu_\Gamma) \nabla\tv
+ \bnu_\Gamma\otimes\bnu_\Gamma \nabla\tv
= \nabla_\Gamma \tv + (\nabla\tv\cdot\bnu_\Gamma) \bnu_\Gamma.
\]
Exploiting next that $\nabla \tv(\bx)\cdot\bnu(\bx)=0$, because $\tv(\bx)$ is constant in the normal direction to $\bP_d(\bx)$, yields
\[
 \nabla_\Gamma \tv \cdot \bnu + (\bnu_\Gamma \cdot \bnu) \nabla \tv \cdot \bnu_\Gamma
 =0 \quad\Rightarrow\quad
 \nabla \tv \cdot \bnu_\Gamma = - \frac{1}{\bnu_\Gamma \cdot \bnu}\nabla_\Gamma \tv \cdot \bnu.
 \]
 Since $\nabla\tv = \nabla_\Gamma\tv + (\nabla \tv\cdot\bnu_\Gamma) \bnu_\Gamma$,
 we deduce
 \[
 \nabla \tv(\bx)  = \left(\bI - \frac{\bnu_\Gamma(\bx) \otimes \bnu(\bx)}{\bnu_\Gamma(\bx) \cdot \bnu(\bx)}\right) \nabla_\Gamma \tv(\bx)
 \quad\forall \, \bx\in\Gamma.
 \]
 Inserting this into the previous expression for $\nabla_\gamma \tv(\bP_d(\bx))$
 leads to \eqref{e:tang_discrete_to_exact}.
 Finally, a density argument of $C^1(\gamma)$
 in $H^1(\gamma)$ for $\gamma$ of class $C^2$ concludes the proof.
 \end{proof}

The following result mimics Lemma \ref{L:geom_consist} (geometric consistency)
except that now it quantifies the effect of perturbing the surface $\gamma$ on the
bilinear forms written in \eqref{consistency} in terms of $\bP_d$.

\begin{lemma}[geometric consistency]\label{L:geom_consist_dist}
The error matrices
$\bE,\bE_\Gamma \in\mathbb{R}^{(n+1)\times(n+1)}$ in \eqref{consistency}
are given on $\Gamma$ by
\begin{gather}\label{error-matrix-gamma}
  \bE\circ\bP_d := \frac{q_\Gamma}{q} \Pi \big(\bI
  - d\bW \big) \Pi_\Gamma
  \big(\bI - d \bW \big) \Pi  - \Pi,
  \\
  \label{error-matrix-Gamma}
  \bE_\Gamma :=  { q \over q_\Gamma } \left( I - \frac{\bnu \otimes \bnu_\Gamma}{ \bnu \cdot \bnu_\Gamma}\right) (\bI -d \bW)^{-2}  \left( I - \frac{\bnu_\Gamma \otimes \bnu}{ \bnu \cdot \bnu_\Gamma}\right)- \Pi_\Gamma.
\end{gather}  
\end{lemma}
\begin{proof}
In view of \eqref{int-gamma}, \eqref{e:tang_exact_to_discrete}, and the fact
that all matrices involved are symmetric and $\Pi_\Gamma^2=\Pi_\Gamma$, we can write
\[
\int_\Gamma \nabla_\Gamma w \cdot \nabla_\Gamma\tv
= \int_\gamma \nabla_\gamma \widetilde{w} \cdot
\Big(\frac{q_\Gamma}{q} \Pi \big(\bI - d\bW \big)
\Pi_\Gamma\big(\bI - d \bW \big) \Pi\Big)
\nabla_\gamma\widetilde\tv
\]
Noticing that $\nabla_\gamma \widetilde{w} = \Pi \nabla_\gamma \widetilde{w}$
the first equality on \eqref{consistency} follows immediately. The second
equality proceeds along the same lines but using \eqref{e:tang_discrete_to_exact}
instead.
\end{proof}  

It is clear from Lemma \ref{L:geom_consist_dist} that the ratio of area
elements $q / q_\Gamma$ matters.
We next derive a representation for $q / q_\Gamma$ for any dimension $n$,
proved originally for $n=2,3$ in \cite{DemlowDziuk:07,De09}.
We stress that, in view of Remark \ref{R:param-indep} (parametric independence),
the solution $u$ of the Laplace-Beltrami equation \eqref{e:weak} is independent
of the parametrization of $\gamma$. This allows us to consider a convenient
parametrization $\bchi$ for theory because it does not change the geometric
objects under consideration. We exploit this flexibility next.

\begin{lemma}[relation between $q$ and $q_\Gamma$]\label{L:area_ratio_distance}
  Given any parametrization $\bchi_\Gamma$ of $\Gamma$, let
  $\bchi := \bP_d \circ \bchi_\Gamma$ be the parametrization of $\gamma$.
  If $\bnu(\bx) \cdot \bnu_\Gamma(\bx)\ge0$ for all $\bx\in\Gamma$,
  then the ratio of area elements $q(\by)/q_\Gamma(\by)$ with
  $\by=\bchi_\Gamma^{-1}(\bx)$ satisfies
\begin{equation}\label{e:area_ratio_distance}
  \frac{q(\by)}{q_\Gamma(\by)}
  = \det\Big(\bI-d(\bx) \bW(\bx)\Big) (\bnu(\bx) \cdot \bnu_\Gamma(\bx))
  \quad\forall \, \bx \in \Gamma.
\end{equation}
\end{lemma}
\begin{proof}
We start with the formula~\eqref{e:area-rep} for the area elements $q$
and $q_\Gamma$ to get
\[
\frac{q}{q_\Gamma} = \det\left(\lbrack \bnu, D\exactparam \rbrack
\,\lbrack \bnu_\Gamma, D\bchi_\Gamma \rbrack^{-1} \right).
\]
We write $ \lbrack \bnu_\Gamma, D\bchi_\Gamma \rbrack^{-1} = \lbrack \bv,\bA \rbrack^t$
for some $\bv \in \mathbb R^{n+1}$ and $\bA\in \mathbb R^{(n+1)\times n}$ to be found.
The identity  $ \lbrack \bv , \bA \rbrack^t  \lbrack \bnu_\Gamma, D\bchi_\Gamma \rbrack
= \bI$ yields
$
\bv = \bnu_\Gamma
$
while $   \lbrack \bnu_\Gamma, D\bchi_\Gamma \rbrack \lbrack \bv, \bA \rbrack^t= \bI$ gives
$
D\bchi_\Gamma \bA^t = \bI - \bnu_\Gamma \otimes \bnu_\Gamma = \Pi_\Gamma
$ and
\[
\lbrack \bnu, D\exactparam \rbrack
\,\lbrack \bnu_\Gamma, D\bchi_\Gamma \rbrack^{-1}
=\bnu\otimes\bnu_\Gamma + D\bchi \, \bA^t.
\]
To obtain an expression for $D\bchi$, let $\bx = \bchi_\Gamma(\by)\in\Gamma$ and
$\bchi(\by) = \bP_d(\bx) = \bx - d(\bx)\nabla d(\bx) \in\gamma$, and utilize
the chain rule
\[
D\bchi(\by) = \big(\bI - d(\bx)\bW(\bx) \big) \, \Pi(\bx) \, D\bchi_\Gamma(\by)
\quad\forall \, \by\in\mathcal{V},
\]
where we have argued as in \eqref{e:rel_dist_grad}. Compute now $D\bchi\bA^t$
and use that $D\bchi_\Gamma\bA^t=\Pi_\Gamma$ together with $\bW \bnu =0$ to arrive at
\begin{align*}
\frac{q}{q_\Gamma} &= \det \big( \bnu\otimes\bnu_\Gamma
+ (\bI - d\bW) \, \Pi \, \Pi_\Gamma  \big) 
\\
& = \det\big( (\bI-d\bW) \big(\bnu \otimes \bnu_\Gamma + \Pi \, \Pi_\Gamma \big) \big)
= \det\big( (\bI-d\bW) \big) \, \det\bB.
\end{align*}
where $\bB := \bnu \otimes \bnu_\Gamma + \Pi \,\Pi_\Gamma$. It thus remains to show that
$\det \bB = \bnu \cdot \bnu_\Gamma$.

We now embark on a spectral analysis of $\bB$.
We first note that the statement is trivial if $\bnu=\bnu_\Gamma$. We thus
assume that $\{\bnu,\bnu_\Gamma\}$ are linearly independent and that the
space $\mathbb{X} = \textrm{span} \{\bnu,\bnu_\Gamma\}$
is generated by two orthonormal vectors $\bnu$ and $\be$. 
We consider the orthogonal decomposition
$\mathbb{R}^{n+1} = \mathbb{X} \oplus \mathbb{X}^\perp$ and
a rotation $\bR\in\mathbb{R}^{(n+1)\times(n+1)}$ on $\mathbb{X}$
that maps $\bnu$ into $\bnu_\Gamma$, namely
\[
\bR \bnu = \bnu_\Gamma = \cos \theta \, \bnu + \sin\theta \, \be,
\quad
\bR \be = -\sin\theta \, \bnu + \cos \theta \, \be;
\]
thus the rotation angle $\theta$ satisfies $\cos\theta=\bnu \cdot \bnu_\Gamma$
and $\det\bR=1$. Consequently,
\[
\bB = \big( \bnu\otimes\bnu + \Pi \, \bR \, \Pi  \big) \bR^t
\quad\Rightarrow\quad
\det \bB = \det \big( \bnu\otimes\bnu + \Pi \, \bR \, \Pi  \big).
\]
The proof concludes upon realizing that $\bnu$ and $\be$ are
eigenvectors of $\bnu\otimes\bnu + \Pi \, \bR \, \Pi$ 
with eigenvalues $1$ and $\cos\theta$, and the remaining eigenvalues are $1$
with eigenspace $\mathbb{X}^\perp$.
\end{proof}

We are now ready to compare solutions $u$ and $u_\Gamma$
of \eqref{e:weak_relax} on two nearby surfaces $\gamma$ and $\Gamma$.
In essence, weak solutions $u$ and $u_\Gamma$ are close in $H^1$ provided
$\gamma$ and $\Gamma$ are close in a Lipschitz sense.
Therefore, to make this statement quantitative we
introduce the following geometric quantities
\begin{equation}\label{geom-quantities}
d_\infty := \|d\|_{L_\infty(\Gamma)},
\quad
\nu_\infty := \|\bnu-\bnu_\Gamma\|_{L_\infty(\Gamma)},
\quad
K_\infty := \|K\|_{L_\infty(\gamma)},
\end{equation}
where $\Gamma\subset\mathcal{N}$ is a Lipschitz surface.
Our goal is to bound $\|u-u_\Gamma\|_{H_\#^1(\Gamma)}$ in terms of the forcing functions
$f,f_\Gamma$, and $d_\infty, \nu_\infty, K_\infty$ in \eqref{geom-quantities}.
\begin{lemma}[perturbation error estimate for $C^2$ surfaces]\label{L:perturbation_bound_dist}
Let $u$ solve \eqref{e:weak_relax} and $u_\Gamma$ solve \eqref{Gamma:LBproblem}
with $\Gamma \subset \mathcal N$.
Let $\bchi_\Gamma$ and $\bchi := \bP_d \circ \bchi_\Gamma$ be the parametrizations
of $\Gamma$ and $\gamma$ that give rise to the area elements $q_\Gamma$ and $q$.
If the normal vectors satisfy  $\bnu\cdot\bnu_\Gamma \ge c >0$, then
\begin{equation}\label{perturbation_bound_dist}
  \|\nabla_\gamma(u-u_\Gamma)\|_{L_2(\gamma)} \lesssim
  \big(d_\infty K_\infty+\nu_\infty^2\big) \|f_\Gamma\|_{H^{-1}_\#(\Gamma)}
  + \|f q q_\Gamma^{-1}-f_\Gamma\|_{H^{-1}_\#(\Gamma)}.
\end{equation}
\end{lemma}
\begin{proof}
We proceed along the lines of Lemma \ref{L:perturbation_bound}
(perturbation error estimate for $C^{1,\alpha}$ surfaces) and realize that
Steps 1 and 3 are exactly the same. Therefore, we only deal with the estimate of
the geometric error matrix $\bE$. If we prove
\begin{equation}\label{est-E}
\|\bE \|_{L_\infty(\gamma)} \lesssim \nu_\infty^2 + d_\infty K_\infty\, ,
\end{equation}
then the assertion will readily follow. We first write
$\bE \circ \bP_d = \bI_1 + \bI_2 + \bI_3$ with
\begin{align*}
\bI_1 & := \Big(\frac{q_\Gamma}{q}-1\Big)
\Pi \, (\bI-d\bW) \, \Pi_\Gamma \, (\bI-d\bW) \, \Pi,
\\
\bI_2 & := \Big(\Pi \, (\bI-d\bW) \, \Pi_\Gamma \, (\bI-d\bW) \, \Pi
- \Pi \, \Pi_\Gamma \, \Pi\Big),
\\
\bI_3 & := \Big(\Pi \, \Pi_\Gamma \, \Pi - \Pi \Big).
\end{align*}
We now estimate these three terms separately. In view of \eqref{e:area_ratio_distance}
we deduce
\[
\frac{q(\by)}{q_\Gamma(\by)} -1 = \Big((\bnu(\bx)\cdot\bnu_\Gamma(\bx)-1) \, \prod_{i=1}^n \big(1-d(\bx)\kappa_i(\bx) \big) \Big)
+ \Big(\prod_{i=1}^n \big(1-d(\bx)\kappa_i(\bx) \big) - 1\Big),
\]
where $\bx = \bchi_\Gamma(\by)\in \Gamma$.
Since $1-\bnu\cdot\bnu_\Gamma = \frac12 |\bnu-\bnu_\Gamma|^2 \le \frac12 \nu_\infty^2$
and $\Gamma \subset \mathcal N$, we readily obtain
\begin{equation}\label{measure_error}
  \Big|\frac{q(\by)}{q_\Gamma(\by)}-1\Big| \lesssim \nu_\infty^2 + d_\infty K_\infty
  \quad\forall \, \by\in\mathcal{V},
\end{equation}
and a similar bound for $\frac{q_\Gamma}{q}$
because $\frac{q_\Gamma}{q}$ is bounded in $\mathcal{V}$ thanks to the
assumption $\bnu \cdot \bnu_\Gamma \geq c >0$.
The desired estimate for $\|\bI_1\|_{L_\infty(\gamma)}$ follows from the fact that
$\Pi, \Pi_\Gamma$ and $\bW$ are bounded. This property again, now combined with
\[
\bI_2 = - \Pi \, \Pi_\Gamma \, d\bW \, \Pi - \Pi \,d\bW \, \Pi_\Gamma \, \Pi
+ \Pi \,d\bW \, \Pi_\Gamma \,d\bW \, \Pi,
\]
yields $\|\bI_2\|_{L_\infty(\gamma)} \lesssim d_\infty K_\infty$. Finally, term $\bI_3$ reads
\[
\bI_3 = - \Pi\bnu_\Gamma \otimes \Pi\bnu_\Gamma =
-\big(\bnu_\Gamma-(\bnu \cdot \bnu_\Gamma) \bnu\big) \otimes
\big(\bnu_\Gamma- (\bnu \cdot \bnu_\Gamma) \bnu \big)
\]
Since $\bnu_\Gamma- (\bnu \cdot \bnu_\Gamma) \bnu = (\bnu_\Gamma-\bnu) + (1-\bnu\cdot\bnu_\Gamma) \bnu$ we infer that $\|\bI_3\|_{L_\infty(\gamma)} \lesssim \nu_\infty^2$.
This ends the proof.
\end{proof} 

It is worth comparing Lemmas \ref{L:perturbation_bound} and
\ref{L:perturbation_bound_dist} (perturbation error estimates). To do so, we next give an estimate for
$\nu_\infty$ in terms of $\lambda_\infty$.

\begin{lemma}[error estimate for normals]\label{L:est-normals}
The errors $\nu_\infty$ and $\lambda_\infty$ defined in \eqref{geom-quantities}
and \eqref{geo-est} satisfy
\begin{equation}\label{est-normals}
\nu_\infty \lesssim \lambda_\infty,
\end{equation}  
where the hidden constant depends on $S_\bchi$ defined in \eqref{stab-const}.
\end{lemma}
\begin{proof}
In view of the definition \eqref{unit-normal} of $\bnu$, we realize that
\[
\bnu-\bnu_\Gamma = \frac{\bN-\bN_\Gamma}{|\bN|}
+ \frac{|\bN_\Gamma|-|\bN|}{|\bN|} \frac{\bN_\Gamma}{|\bN_\Gamma|}
\quad\Rightarrow\quad
|\bnu-\bnu_\Gamma| \le 2 \frac{|\bN-\bN_\Gamma|}{|\bN|}.
\]
Since $\bN = \sum_{i=1}^{n+1} \det([\be_i,D\bchi]) \be_i$ and
$\det([\be_i,D\bchi]) - \det([\be_i,D\bchi_\Gamma])$ is a sum of products of
$\partial_j (\bchi-\bchi_\Gamma)\cdot\be_k$ with $k\ne i$ times $n-1$ factors
$\partial_\ell\bchi_m$, we have
\[
\big| \det([\be_i,D\bchi]) - \det([\be_i,D\bchi_\Gamma])  \big|
\lesssim |D(\bchi-\bchi_\Gamma)| \, |D\bchi|^{n-1}.
\]
We finally resort to $|\bN|=q$, proved in \eqref{q-N}, as well as $q\approx|D\bchi|^n$,
showed in the proof of Lemma \ref{L:error-est}, to conclude \eqref{est-normals}.
\end{proof}

We now stress the advantage of using the distance function lift $\bP_d$ to
represent the error $u-u_\gamma$ whenever the surface $\gamma$ is of class $C^2$.
Comparing \eqref{perturbation_bound} and
\eqref{perturbation_bound_dist} we see that the geometric error becomes of order $\|d\|_{L_\infty(\Gamma)}$ plus a quadratic term in $\lambda_\infty$ rather than linear. In the context of finite element methods, $\Gamma$ is often a polyhedral approximation to $\gamma$ having faces of diameter $h$. In this case $\|d\|_{L_\infty(\Gamma)}$ essentially becomes a Lagrange interpolation error measured in $L_\infty$ and $\lambda_\infty$ a Lagrange interpolation error measured in $W_\infty^1$.  The former error is of order $h^2$ and the latter of order $h$. Consequently, the perturbation error for $C^2$ surfaces is of order $h^2$, whereas for $C^{1,\alpha}$ surfaces with $\alpha<1$ it is of order $h^{\alpha}$ from the analysis of the previous subsection.  The increased approximation order for $C^2$ surfaces is a {\it superconvergence} effect.  We also recall from Theorem \ref{t:C1_implies_C2} ($C^1$ distance function implies $C^{1,1}$ surfaces) that the elegant properties of the distance function and closest point projection that lead to this superconvergence effect are not available when $\gamma$ is not of class $C^2$, thus the necessity of developing a separate perturbation theory for less regular surfaces as in the previous subsection.  It is not clear whether the order of the perturbation error actually jumps in this manner when crossing from $C^{1,\alpha}$ to $C^2$ surfaces, or if this jump is an artifact of proof.

\subsection{$H^2$ extensions from $C^2$ surfaces}
\label{S:apriori-C1a}

The analysis of the trace and narrow band methods that we carry out in later sections requires us to extend the solution $\wu\in H^2(\gamma)$
of \eqref{e:weak_relax} to tubular neighborhoods 
$
\mathcal N(\delta)
$
with the property
\begin{equation}\label{H2-extension}
\|u\|_{H^2({\mathcal N(\delta)})} \lesssim \delta^{\frac12} \|\wu\|_{H^2(\gamma)};
\end{equation}
we recall that $\Nd$ is defined in \eqref{e:delta-tube}.
The distance function lift $\bP_d$ provides a natural way to achieve this upon
setting $u=\wu\circ\bP_d$, namely
\[
u(\bx) = \wu \big(\bx - d(\bx) \nabla d(\bx) \big)
\quad\forall \, \bx\in{\mathcal N(\delta)}.
\]
However, this is problematic because it requires $\bP_d$ to be of class $C^2$,
and thus $\gamma$ of
class $C^3$, for \eqref{H2-extension} to hold. We now construct an extension 
that satisfies \eqref{H2-extension} for $\gamma$ of class $C^2$.   Our approach employs a regularization $\de$ of the signed distance function $d$ and construction of a regularized surface $\gae$ close to $\gamma$, with the regularization parameter $\varepsilon$ appropriately related to the desired value of $\delta$ above.  We begin by describing properties of this regularization.  

\medskip\noindent
{\bf Regularization.}
Recall that given $\gamma$ of class $C^{2}$ there exists a sufficiently thin tubular
neighborhood $\mathcal{N}$ so that the signed distance
function $d$ to $\gamma$ satisfies $d \in C^2(\mathcal{N})$.
Let $\delta>0$ and $\varepsilon = c \delta \le \frac{\delta}{2}$ be sufficiently small so that the
tubular neighborhood $\mathcal{N}(\delta)$ of width $\delta$ satisfies the property
\[
\mathcal{N}(\delta+2\varepsilon) \subset  \mathcal{N}.
\]
Let $\Be:=B(0,\varepsilon)$ be the ball of center $0$ and radius $\varepsilon$, $\re$ be a smooth and radially symmetric mollifier with support in
$\Be$
and
\[
\de(\bx) := d \star \re(\bx) = \int_\Be d(\bx-\by) \re(\by) d\by
\quad\forall \, \bx\in{\mathcal N(\delta)}
\]
be a regularized distance function. This function $\de$ induces the smooth
surface
\[
\gae := \big\{\bx\in\mathcal{N}: \quad \de(\bx) = 0 \big\},
\]
but is not the signed distance function to $\gae$, which we denote $\wde$.
The following properties are immediate from the previous definitions.

\begin{lemma}[properties of $\de$]\label{L:properties-de}
If $d\in C^2(\overline{\mathcal{N}})$, then
$\de$ satisfies
\[
\| d - \de \|_{L_\infty({\mathcal N(\delta)})}
+ \varepsilon \| \nabla(d - \de) \|_{L_\infty({\mathcal N(\delta)})}
\lesssim \varepsilon^{2} |d|_{W_\infty^2 (\overline{\mathcal{N}})}
\]
and $\|D^2 \de\|_{L_\infty({\mathcal N(\delta)})} \lesssim
 |d|_{W_\infty^2(\overline{\mathcal{N}})}$.
Moreover, the surface $\gae$ is smooth and the Hausdorff distance $\dist_H(\gamma,\gae)$
between $\gamma$ and $\gae$ satisfies
\[
\dist_H(\gamma,\gae) \lesssim \varepsilon^{2}
|d|_{W_\infty^2(\overline{\mathcal{N}})}
\]
provided $\varepsilon$ is small enough so that $C\varepsilon
|d|_{W_\infty^2(\overline{\mathcal{N}})} \le \frac12$ for a suitable constant $C$.
\end{lemma}
\begin{proof}
Since $\re$ is radially symmetric, we have that
\[
\big( d - \de  \big)(\bx) = \int_{\Be} \Big( d(\bx) - \nabla d(\bx) \cdot \by
- d(\bx-\by) \Big) \re(\by) d \by
\]
and
\[
\nabla \big( d - \de  \big)(\bx) =
\int_{\Be} \Big( \nabla d(\bx) - \nabla d(\bx-\by) \Big) \re(\by) d \by.
\]
These relationships imply the first assertion upon employing a Taylor expansion of $d$ and the Lipschitz nature of $\nabla d$, respectively. We also note that 
\[
D^2 \de(\bx) = \int_{\Be} \nabla d(\bx-\by) \otimes \nabla\re(\by) d\by
= \int_{\Be} \nabla \Big( d(\bx-\by) - d(\bx) \Big) \otimes \nabla\re(\by) d\by
\]
because $\int_{\Be} \nabla\re(\by) d\by = 0$ in view of the radial symmetry of $\re$. The second relationship bounding $D^2 \de$ then follows from the Lipschitz nature of $\nabla d$ (i.e. $|\nabla (d(\bx-\by)-d(\bx))| \lesssim \varepsilon$, $\by \in B_\varepsilon$) and the standard property $\|\nabla \rho_\varepsilon\|_{L_1(B_\varepsilon)} \lesssim \varepsilon^{-1}$ of the mollifier. 

To establish the smoothness of $\gamma_\varepsilon$, note that the closeness of $\nabla d$ and $\nabla \de$ implies that $\nabla \de$ is nondegenerate for $C \varepsilon |d|_{W_\infty^2(\overline{\mathcal{N}})} \le 1/2$.  The smoothness of $\gamma_\varepsilon$ then follows from the implicit function theorem.

The last assertion is a consequence of the nondegeneracy
of the distance function: Given $\by\in\gae$ let
$\bx = \bP_d(\by) =\by - d(\by) \nabla d(\by)\in\gamma$
be the closest point to $\by$ and note that
\[
|\by-\bx| = |d(\by)| = |d(\by)-\de(\by)| \lesssim \varepsilon^2 |d|_{W_\infty^2(\overline{\mathcal{N}})}. 
\]
Likewise, given $\bx\in\gamma$ let $\by(s) = \bx + s \nabla d(\bx)$.  There is $s \in (-\varepsilon, \varepsilon)$ such that $\de(\by(s))=0$.  To see this,
note that $d(\by(s))=\pm \varepsilon$ for $s=\pm\varepsilon$ and
\[
\de(\by(\varepsilon)) \ge d(\by(\varepsilon)) - C\varepsilon^{2} |d|_{W_\infty^2(\overline{\mathcal{N}})}
= \varepsilon \big( 1 - \varepsilon |d|_{W_\infty^2(\overline{\mathcal{N}})} \big) > 0
\]
provided $C\varepsilon |d|_{W_\infty^2(\overline{\mathcal{N}})}\le\frac12$; similarly
$\de(\by(-\varepsilon))<0$. Letting $\by=\by(s)$ be such that $\de(\by(s))=0$, we note that $\bx = \bP_d(\by)$, and so arguing as before we have that 
\[ |\by -\bx|= |d(\by)|=|d(\by)-\de(\by)| \lesssim \varepsilon^2 |d|_{W_\infty^2(\overline{\mathcal{N}})},
\]
which concludes the proof.
\end{proof}

We recall that 
$\wde$ is the signed distance function to the zero level set $\gae$ of $\de$.
Consider the $C^\infty$ lift
\begin{equation}\label{eps-lift}
\bPe(\bx) := \bx - \wde(\bx) \nabla \wde(\bx)
\quad\forall \, \bx\in{\mathcal N(\delta)}.
\end{equation}
It is natural and useful for later considerations to compare tubular neighborhoods
dictated by $d$ and $\wde$. Let
\[
{\mathcal{N}}_\varepsilon({\delta_\varepsilon}) :=
\{\bx\in\mathbb{R}^{n+1}:  |\wde(\bx)| \le {\delta_\varepsilon} \},
\]
where we choose $\delta_\varepsilon$ as follows depending
on $\delta$ and $\varepsilon$.
Given $\bx\in{\mathcal N(\delta)}$ let $\wx\in\gamma$ be the point at shortest distance,
whence $|\bx-\wx|\le\delta$, and let $\bxe\in\gae$ be a point such that
$|\wx-\bxe|\le C |d|_{W^2_\infty(\mathcal N)}\varepsilon^2$ which is guaranteed to exist
because ${\rm dist}_H(\gamma, \gamma_\varepsilon) \le C \varepsilon^2|d| \le \varepsilon$
in view of Lemma \ref{L:properties-de} (properties of $\de$). Therefore
\[
|\wde(\bx)| = \dist(\bx,\gae) \le |\bx-\bxe| \le |\bx-\wx| + |\wx-\bxe|
\le \delta + C|d|_{W^2_\infty(\mathcal N)}\varepsilon^2
\le \delta + \varepsilon =: {\delta_\varepsilon} ,
\]
provided $C\varepsilon|d|_{W^2_\infty(\mathcal N)}\le1$; note that
${\delta_\varepsilon} \le \frac32 \delta$. This implies
\begin{equation}\label{e:Ntilde}
  {\mathcal N(\delta)}\subset {\mathcal{N}}_\varepsilon\big( {\delta_\varepsilon} \big).
\end{equation}
Similarly, using again $C \varepsilon |d|_{W_\infty^2(\overline{\mathcal{N}})} \le 1$
in conjunction with Lemma \ref{L:properties-de} yields
\[
{\mathcal{N}}_\varepsilon\big( {\delta_\varepsilon} \big) \subset \mathcal{N}({\delta_\varepsilon}+\varepsilon)=\mathcal{N}(\delta + 2 \varepsilon) \subset \mathcal{N}.
\]

The next lemma and corollary study important properties of $\wde$ and $\bPe$,
in particular how derivatives degenerate with $\varepsilon$.

\begin{lemma}[properties of $\wde$]\label{L:properties-Pe}
The function $\wde\in C^\infty(\mathcal N(\delta))$ and satisfies
\[
\|\wde\|_{W^2_\infty({\mathcal N(\delta)})}
+ \varepsilon \|\wde\|_{W^3_\infty({\mathcal N(\delta)})}  \lesssim |d|_{W^2_\infty(\mathcal N)}.
\]
Moreover, the following error estimates hold
$$
\|\nabla(d-\wde)\|_{L_\infty(\mathcal N(\delta))}
  \lesssim \delta |d|_{W^2_\infty(\mathcal{N})},
\quad
\| 1 - \nabla d \cdot \nabla \wde \|_{L_\infty(\mathcal N(\delta))}
\lesssim \delta^2 |d|_{W^2_\infty(\mathcal{N})}^2.
$$
\end{lemma}
\begin{proof}
Since $\de(\bx)=0$ and $|\nabla\de(\bx)|\ge\frac12$
for all $\bx\in\gae$, fix $\bx_0\in\gae$ and a system of coordinates such that
$\bx = (\bx',x_{n+1})$ is a generic point and $\nabla\de(\bx_0)$ points in the
$(n+1)$-th coordinate direction. The Implicit Function Theorem guarantees the existence
of a ball $B$ in $\mathbb{R}^n$ centered at $\bx_0'$ and a $C^\infty$ function
$\psi:B\to\mathbb{R}$ such that
\[
\de(\bx',\psi(\bx')) = 0
\quad\forall \, \bx'\in B.
\]
In other words, $\gae$ is locally described in $B$ as a graph $x_{n+1}=\psi(\bx')$
for $\bx'\in B$. It is not difficult but tedious to see that
\begin{align*}
\|\psi\|_{W^2_\infty(B)} &\lesssim \|\de\|_{W^2_\infty({\Nd})}
\lesssim \|d\|_{W^2_\infty(\mathcal N)},
\\
\|\psi\|_{W^3_\infty(B)} &\lesssim \|\de\|_{W^3_\infty({\Nd})}
\lesssim \frac{1}{\varepsilon} \|d\|_{W^2_\infty(\mathcal N)},
\end{align*}
which translates into the first estimates for $\wde$
\[
\|\wde\|_{W^2_\infty(\Nd)} \lesssim \|d\|_{W^2_\infty(\mathcal N)},
\qquad
\|\wde\|_{W^3_\infty(\Nd)} \lesssim \frac{1}{\varepsilon}\|d\|_{W^2_\infty(\mathcal N)}.
\]

To prove the error estimates, let $\bx\in\Nd\subset\mathcal{N}$ and note that 
$$
\nabla \wde(\bx) = \nabla \wde (\by) = \frac{\nabla \de(\by)}{|\nabla \de(\by)|},
\qquad
\nabla d(\bx) = \nabla d(\bw)
$$
with $\by = \bx-\wde(\bx) \nabla \wde(\bx)\in\gae$ and
$\bw = \bx- d(\bx)\nabla d(\bx)\in\gamma$.
Hence, 
\[
|\bw-\by| \le |\bw-\bx| + |\by-\bx| \le \delta + {\delta_\varepsilon} \le \frac52\delta
\]
because of \eqref{e:Ntilde}. Since $|\nabla d(\by)|=1$, we now write
\[
\nabla d(\bx) - \nabla \wde(\bx) =
\nabla d(\bw) - \nabla d(\by)
+ \nabla d(\by) - \nabla \de(\by)
+ \frac{\nabla \de(\by)}{|\nabla \de(\by)|}
  \big(|\nabla \de(\by)| - |\nabla d(\by)|\big)
\]
and estimate pairs of terms on the right hand side separately. Since
$d\in W^2_\infty(\mathcal{N})$, we get
\[
\big| \nabla d(\bw) - \nabla d(\by)  \big| \le |\bw-\by| \, |d|_{W^2_\infty(\mathcal{N})}
\lesssim \delta |d|_{W^2_\infty(\mathcal{N})},
\]
and using Lemma \ref{L:properties-de} (properties of $\de$) we also obtain
\[
\big| |\nabla d(\by)| - |\nabla \de(\by)| \big| \le
\big| \nabla d(\by) - \nabla \de(\by) \big| \le \varepsilon |d|_{W^2_\infty(\mathcal{N})}
< \delta |d|_{W^2_\infty(\mathcal{N})},
\]
whence the first error estimate follows
\[
\big| \nabla d(\bx) - \nabla \wde(\bx)  \big| \lesssim
\delta |d|_{W^2_\infty(\mathcal{N})}
\quad\forall \, \bx \in \Nd.
\]
To show the desired second error estimate we observe that
$\big| 1 - \nabla d(\bx)\cdot\nabla\wde(\bx) \big|
= \frac12 \big| \nabla d(\bx) - \nabla\wde(\bx) \big|^2$. This concludes the proof.
\end{proof}

\begin{corollary}[property of $\bP_\varepsilon$]\label{c:properties-Pe}
The lift $\bPe$ belongs to $C^\infty(\mathcal N(\delta))$ and satisfies
$$
|\bPe|_{W^2_\infty({\mathcal N(\delta)})} \lesssim |d|_{W^2_\infty(\mathcal N)}
$$
for suitable constants $C_1,C_2$ so that $C_1\delta \le \varepsilon \le \frac \delta 2$
and $C_2\varepsilon|d|_{W^2_\infty(\mathcal N)}\le 1$.
  \end{corollary}
  \begin{proof}
Differentiate the $k$-th component of $\bPe$ with respect to $x_i$ and $x_j$
to obtain
\[
\partial_{ij} \bP_{\varepsilon,k} = - \partial_{ij}^2 \wde \, \partial_k \wde
- \partial_i \wde \, \partial_{jk}^2 \wde - \partial_j \wde \, \partial_{ik}^2 \wde -  \wde \, \partial_{ijk}^3 \wde,
\]
whence invoking Lemma~\ref{L:properties-Pe} (properties of $\wde$) yields
\[
\|D^2 \bPe\|_{L_\infty(\Nd)} \lesssim
|d|_{W^2_\infty(\mathcal N)} + \frac{\delta}{\varepsilon} |d|_{W^2_\infty(\mathcal N)}\lesssim
|d|_{W^2_\infty(\mathcal N)}
\]
because of $|\nabla\wde|=1$ and \eqref{e:Ntilde}.
This completes the proof.
  \end{proof}
  
Given a function $\wu\in H^2(\gamma)$ we are now ready to introduce an
$H^2$ extension to ${\mathcal N(\delta)}$. 
For this, we assume that $\delta$ is sufficiently small so that
\eqref{e:Ntilde} is valid.
We first define the auxiliary function
$\ue = \wu\circ\bQe:\gae\to\mathbb{R}$, where
$\bQe=\bPe^{-1}:\gae\to\gamma$, and then the
extension $u = \ue\circ\bPe:{\mathcal N(\delta)}\to\mathbb{R}$, namely
\begin{equation}\label{e:non_const_ext}
  u(\bx) := \ue \big( \bx - \wde(\bx) \nabla \wde(\bx)   \big)
  \quad\forall \, \bx\in{\mathcal N(\delta)}.
\end{equation}
Consequently, we realize that $u=\wu\circ\bQe\circ\bPe$.
We introduce the notation $\bQe$ to avoid
confusion between $\bQe\circ\bPe:{\mathcal N(\delta)}\to\gamma$ and the identity.
We recall that the {\it coarea
formula}
\begin{equation}\label{e:coarea}
\int_{\mathcal N(\delta)} g  = \int_{\mathcal N(\delta)} g | \nabla d| =  \int_{-\delta}^{\delta} \int_{ \{ d^{-1}(s) \}} g d\sigma_s,
\end{equation}
is valid for any integrable function $g: \mathcal N(\delta) \rightarrow \mathbb R$
[Theorem 3.14, Evans and Gariepy]. We will use this formula next and later in
this chapter.
 
\begin{proposition}[$H^2$ extension]\label{P:H2-extension}
Let $\varepsilon$ and $\delta$ be as in Corollary \ref{c:properties-Pe}
(property of $\bPe$), and assume that $\varepsilon |d|_{W_\infty^2(\mathcal{N})}\le c$
for a sufficiently small constant $c$.  
If $\wu\in H^2(\gamma)$, then $u\in H^2({\mathcal N(\delta)})$ and 
\[
\|u\|_{H^2({\mathcal N(\delta)})} \lesssim \delta^{\frac12} |d|_{W^2_\infty(\mathcal N)} \|\wu\|_{H^2(\gamma)}.
\]
Moreover, the trace of $u$ on $\gamma$ coincides with $\wu$, that is $u$ an $H^2$
extension of $\wu$.
\end{proposition}
\begin{proof}
In view of \eqref{e:rel_dist_grad}, the $i$-partial derivative of $u$ reads
\[
\partial_i u = \sum_{j=1}^{n+1} \big( \delta_{ij} - \wde \, \partial_{ij}^2\wde \big)
\, \overline{\partial}_j\ue\circ\bPe
\]
where $\overline{\partial}_j\ue$ stands for the $j$-component of $\nabla_\gae\ue$.
We use again \eqref{e:rel_dist_grad} to obtain
\begin{align*}
\nabla \partial_j u = &- \sum_{j=1}^{n+1} \big( \nabla\wde \, \partial_{ij}^2\wde
+ \wde  \, \partial_{ij}^2\nabla\wde \big) \, \overline{\partial}_j\ue\circ\bPe
\\ &+ \sum_{j=1}^{n+1} \big( \delta_{ij} - \wde \, \partial_{ij}^2\wde  \big)
\, \big( \bI - \wde \, D^2 \wde \big) \nabla_\gae \overline{\partial}_j\ue\ \circ \bPe.
\end{align*}
Setting $\Lambda := 1 + |d|_{W^2_\infty(\mathcal N)}$ and
applying Lemma \ref{L:properties-Pe} (properties of $\wde$) yields
\[
\big| D^2 u \big| \lesssim \Lambda
\big( |\nabla_\gae\ue\circ\bPe| + |\nabla_\gae^2 \ue\circ\bPe| \big).
\]
We reduce the computation of integrals in the bulk ${\mathcal N(\delta)}$ to integrals on
parallel surfaces $\gae(s) := \{\bx\in\mathbb{R}^{n+1}: \wde(\bx) = s\}$ via the
coarea formula \eqref{e:coarea}. 
Since $|\nabla\wde|=1$ in view of \eqref{e:Ntilde} the co-area formula implies
\begin{align*}
\int_{\mathcal N(\delta)} |D^2 u(\bx)|^2 d\bx & \lesssim \Lambda^2 \int_{\mathcal N(\delta)}
\sum_{k=1}^2|\nabla_\gae^k\ue(\bPe(\bx))|^2 \, |\nabla\wde(\bx)| \, d\bx
\\ & \le  \Lambda^2 \int_{-{\delta_\varepsilon}}^{{\delta_\varepsilon}} \int_{\gae(s)}
\sum_{k=1}^2|\nabla_\gae^k\ue(\bPe(\bx))|^2 \, d \sigma_{\varepsilon,s}(\bx) \, ds
\\& \lesssim \delta  \Lambda^2 \int_\gae \sum_{k=1}^2|\nabla_\gae^k\ue(\bx)|^2
\, d\sigma_\varepsilon(\bx),
\end{align*}
Lemma \ref{L:norm-equiv} (norm equivalence) immediately yields
\[ \int_{\gamma_\varepsilon} |\nabla_{\gamma_\varepsilon} u_\varepsilon(\bP_\varepsilon(\bx))|^2 \, d\sigma_\varepsilon(\bx) \lesssim \int_\gamma |\nabla_\gamma \wu(\bx)|^2 d \sigma(\bx).
\]
In order to relate second derivatives of $u_\varepsilon$ on $\gamma_\varepsilon$ to those of $\wu$ on $\gamma$, we apply \eqref{e:tang_discrete_to_exact} with $\gamma_\varepsilon$ playing the role of $\gamma$ and $\Gamma=\gamma$.  Then
\[
\nabla_{\gamma_\varepsilon} u_\varepsilon (\bP_\varepsilon(\bx)) = ({\bf I} -\wde \bW_\varepsilon)^{-1}(\bx) \left(\bI - \frac{\bnu_\gamma(\bx) \otimes \bnu_\varepsilon(\bx)}{\bnu_\gamma(\bx) \cdot \bnu_\varepsilon(\bx)}\right) \nabla_\gamma \wu(\bx) \quad \bx \in \gamma,
\]
and after applying this formula again to $\nabla_{\gamma_\varepsilon} u_\varepsilon (\bP_\varepsilon(\bx))$ we obtain
\[
|D_{\gamma_\varepsilon}^2 u_\varepsilon(\bP_\varepsilon(\bx))| \le |D_\gamma {\bf M}(\bx)| \, |\nabla_\gamma \wu(\bx)|+ |{\bf M}(\bx)| \, |D_\gamma^2 \wu(\bx)|,
\]
where ${\bf M}(\bx)=({\bf I} -\wde \bW_\varepsilon)^{-1}(\bx) \left(\bI - \frac{\bnu_\gamma(\bx) \otimes \bnu_\varepsilon(\bx)}{\bnu_\gamma(\bx) \cdot \bnu_\varepsilon(\bx)}\right)$.
We thus wish to bound $\|M\|_{W_\infty^1(\gamma)}$.  First we note that combining the bound on the Hausdorff distance between $\gamma$ and $\gamma_\varepsilon$ from Lemma \ref{L:properties-de} (properties of $\de$) with $\|\wde\|_{W_\infty^2(\mathcal{N}(\delta))} \lesssim |d|_{W_\infty^2(\mathcal{N})}$ from Lemma \ref{L:properties-Pe} (properties of $\wde$) yields for $\bx \in \gamma$ that the eigenvalues of $\wde(\bx) \bW_\varepsilon(\bx)$  are bounded by $C \varepsilon^2|d|_{W_\infty^2(\overline{\mathcal{N}})}^2$, which is less than $\frac{1}{2}$ under the assumption that $\varepsilon |d|_{W_\infty^2(\mathcal{N})}$ is sufficiently small; thus $\|(\bI-\wde \bW_\varepsilon)^{-1}\|_{L_\infty(\Nd)} \le 2$. In addition, combining the same assumption with $\varepsilon \simeq \delta$ and Lemma \ref{L:properties-Pe} yields
\[ 
\|1-\bnu_\gamma\cdot \bnu_\varepsilon\|_{L_\infty(\mathcal{N}(\delta))} \lesssim \delta^2|d|_{W_\infty^2(\mathcal{N})}^2 \lesssim \varepsilon^2 |d|_{W_\infty^2(\mathcal{N})}^2 \le \frac{1}{2},
\]
so that $\bnu_\gamma\cdot \bnu_\varepsilon \ge 1/2$ and
 \[
 \left \|\bI - \frac{\bnu_\gamma \otimes \bnu_\varepsilon}{\bnu_\gamma \cdot \bnu_\varepsilon}\right\|_{L_\infty(\Nd)} \lesssim 1 ;
 \]
thus $\|{\bf M}\|_{L_\infty(\Nd)} \lesssim 1$. In order to bound the derivatives of ${\bf M}$, we note that for a matrix ${\bf A}$ there holds  $\partial_i {\bf A}^{-1} = -{\bf A}^{-1}( \partial_i {\bf A}) {\bf A}^{-1}$.  For ${\bf A}=\bI-\wde \bW_\varepsilon$, we use Lemmas \ref{L:properties-de} and \ref{L:properties-Pe}, $|\nabla \wde|=1$, and the assumption $C_1\delta\le\varepsilon$ to deduce in $\Nd$
 \[ 
 \begin{aligned}
 |\partial_i {\bf A}|& = \big|(\partial_i \wde) \bW_\varepsilon + \wde \,
 \partial_i \bW_\varepsilon\big| 
 \\ & \lesssim \|\wde\|_{W_\infty^2(\mathcal{N}(\delta)} + \delta \|\wde\|_{W_\infty^3(\mathcal{N})} \lesssim |d|_{W_\infty^2(\mathcal{N})}.
 \end{aligned}
 \]
 Since we have already established that $\|\bf A^{-1}\|_{L_\infty(\Nd)}\lesssim 1$, we infer that $|(\bI-\wde \bW_\varepsilon)^{-1}|_{W_\infty^1(\Nd} \lesssim |d|_{W_\infty^2(\mathcal{N})}$.   A similar calculation for $\bI - \frac{\bnu_\gamma\otimes \bnu_\varepsilon}{\bnu_\gamma \cdot \bnu_\varepsilon}$, while recalling that $\bnu_\gamma \cdot \bnu_\varepsilon \ge 1/2$, yields  $|{\bf M}|_{W_\infty^1(\gamma)} \lesssim |d|_{W_\infty^2(\mathcal{N})}$ and, after applying Lemma \ref{L:norm-equiv} (norm equivalence), gives
\[ \|D_{\gamma_\varepsilon}^2 u_\varepsilon\|_{L_2(\gamma_\varepsilon)} \lesssim |d|_{W_\infty^2(\mathcal{N})} \left(\|\nabla_\gamma \wu\|_{L_2(\gamma)} + \|D_\gamma^2 \wu\|_{L_2(\gamma)}\right). \]

The asserted estimate follows from applying again the 
co-area formula \eqref{e:coarea}, which leads to
\[
\int_{\mathcal N(\delta)} |u|^2+|\nabla u|^2+|D^2u|^2\lesssim \delta \Lambda^2
\int_\gamma |\wu|^2 + |\nabla_\gamma \wu|^2 + |D^2\wu|^2.
\]
Finally, we take $\bx\in\gamma$, note that $\bQe(\bPe(\bx))=\bx$, and compute
\[
u(\bx) =  \wu\circ\bQe\circ\bPe(\bx) = \wu(\bx)
\]  
to realize that $u$ is indeed an extension of $\wu$ to ${\mathcal N(\delta)}$.
\end{proof}
  
We now derive the elliptic PDE's satisfied by $\ue$ on $\gae$ and $u$ in ${\mathcal N(\delta)}$.
For $\wu\in H^2(\gamma)$, let
$\wf = -\Delta_\gamma \wu \in L_{2,\#}(\gamma)$ and consider the extension $\wfe$ to
$\gae$
\[
\wfe := \wf \circ \bQe.
\]

\begin{lemma}[PDE satisfied by $\ue$]\label{L:PDE-ue}
If $\gamma$ is closed and of class $C^2$, then $\gae$ is also closed and of class
$C^\infty$, and the extension $\ue = \wu\circ\bQe$ satisfies on $\gae$
\[  
- \wmue \mathrm{div}_\gae \Big(\frac{1}{\wmue} \wbAe \nabla_\gae \ue \Big) = \wfe,
\]
where $\wbAe := \big(\bI - \wde \, D^2\wde \big) \Pi
\big(\bI - \wde \, D^2\wde \big)\circ\bQe$,
$\Pi$ stands for the orthogonal  projection $\Pi = (\bI - \nabla d\otimes\nabla d)$
on $\gamma$ and $\wmue := \frac{\qe}{q\circ\bQe}$ reads
\[
\wmue = \det\Big(\bI - \wde \, D^2\wde  \Big)
\big( \nabla d \cdot \nabla\wde  \big) \circ\bQe.
\]
\end{lemma}
\begin{proof}
Given $\wv\in H^1(\gamma)$, let $\tv = \ttv\circ\bQe \in H^1(\gae)$.
We resort to \eqref{e:tang_exact_to_discrete} to write
\[
\nabla_\gamma \wu = \Pi \big(\bI - \wde \, D^2 \wde  \big)
\nabla_\gae \ue \circ \bPe \quad\textrm{on }\gamma,
\]
because $\nabla_\gae \ue \circ \bPe = \Pi_\varepsilon \nabla_\gae \ue \circ \bPe$.
This combined with \eqref{int-gamma} and Corollary \ref{C:int-parts}
(integration by parts) on the closed surface $\gae$ yields
\[
\int_\gamma \nabla_\gamma \wu \cdot \nabla_\gamma \ttv
= \int_\gae \frac{1}{\wmue} \wbAe \nabla_\gae \ue \cdot \nabla_\gae \tv
= -  \int_\gae \mathrm{div}_\gae \Big( \frac{1}{\wmue} \wbAe \nabla_\gae \ue \Big) \tv
\]
with $\wmue=\frac{\qe}{q\circ\bQe}$ given by \eqref{e:area_ratio_distance}.
Likewise,
\[
\int_\gamma \wf \, \ttv = \int_\gae \frac{1}{\wmue}\wfe \tv .
\]
Since the last two equalities hold for all $\tv\in H^1(\gae)$, the assertion
follows.
\end{proof}

We extend the function $\wfe$ to ${\mathcal{N}}_\varepsilon({\delta_\varepsilon})$ as follows:
\[
\fe := \wfe \circ \bPe = \wf \circ \bQe \circ \bPe.
\]
Equivalently, given $\bx\in{\mathcal{N}}_\varepsilon({\delta_\varepsilon})$
let $\wx\in\gamma$ be the unique point such that for some $s$
\[
\wx = \bx + s \nabla\wde(\bx)
\quad\Rightarrow\quad
\fe(\bx) = \wf(\wx).
\]

\begin{proposition}[PDE satisfied by $u$]\label{P:BVP}
Let $\varepsilon$ and $\delta$ be as in Corollary \ref{c:properties-Pe}
(property of $\bPe$).
The extension $u\in H^2({\mathcal N(\delta)})$ of $\wu$ of
Proposition \ref{P:H2-extension} satisfies the PDE
\[
-\frac{1}{\mue} \div { \mue \bBe \nabla u}  = \fe
\qquad \hbox{ in }\qquad  {\mathcal N(\delta)},
\]
where
\[
\bBe := \big(\bI - \wde \, D^2\wde \big)^{-1} \Pi_\varepsilon \bAe \Pi_\varepsilon
\big( \bI - \wde \, D^2\wde \big)^{-1} ,
\]
$\bAe:=\wbAe\circ\bPe$ with $\wbAe$ given in Lemma \ref{L:PDE-ue},
$\Pi_\varepsilon = \bI - \nabla\wde\otimes\nabla\wde$, $\mue$ is given by
\[
\mue := \frac{1}{\wmue\circ\bPe} \det\Big( \bI - \wde \, D^2\wde  \Big),
\]
and $\wmue$ is defined in Lemma \ref{L:PDE-ue}.
\end{proposition}
\begin{proof}
We proceed as in Proposition \ref{P:H2-extension} ($H^2$ extension).
Let $\gamma_\varepsilon(s)$ be a parallel surface to $\gae$ at distance $s$, and let
$|s| \le {\delta_\varepsilon}$ with ${\delta_\varepsilon} = \frac{3}{2}\delta$  so that \eqref{e:Ntilde} holds.
We first employ \eqref{e:tang_discrete_to_exact} to obtain the bilinear form
for $u$ on $\gamma_\varepsilon(s)$. For
$\delta$ sufficiently small Lemma~\ref{L:properties-Pe} (properties of $\wde$)
guarantees that $\big(\bI - \wde \, D^2\wde \big)$ is invertible in
${\mathcal{N}}_\varepsilon(\delta_\varepsilon)$.
Hence, if $\bDe = \big(\bI - \wde \, D^2\wde \big)^{-1} \Pi_\varepsilon$
and $\tv \in C_0^\infty({\mathcal N(\delta)})$, we restrict $\tv$ to $\gamma_\varepsilon(s)$,
define the auxiliary function
$\ttv := \tv \circ \bPe^{-1}\in C^\infty(\gae)$ and observe that
\eqref{e:tang_discrete_to_exact} reads on $\gamma_\varepsilon(s)$
\[
\nabla_\gae \ttv \circ \bPe = \bDe \nabla \tv, 
\]
where $\nabla \tv$ is the full gradient of $\tv$; this is because of the presence
of the projection matrix $\Pi_\varepsilon$ on the tangent hyperplane to $\gamma_\varepsilon(s)$
in the definition of $\bDe$. We get
\begin{align*}
\int_\gae \frac{1}{\wmue} \wbAe \nabla_\gae\ue \cdot \nabla_\gae \ttv
= \int_{\gamma_\varepsilon(s)} \mue \bAe \bDe
\nabla u \cdot \bDe \nabla \tv
\end{align*}
where $\wmue$ is given in Lemma \ref{L:PDE-ue} (PDE satisfied by $\ue$) and
$\mue$ is the surface measure density on $\gae(s)$ due to the change of variables, namely
\[
\mue = \frac{1}{\wmue\circ\bPe} \frac{\qe}{q_{\varepsilon,s}}
= \frac{1}{\wmue\circ\bPe} \det\Big( \bI - \wde \, D^2\wde  \Big)
\]
according to \eqref{e:area_ratio_distance}.
Similarly, the linear form for the forcing reads
\[
\int_\gae \frac{1}{\wmue} \, \wfe \, \ttv = \int_{\gamma_\varepsilon(s)} \mue \, \fe \, \tv.
\]
Since the left hand sides of the previous integral expressions coincide,
in view of Lemma \ref{L:PDE-ue},
we now integrate over
$s\in(-{\delta_\varepsilon},{\delta_\varepsilon})$ and use the co-area
formula \eqref{e:coarea} to convert the resulting integrals into bulk integrals
\begin{align*}
  \int_{{\mathcal{N}}_\varepsilon({\delta_\varepsilon})}
  \mue \bAe \bDe \nabla u \cdot \bDe \nabla \tv
  & = \int_{{\mathcal{N}}_\varepsilon({\delta_\varepsilon})}
  \mue \bAe \bDe \nabla u \cdot \bDe \nabla \tv
\, |\nabla\wde|
\\
& =
\int_{-{\delta_\varepsilon}}^{{\delta_\varepsilon}}\int_{\gamma_\varepsilon(s)}
\mue \bAe \bDe \nabla u \cdot \bDe \nabla \tv \, d\sigma_{\varepsilon,s} \, ds
\\
& = \int_{-{\delta_\varepsilon}}^{{\delta_\varepsilon}}\int_{\gamma_\varepsilon(s)}
\fe \tv \, \mue \, d\sigma_{\varepsilon,s} \, ds
\\
& = \int_{{\mathcal{N}}_\varepsilon({\delta_\varepsilon})} \fe \tv \, \mue \, |\nabla\wde|
= \int_{{\mathcal{N}}_\varepsilon({\delta_\varepsilon})} \fe \tv \, \mue,
\end{align*}
because $|\nabla \wde|=1$ in ${\mathcal{N}}_\varepsilon({\delta_\varepsilon})$.
Since ${\mathcal N(\delta)}\subset {\mathcal{N}}_\varepsilon({\delta_\varepsilon})$ according to \eqref{e:Ntilde}, integration by parts gives
\[
- \int_{\mathcal N(\delta)} \mathrm{div} \big(\mue \bDe \bAe \bDe \nabla u \big) \tv =
\int_{\mathcal N(\delta)} \fe \, \tv \mue
\quad\forall \, v\in C_0^\infty({\mathcal N(\delta)}),
\]
whence the desired PDE follows after noticing that
$\big(\bI - \wde \, D^2\wde \big)^{-1}$ and  $\Pi_\varepsilon$ commute.
This completes the proof.
\end{proof}



\section{Parametric Finite Element Method}\label{sec:parametric}

The parametric method hinges on a surface approximation $\Gamma$
``interpolating'' the exact surface $\gamma$.
Recall that the latter is assumed to be a closed, compact, orientable hypersurface in $\mathbb R^{n+1}$.
In the lowest order case of piecewise linear polynomials, this corresponds to a polyhedral surface $\Gamma$ whose vertices lie on $\gamma$ or, more generally, sufficiently close to $\gamma$.  The finite element space is then obtained in a classical way by mapping a finite element triplet defined on a reference element in $\mathbb R^n$ to a facet of $\Gamma$ in $\mathbb R^{n+1}$. The FEM requires a bi-Lipschitz map $\bP:\Gamma\to\gamma$ which is not necessarily the distance function lift $\bP_d$. The latter is used for numerical analysis purposes only even for smooth surfaces.

There are two sources of error: the approximation of the exact surface $\gamma$ by the polyhedral surface $\Gamma$, the so-called \emph{geometric consistency error}, and the \emph{Galerkin error} arising from the actual finite element approximation on $\Gamma$. In this section we quantify these two errors depending on the regularity of $\gamma$.
For the former we rely on the discussion of section \ref{S:perturbation} that addresses
the effect of perturbing $\gamma$.
For $\gamma$ of class $C^{1,\alpha}$ we deal with
a generic lift $\bP:\Gamma\to\gamma$ and
obtain a suboptimal geometric consistency error.
For $C^2$ surfaces, instead, we resort to $\bP_d$
for error analysis and restore geometric optimality.

\subsection{FEM on Lipschitz Parametric Surfaces.}

\noindent
{\bf Lipschitz Parametric Surfaces.}
We adopt the viewpoint that the surface $\gamma$ is described as the deformation of an
$n$-dimensional polyhedral surface $\lingamma$ by a globally bi-Lipschitz
\emph{homeomorphism} $\bP:\lingamma \rightarrow \gamma \subset \mathbb
R^{n+1}$. 
Thus there exists $L>0$ such that for all $\bx_1, \bx_2 \in \lingamma$
\begin{equation}
\label{surf_bi_lipschitz}
L^{-1} |\bx_1-\bx_2| \le |\widetilde{\bx}_1 -\widetilde{\bx}_2| \le L|\bx_1-\bx_2|, \qquad
\widetilde{\bx}_i = \bP(\bx_i), ~i=1,2.
\end{equation}
If $\gamma$ is $C^2$, we may take $\bP=\bP_d$, but our current definition allows for much more flexibility in the choice of $\bP$.  For example, if $\gamma$ has nonempty boundary and is given as the graph of a function $\psi:\Omega \rightarrow \mathbb{R}^{n+1}$ with $\Omega \subset \mathbb{R}^n$, then the map between $\bx=(x,z) \in \Gamma$ with $x \in \Omega$ and $z \in \mathbb{R}$ could be given by $\bP(x,z)=(x, \psi(x))\in\gamma$, i.e., the ``vertical'' graph map.  

The (closed) facets of $\lingamma$ are denoted by $T$, and form the collection $\T=\{T \}$.   We assume that these facets are all simplices and denote by $S_\T$ the set of interior faces of $\T$.  Extension to other element shapes such as $n$-quadrilaterals and to nonconforming discretizations is possible under reasonable assumptions with minor modifications, but we do not elaborate them further.  
We let $\bP_T:T \rightarrow \mathbb R^{n+1}$ be
the restriction of $\bP$ to $T$.
The partition $\T$ of $\lingamma$ induces the partition
$\widetilde{\T}=\{\widetilde{T}\}_{T \in \T}$ of $\gamma$ upon setting
\[
\widetilde{T} := \bP_T(T)
\quad\forall \, T\in \T.
\]
Note that this \emph{non-overlapping} parametrization of $\gamma$ allows for Lipschitz surfaces rougher than globally $C^2$.  We additionally define {\it macro patches} 
\begin{equation}\label{patch}
\linpatch_T=\cup \big\{T': T' \cap T \neq \emptyset \big\},
\qquad
\widetilde{\omega}_T =  \bP(\linpatch_T).
\end{equation}
Let $h_T := | T |^{\frac 1 n}$ and
$\sigma < \infty$ be the triangulation shape-regularity constant, i.e.
\begin{equation}\label{e:shape_reg}
\sigma:=\sup_{\T} \max_{T \in \T} \frac{{\rm diam}(T)}{h_T}.
\end{equation}
We further assume that the number of elements in each patch $\omega_T$ is uniformly bounded.  This assumption automatically follows from shape regularity for triangulations of Euclidean domains, but the situation is more subtle for surface triangulations as illustrated in Figure~\ref{f:valence}.  Such a bound does for example hold if $\lingamma$ is obtained by systematic refinement of an initial surface mesh with a uniform bound on the number of elements in a patch \cite{DemlowDziuk:07}, or more generally using adaptive refinement strategies \cite{BCMMN16,BCMN:Magenes}.  In addition, this implies that all elements in $\linpatch_T$ have uniformly equivalent diameters, as it happens for shape regular triangulations on Euclidean domains.

\begin{figure}[ht!]
\centerline{\includegraphics[width=0.5\textwidth]{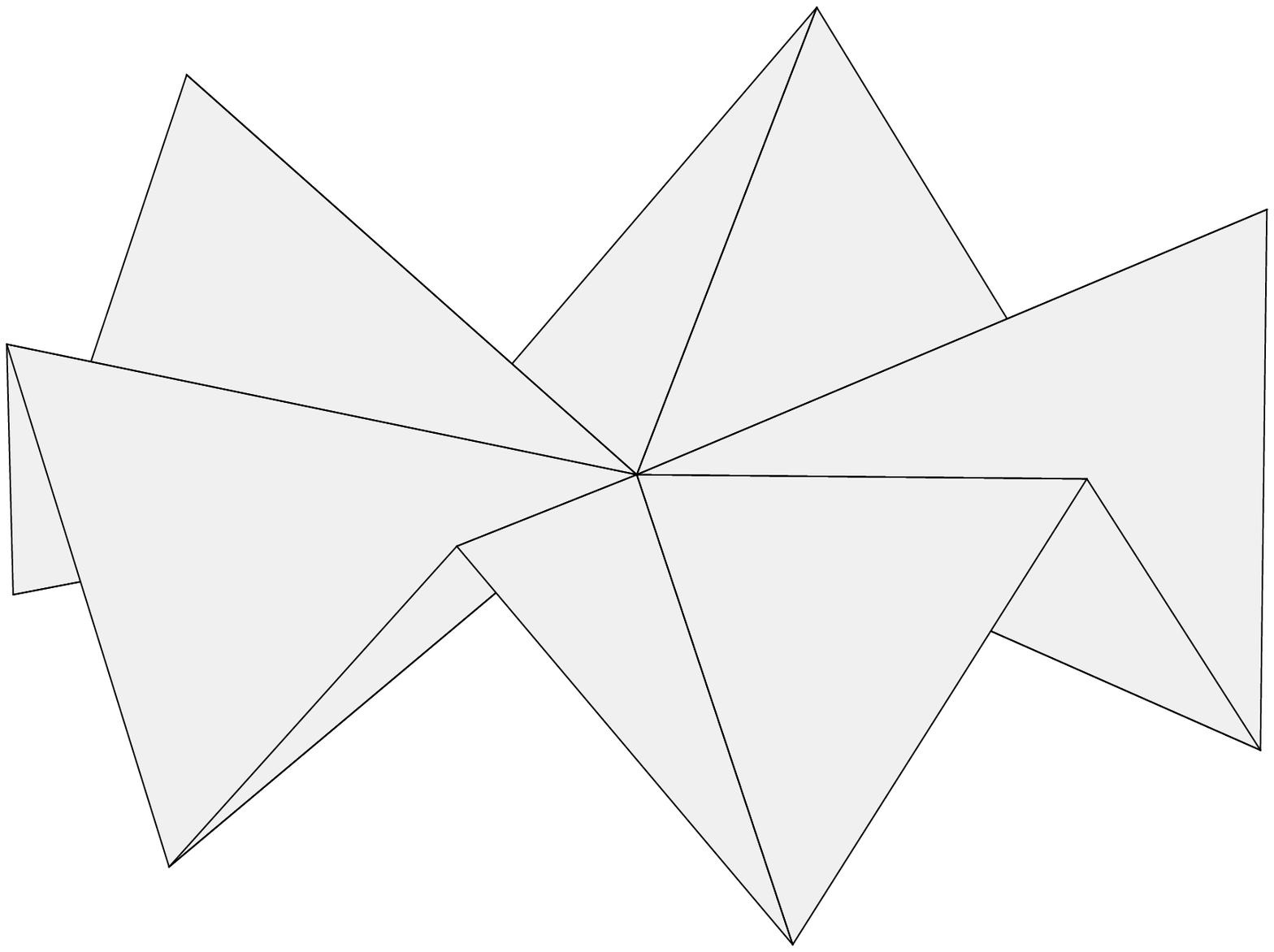}
\includegraphics[width=0.5\textwidth]{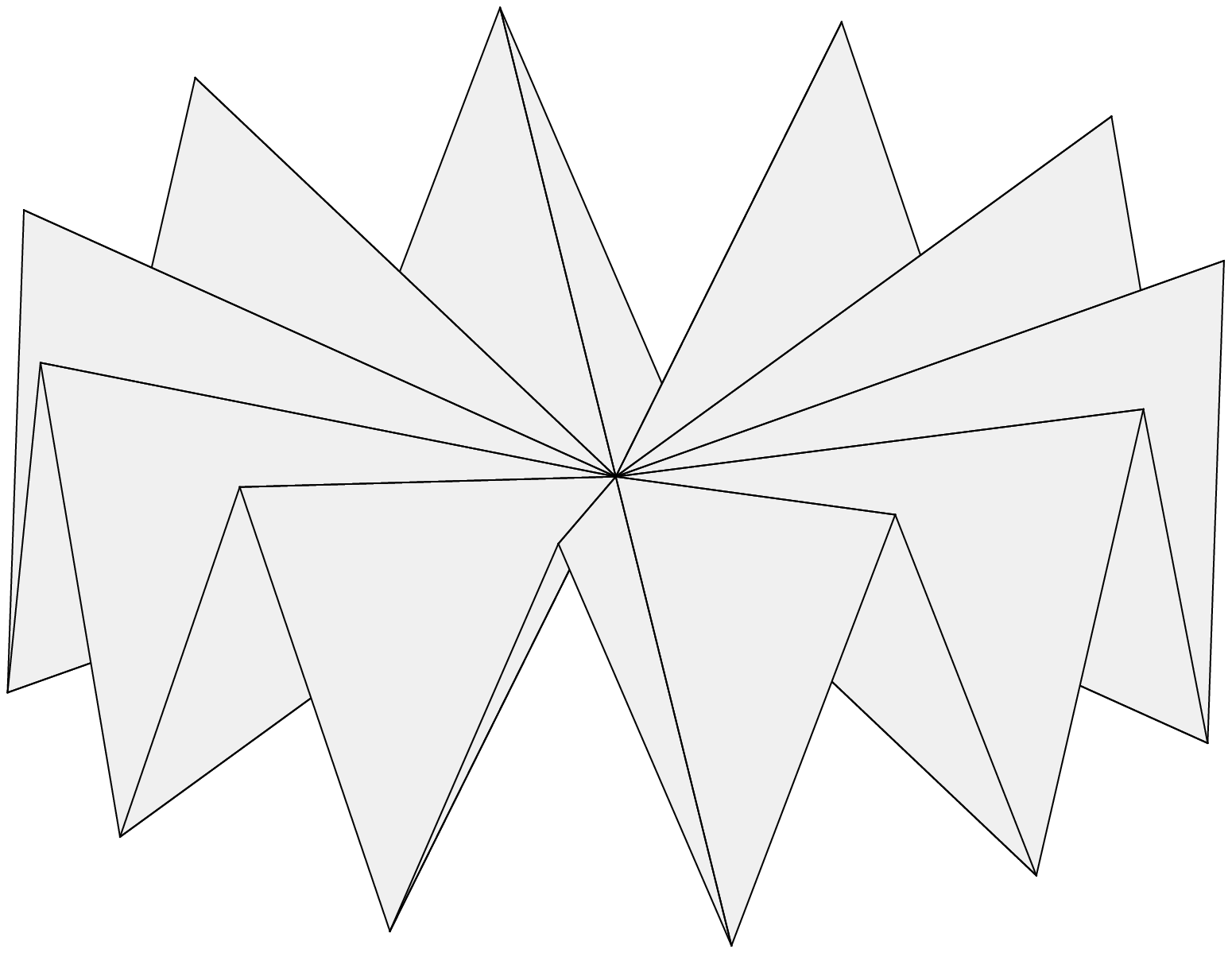}}
\caption{
Two different configurations when $n=2$ illustrating that the number of elements sharing the same vertex could be arbitrarily large even when using triangles satisfying \eqref{e:shape_reg_init}.} \label{f:valence}
\end{figure}

To provide a parametric description,
let $\widehat{T}$ be the unit reference simplex, sometimes called the {\it universal parametric domain}.
We denote by $\linparam_T:\mathbb R^n\to\mathbb R^{n+1}$ the affine map such that $T = \linparam_T(\widehat{T})$ and note that \eqref{e:shape_reg} implies
\begin{equation}\label{e:shape_reg_init}
h_T | \bw|  \lesssim | D \linparam_T \bw | \lesssim h_T | \bw|, \qquad \forall \bw \in \mathbb R^n. 
\end{equation}
Hereafter we omit to specify when the constants (possibly hidden in $\lesssim$ signs) depend on $\sigma$.
As pointed out in \cite{BP:11}, even if the initial surface approximation satisfies \eqref{e:shape_reg_init}, this property might not hold for refinements unless the initial polyhedral surface approximates the exact surface well. 
We refer to \cite{BCMMN16,BCMN:Magenes} for a discussion on how to circumvent this in an adaptive strategy.
However, since this work focusses on a-priori and a-posteriori error estimation rather that adaptivity, we assume \eqref{e:shape_reg_init} directly. 

We are now ready to introduce the local {\it non-overlapping} parametrization $\bchi$
of $\gamma$.
Let $\exactparam_T := \bP \circ\linparam_T: \widehat{T} \rightarrow \widetilde{T}$ be the corresponding local parametrization of $\widetilde{T}$ and $\bchi:= \{ \bchi_T \}_{T\in \T}$; see Figure~\ref{f:param}.
We record for latter use that thanks to the Lipschitz properties \eqref{surf_bi_lipschitz} and \eqref{e:shape_reg_init}, $\exactparam_T$ also satisfies 
\begin{equation}\label{e:shape_reg_exact}
  h_T | \bw|  \lesssim | D \exactparam_T(\by) \bw | \lesssim h_T | \bw |
 \quad \forall \bw \in \mathbb R^n, ~ \by\in\widehat{T}.
\end{equation}
\begin{figure}[ht!]
\begin{picture}(0,0)
\put(-75,-140){$\widehat T$}
\put(10,-120){$\linparam_T$}
\put(30,-65){$T\subset \Gamma$}
\put(15,-50){$\bP$}
\put(30,5){$\widetilde T \subset \gamma$}
\put(-40,-75){$\exactparam_T$}
\end{picture}
\centerline{\includegraphics[width=0.5\textwidth]{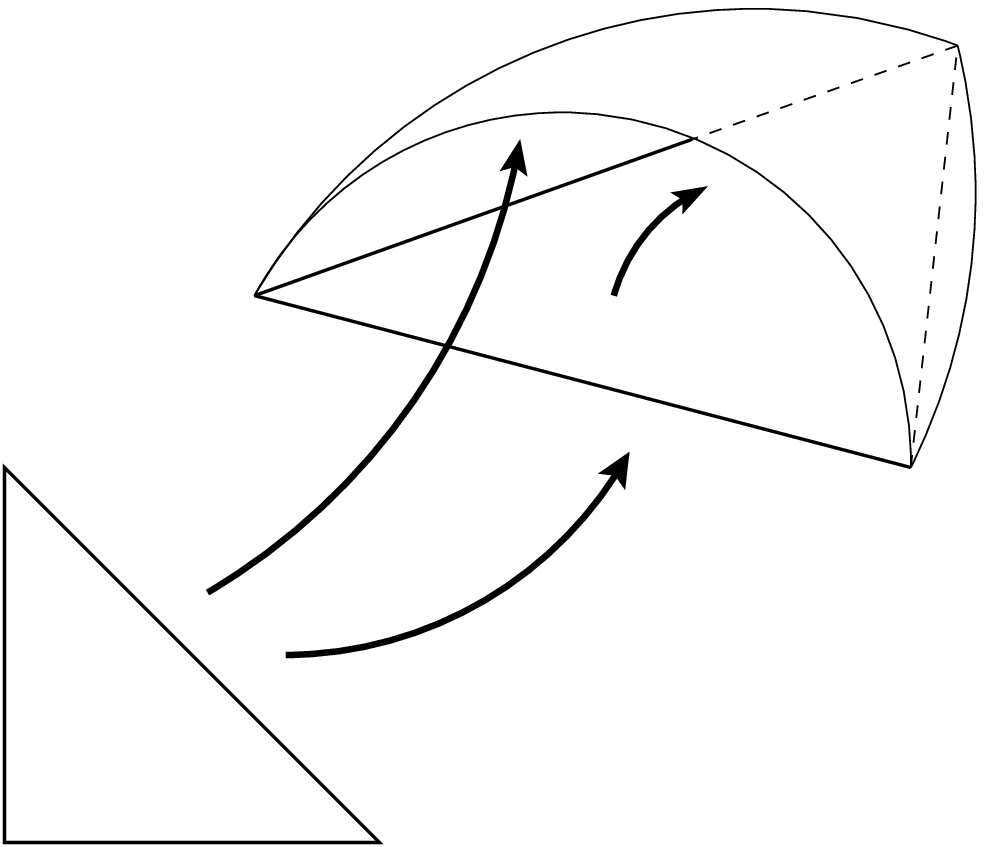}}
\caption{
Non-overlapping parametrizations  $\linparam_T : \widehat T \rightarrow T$ of $\Gamma$ and $\exactparam_T : \widehat T \rightarrow \widetilde T$ of $\gamma$.} \label{f:param}
\end{figure}

It turns out that it
will be convenient to consider $\bchi_T$ to be defined on a larger domain than $T$,
say $\widehat{\omega}_T\subset\mathbb{R}^n$, so that 
$\exactparam_T = \bP \circ\linparam_T:\widehat{\omega}_T \to\widetilde{\omega}_T$
is a  bi-Lipschitz local parametrization of $\gamma$: there exists
a universal constant $L\geq1$ such that for each fixed $T \in \T$ and
for all $\widetilde \bx_1=\exactparam_T(\by_1), \widetilde \bx_2=\exactparam_T(\by_2)
\in \widetilde{\omega}_T$, 
\begin{equation}\label{bi_lipschitz}
L^{-1} h_T  |\by_1-\by_2| 
\le |\widetilde \bx_1 -  \widetilde \bx_2 | 
\le L h_T |\by_1-\by_2|;
\end{equation}
in this case $\bchi:= \{ \bchi_T \}_{T\in \T}$ is an {\it overlapping} parametrization
of $\gamma$.
We further assume that $\bP(\mathbf v) = \mathbf v$ 
for all vertices $\mathbf v$ of $\lingamma$, so that $\linparam_T$ is the nodal
Lagrange interpolant of $\exactparam_T$ into linears.  

We finally note that a function $\widetilde \tv_T:\widetilde{T}\to\mathbb R$ defines uniquely two
functions $\widehat{\tv}_T:\widehat{T}\to\mathbb R$ and $\tv_T:T\to\mathbb R$ via the
maps $\bXi_T$ and $\bP$, namely
\begin{align}  \label{para_lift}
\widehat{\tv}_T(\by) := \widetilde \tv_T(\exactparam_T(\by)) \quad \forall\; \widehat{\bx}\in
\widehat{T} \quad \text{ and }\quad \tv_T(\bx) := \widetilde \tv_T(\bP(\bx))
\quad\forall\; \bx\in T.
\end{align}
Moreover, each one of these functions induces the other two uniquely. Accordingly,
we will use the symbol $\tv$ for all three functions if no confusion arises.

\medskip\noindent
{\bf Differential Geometry on Polyhedral Surfaces.} 
%
We use the  atlas $\{\widehat{T},\widetilde{T},\bchi_T\}_{T\in\T}$,
induced by the non-overlapping parametrization $\bchi:= \{ \bchi_T \}_{T\in \T}$,
to describe $\gamma$ in the spirit of Section \ref{sec:preliminaries}.
Likewise, we employ the atlas $\{\widehat{T},T,\bX_T\}_{T\in\T}$ to describe
the polyhedral surface $\Gamma$.
In view of \eqref{e:shape_reg_exact}, 
the discrete first fundamental form $\bg_T$ and area element $q_T$ of $\Gamma$
are given elementwise by
\begin{equation}\label{area}
  \bg_T := (D \bX_T)^t D \bX_T,
  \quad
  q_T := \sqrt{\det\bg_T},
  \quad\forall \, T\in\T.
\end{equation}
and satisfy
\begin{equation}\label{e:eigen_g}
  \textrm{eigen}(\bg_T) \approx h_T^2, \quad  q_T \approx h_T^n,
  \quad\forall \, T\in\T.
\end{equation}
They give rise to the piecewise constant functions $\bg_\Gamma:=\{\bg_T\}_{T\in\T}$
and $q_\Gamma:=\{q_T\}_{T\in\T}$. Similar properties are enjoyed by $\bchi$, which
imply that the stability constant $S_\bchi$ defined in \eqref{stab-const} is purely
geometric and independent of meshsize:
\begin{equation}\label{Sapprox1}
S_\bchi \approx 1.
\end{equation}
In addition, notice that \eqref{e:shape_reg_exact} and \eqref{bi_lipschitz} imply that
\begin{equation}\label{q:nondegen:assume}
C_1 \le \frac{q}{q_\Gamma} \le C_2
\end{equation}
for constants $C_1$, $C_2$ independent of discretization parameters.
Moreover, the
vector $\bN_T := \sum_{i=1}^{n+1} \det\big([\be_j, D \bX_T)]\big) \be_j$
is perpendicular to $T\in\T$
provided $\{ \be_j \}_{j=1}^{n+1}$ are the canonical unit vectors of $\mathbb R^{n+1}$.
This vector satisfies $|\bN_T|=q_T$ and yields the unit normal to $T$
\[
\bnu_T := \frac{\bN_T}{|\bN_T|}
\quad\forall \, T\in\T,
\]
and corresponding piecewise constant unit normal vector
$\bnu_\Gamma := \{ \bnu_T \}_{T \in \T}$ to $\Gamma$.

Given a function $\tv:\Gamma\to\mathbb{R}$, its tangential gradient
$\nabla_\Gamma\tv$ and Laplace-Beltrami operator $\Delta_\Gamma \tv$ over $\Gamma$
obey the formulas
\begin{equation}\label{grad-Gamma}
  \nabla \widehat\tv = (D \bX)^t \, \nabla_{\Gamma} \tv, \quad
  \nabla_{\Gamma} \tv = (D \bX) \, \bg_\Gamma^{-1} \, \nabla \widehat\tv,
\end{equation}
and 
\begin{equation}\label{lap-bel}
\Delta_\Gamma \tv = \frac{1}{q_\Gamma} {\text{div}} 
\big(q_\Gamma  \, \bg_\Gamma^{-1} \, \nabla\widehat\tv \big),
\end{equation}
where $\widehat \tv:{\widehat{T}}\to\mathbb{R}$ is defined in \eqref{para_lift}.
The strong form of $\Delta_\Gamma \tv$ is well defined only elementwise because
$q_\Gamma  \, \bg_\Gamma^{-1}$ is piecewise constant and so discontinuous over $\T$.
To find the correct strong form we start from the weak form \eqref{e:weak_relax},
split the integral over elements and use Corollary \ref{C:int-parts}
(integration by parts) to obtain
\begin{align*}
  \int_\Gamma \nabla_\Gamma \tv \cdot \nabla_\Gamma w & =
  \sum_{T\in\T} -\int_T w \Delta_\Gamma \tv +
  \int_{\partial T} w \nabla_\Gamma\tv\cdot\bmu_T
  \\
  & = \sum_{T\in\T} -\int_T w \Delta_\Gamma \tv +
  \sum_{S\in\mathcal{S}} \int_S w [\nabla_\Gamma \tv],
\end{align*}
where the jump residual is computed over each face $S\in\mathcal{S}$ of elements of $\T$
via
\begin{equation}\label{jump-residual}
  [\nabla_\Gamma v] := \nabla_\Gamma \tv_+ \cdot \bmu_+
    + \nabla_\Gamma \tv_- \cdot \bmu_-
\end{equation}
and $T_\pm \in\T$ are such that $S=T_+ \cap T_-$ and $\tv_\pm, \bmu_\pm$ are the
restrictions of $\tv$ and the outer unit normal to $\partial T_\pm$ which is parallel
to $T_\pm$. We then see that $\Delta_\Gamma \tv$ consists of the absolutely
continuous part \eqref{lap-bel} with respect to surface measure defined elementwise
and the singular part \eqref{jump-residual} supported on the skeleton of $\T$. This
formula makes sense for functions which are piecewise $H^2$ and globally $H^1$, such
as continuous piecewise polynomials.

\medskip\noindent
{\bf Parametric Finite Element Method.} 
%
In this work, we focus on continuous piecewise linear finite elements and
polyhedral surface approximations.
Let $\mathcal P$ be the space of linear polynomials and let $\V(\T)$ be the
space of continuous piecewise linear polynomial functions over $\Gamma$, namely
\begin{equation*}
\V(\T) :=\left\lbrace V \in C^0(\Gamma)  \;\big|\; 
	   V|_{T}  = \widehat V \circ \bX^{-1} \mbox{ for some } \widehat V\in \mathcal P, ~T \in \T\right\rbrace.
\end{equation*}
The finite element space associated with the Laplace-Beltrami equation over
$\Gamma$ is the restriction of $\V(\T)$ to functions with vanishing mean
\[
\V_\#(\T):= \V(\T) \cap L_{2,\#}(\Gamma).
\]
We define $\interp: C^0(\Gamma) \to \V(\T)$ to be the Lagrange interpolation operator
and use the same notation for vector-valued functions.

We are now ready to introduce the parametric FEM: seek $U:=U_\T \in \V_\#(\T)$
that solves
\begin{equation}\label{e:galerkin}
  \int_\Gamma \nabla_\Gamma U \cdot \nabla_\Gamma V = \int_\Gamma F V
  \quad \forall V \in \V_\#(\T),
\end{equation}
where $F \in L_{2,\#}(\Gamma)$ is an approximation of $f \in L_{2,\#}(\gamma)$ to be specified later.
Lax-Milgram theory guarantees that $U \in \V_\#(\T)$ is well defined. 
Observe that because $F \in L_{2,\#}(\Gamma)$, we also have
\begin{equation}\label{e:galerkin_relax}
  \int_\Gamma \nabla_\Gamma U \cdot \nabla_\Gamma V = \int_\Gamma F V
  \quad \forall V \in \V(\T).
\end{equation}

Since the exact problem \eqref{e:weak_relax} and discrete problem \eqref{e:galerkin_relax}
are defined on different domains $\gamma$ and $\Gamma$, the error $u-U$ does not
satisfy Galerkin orthogonality in either one. The next statement accounts for
geometric inconsistency and uses the convention~\eqref{para_lift} for the
generic lift $\bP$.

\begin{lemma}[Galerkin quasi-orthogonality]\label{L:GO}
Let $\bE$ and $\bE_\Gamma$ be defined in \eqref{error-matrix-g} and
  \eqref{error-matrix-G} via the parametrizations $\bchi=\bP\circ\bX$
  and $\chi_\Gamma=\bX$. Then, for all $V \in \V(\T)$, there holds
\begin{equation}\label{e:galerkin_Gamma}
\int_\Gamma \nabla_\Gamma (u-U) \cdot \nabla_\Gamma V = \int_\Gamma \Big(f \frac{q}{q_\Gamma}-F \Big) V + \int_\Gamma \nabla_\Gamma u \cdot \bE_\Gamma \, \nabla_\Gamma V,
\end{equation}
and 
\begin{equation}\label{e:galerkin_gamma}
\int_\gamma \nabla_\gamma (\wu-\wU) \cdot \nabla_\gamma \wV = \int_\gamma \Big (\widetilde f - \wF \frac{q_\Gamma}{q} \Big) \wV + \int_\gamma \nabla_\gamma \wU \cdot \bE \, \nabla_\gamma \wV.
\end{equation}
\end{lemma}
\begin{proof}
We only prove \eqref{e:galerkin_Gamma} as \eqref{e:galerkin_gamma} follows similarly.
Using the equation \eqref{e:galerkin_relax} satisfied by $U$ and the consistency relation \eqref{consistency}, we obtain
\begin{align*}
\int_\Gamma \nabla_\Gamma (u-U) \cdot \nabla_\Gamma V
= \int_\gamma \nabla_\gamma \wu\cdot \nabla_\gamma \wV + \int_\Gamma \nabla_\Gamma u \cdot \bE_\Gamma \, \nabla_\Gamma V- \int_\Gamma F V.
\end{align*}
The first term on the right-hand side
equals $\int_\gamma \widetilde f V = \int_\Gamma f \frac{q}{q_\Gamma} V$, in view
of \eqref{e:weak_relax} and \eqref{int-gamma}, and thus yields \eqref{e:galerkin_Gamma}.
\end{proof}

\subsection{Geometric Consistency}\label{S:geom_consistency}
%
In this section we study the error inherent to approximating $\gamma$ with $\Gamma$.
The polyhedral surface $\Gamma$
is always represented by a lift $\bP$ whose regularity depends
on that of $\gamma$. We present two scenarios depending on such regularity. We
first assume that $\gamma$ is piecewise $C^{1,\alpha}$ and globally Lipschitz,
and later assume that $\gamma$ is $C^2$ and exploit the distance function lift
$\bP_d$ to improve the error estimates.

\medskip\noindent
{\bf Uniform Poincar\'e-Friedrichs estimate on $\Gamma$.} 
The analysis below takes advantage of the \emph{uniform Poincar\'e-Friedrichs estimate on $\Gamma$}
\begin{equation} \label{poin-unif} \|\tv\|_{L_2(\Gamma)} \lesssim \|\nabla \tv\|_{L_2(\Gamma)}
\quad\forall \tv \in H^1_\#(\Gamma),
\end{equation}
where the constant hidden in the above inequality is independent of $\Gamma$.
Note that when $\gamma$ is of class $C^{1,\alpha}$, $0<\alpha \leq 1$, Lemma~\ref{L:Poincare-unif} (uniform Poincar\'e-Friedrichs constant) implies that \eqref{poin-unif} follows from \eqref{Sapprox1} and \eqref{q:nondegen:assume}, which in turn are consequences of assumption~\eqref{surf_bi_lipschitz}.
Furthermore, when $\gamma$ is of class $C^2$ and $\bP=\bP_d$, the discussion in Section \ref{S:perturb-C2} yields conditions which are also  easy to verify: $\Gamma \subset \mathcal{N}(1/2K_\infty)$ and  $\bnu\cdot \bnu_\Gamma \ge c >0$ on $\Gamma$.

\medskip\noindent
{\bf Geometric Estimators.} 
%
Since $\gamma$ is described by $\bchi$ and $\Gamma$ by $\bX$ it is natural to
consider the difference  $D\bchi-D\bX$ as a measure of geometric mismatch \cite{BCMMN16}.
We thus start with the \textit{geometric element indicator}
\begin{equation}
  \lambda_T:= \| \smash { D (\bP-\interp \bP)}\|_{ L_\infty(T)} =
  \| \smash { D \bP- \bI}\|_{ L_\infty(T)}
  \quad\forall \, T\in\T
\end{equation}
and the corresponding \textit{geometric estimator}
\begin{equation}\label{geo-estimator}
\lambda_\T(\Gamma) :=\max_{ T \in \T} \lambda_T.
\end{equation}
We have seen that the relative measure of accuracy \eqref{geo-est}
controls the geometric error. In this vein, 
we observe that $D\bchi_T=D\bP \, D\bX_T$ because
$\exactparam_T = \bP \circ \linparam_T$, whence such measure satisfies
\begin{equation}\label{e:approx_X}
  \max_{\by\in\widehat{T}} \frac{| D (\exactparam_T-\linparam_T)(\by)|}
       {\min\big\{|D^- \exactparam_T(\by)| \, , |D^- \linparam_T(\by)|\big\}}
  \leq S_\bchi \lambda_T
  \quad\forall \, T\in\T,
\end{equation}
with a stability constant $S_\bchi\approx 1$ according to \eqref{stab-const}.
Therefore $\lambda_\T(\Gamma)$ is expected to dominate the geometric error
for surfaces of class $C^{1,\alpha}$ with $0<\alpha<1$. This is consistent
with Lemma \ref{L:perturbation_bound} (perturbation error estimate for $C^{1,\alpha}$
surfaces).

For $C^2$ surfaces, however, $\lambda_\T(\Gamma)$ is suboptimal in that it
overestimates the influence of geometry \cite{BD:18}. According to Lemma
\ref{L:perturbation_bound_dist} (perturbation error estimate for $C^2$ surfaces)
and Lemma
\ref{L:est-normals} (error estimates for normals), the following quantities
should play a crucial role in dealing with geometry via the auxiliary lift $\bP_d$
\begin{equation}
\label{d:beta}
\beta_T:=\|\bP-\interp\bP\|_{L_\infty(T)},
\quad\beta_\T(\Gamma):=\max_{T \in \T} \beta_T \, ,
\end{equation}
and
\begin{equation}\label{d:mu}
\mu_T := \beta_T + \lambda_T^2, \quad \mu_\T(\Gamma) := 
\max_{T \in \T} \mu_T;
\end{equation}
we stress that $\mu_\T(\Gamma)$ is formally of higher order than $\lambda_\T(\Gamma)$. 
We will show below that $\mu_\T(\Gamma)$ indeed controls the geometric error and
accounts for the ``superconvergence'' property associated with
the projection $\bP_d$ along the normal direction to $\gamma$
alluded to at the end of section \ref{S:perturb-C2}.


\medskip\noindent
{\bf Geometric Consistency Error for $C^{1,\alpha}$ Surfaces.}  
%
We now quantify the geometric error incurred when replacing $\gamma$ by its polygonal approximation $\Gamma$.

\begin{corollary}[geometric consistency errors for $C^{1,\alpha}$ surfaces] \label{C:lambda}
If $\bX$ and $\bchi$ satisfy \eqref{e:shape_reg_init} and
\eqref{e:shape_reg_exact}, then for all $T \in \T$ we have
\begin{equation}\label{e:approx_q_g}
  \| 1 - q^{-1}q_\Gamma \|_{L_\infty(\widehat T)} ,~~
  \| \bI - \bg_\Gamma\bg^{-1} \|_{L_\infty(\widehat T)} ,~~
  \| \bnu_\Gamma - \bnu \|_{L_\infty(T)}
  \lesssim \lambda_T,
\end{equation}
where the hidden constants depend on $S_\bchi \approx 1$ defined in
\eqref{stab-const}. Moreover,
\begin{equation}\label{e:estim_consist}
  \| \bE \|_{L_\infty(\widehat T)} + \| \bE_\Gamma\|_{L_\infty(\widehat T)} \lesssim \lambda_T
  \quad\forall \, T \in \T. 
\end{equation}
\end{corollary}
\begin{proof}
We first point out that \eqref{e:shape_reg_init} and \eqref{e:shape_reg_exact}
yield $S_\bchi \approx 1$ according to \eqref{Sapprox1}. The asserted estimates
follow from Lemma \ref{L:error-est} (error estimates for $\bg$ and $q$) and
Lemma \ref{L:est-normals} (error estimate for normals)
in conjunction with \eqref{bound-E} and \eqref{e:approx_X}.
\end{proof}

\medskip\noindent
{\bf Geometric Consistency Error for $C^2$ Surfaces.}
%
We now take advantage of the lift $\bP_d$ for error representation. We recall
that, as in section \ref{S:perturb-C2},
the parametrizations of $\gamma$ and $\Gamma$ are given by
$\bchi = \bP_d \circ \bX$ and $\bX$. In particular, the infinitesimal area element $q$ of $\gamma$ is defined using $\bP_d$ and so are the consistency matrices $\bE$, $\bE_\Gamma$; see \eqref{error-matrix-gamma}, \eqref{error-matrix-Gamma}.
To improve upon Corollary \ref{C:lambda} (geometric
consistency errors for $C^{1,\alpha}$ surfaces) we need more stringent geometric
assumptions than simply $\Gamma \subset \mathcal N$.
These assumptions are somewhat technical but are checkable computationally
with information extracted from $\bP$ but without access to $\bP_d$ \cite{BD:18}.
We list them now.

\begin{itemize}
\item
{\bf Distance between $\gamma$ and $\Gamma$.}
Invoking the closest point property of the distance function projection
$\bP_d$ and the definition \eqref{d:beta} of $\beta_\T(\Gamma)$, we see that
$|\bx - \bP_d(\bx)| \le |\bx -\bP(\bx)| \le \beta_\T(\Gamma)$ for all $\bx\in\Gamma$.
We thus assume that $\Gamma$ is sufficiently close to $\gamma$ in the sense that
\begin{equation}\label{beta-small}
\beta_\T(\Gamma) < \frac{1}{2 K_\infty}
\quad\Rightarrow\quad
\Gamma\subset\mathcal{N},
\end{equation}
according to \eqref{N:def-2}. Therefore, the estimates of section \ref{S:perturb-C2}
are valid. Moreover, the discrepancy between the two lifts satisfies for all $T\in\T$
\[
| \bP(\bx) - \bP_{d}(\bx)| \leq | \bP(\bx) - \bx| + |\bx -\bP_{d}(\bx)|
\leq 2 | \bx -\bP(\bx)| \le 2 \beta_T
\quad\forall \,\bx\in T.
\]
\item
{\bf Mismatch between $\bP$ and $\bP_d$.} We assume that
\begin{equation}\label{P_Pd:mismatch}
  \bP_d \circ \bP^{-1} (\widetilde{T}) \subset \widetilde{\omega}_T
  \quad\forall \, T \in \T,
\end{equation}
where $\widetilde{\omega}_T$ is the patch around $\widetilde{T}$ within $\gamma$.
If $\widetilde\bx=\bP(\bx)\in\gamma$ for $\bx\in\Gamma$, then
\begin{equation}\label{error-mismatch}
|\widetilde\bx -  \bP_d \circ \bP^{-1}(\widetilde\bx)|
= |\bP(\bx) - \bP_d(\bx)| \le 2 \beta_T
\quad\forall \, \bx\in T.
\end{equation}
and all $T\in\T$. Since $\gamma$ is of class $C^2$, we expect $\frac{\beta_T}{h_T}\to0$
as $h_T\to0$ and realize that \eqref{P_Pd:mismatch} is always valid asymptotically.
We emphasize that it is possible to check \eqref{P_Pd:mismatch} computationally
without accessing $\bP_d$ \cite{BD:18}.

\end{itemize}

\begin{corollary}[geometric consistency errors for $C^2$ surfaces] \label{C:lambda-2}
If \eqref{beta-small} and \eqref{q:nondegen:assume} hold, then
so do the following estimates for all $T\in\T$
\begin{equation}\label{q-d-nu}
  \|d\|_{L_\infty(T)} \lesssim \beta_T,
  \quad
  \|\bnu - \bnu_\Gamma \|_{L_\infty(T)} \lesssim \lambda_T,
  \quad
  \| 1 - q^{-1} q_\Gamma\|_{L_\infty(T)} \lesssim \mu_T,
\end{equation}
where all the geometric quantities are defined using the parametrizations
$\bchi = \bP_d \circ \bX$ and $\bX$.
Moreover,
\begin{equation}\label{est-E-EG}
  \|\bE\|_{L_\infty(T)} , ~~ \|\bE_\Gamma\|_{L_\infty(T)} \lesssim \mu_T
  \quad\forall \, T\in\T.
\end{equation}
\end{corollary}
\begin{proof}
The first estimate in \eqref{q-d-nu} is trivial from the definition \eqref{d:beta}
of $\beta_T$, whereas the second estimate in \eqref{q-d-nu} is a consequence of
\eqref{est-normals}. The third estimate in \eqref{q-d-nu} results from 
\eqref{measure_error} and \eqref{est-normals}. With these estimates at hand,
the estimate for $\bE$ in \eqref{est-E-EG} comes from \eqref{est-E} and that
for $\bE_\Gamma$ is similar.
\end{proof}

We conclude with a technical result assessing the mismatch between $\bP$ and $\bP_d$.
We motivate it with the following simpler $L_\infty$-estimate valid for all $T\in\T$
\[
\| \widetilde w - \widetilde w \circ \bP_d \circ \bP^{-1}\|_{L_\infty(\widetilde T)}
\lesssim \| \nabla_\gamma \widetilde w \|_{L_{\infty}(\widetilde{\omega}_T)} \, \beta_T
\quad\forall \, \bx\in T.
\]
This is a trivial consequence of the property \eqref{error-mismatch}
for $\widetilde{\bx} \in \widetilde{T}$
\[
| \widetilde{w}(\widetilde \bx) - \widetilde{w} ( \bP_d \circ \bP^{-1}(\widetilde \bx))|
\leq \| \nabla_\gamma \widetilde w \|_{L_{\infty}(\widetilde{\omega}_T)} | \widetilde \bx -   \bP_d \circ \bP^{-1}(\widetilde\bx)|\le 2 \| \nabla_\gamma \widetilde w \|_{L_{\infty}(\widetilde{\omega}_T)} \beta_T.
\]
The estimate below is $L_2$-based and its proof entails regularization by convolution
\cite[Lemma~3.4]{BD:18}.

\begin{proposition}[mismatch between $\bP$ and $\bP_d$]\label{l:convolution}
Assume that \eqref{e:shape_reg_init} as well as the assumptions
\eqref{q:nondegen:assume}, \eqref{beta-small} and \eqref{P_Pd:mismatch} hold.
Then there exists $\lambda_* >0$ such for $\widetilde{w} \in H^1(\gamma)$ and $T \in \T$ we have
$$
\| \widetilde w - \widetilde w \circ \bP_d \circ \bP^{-1} \|_{L_2(\widetilde T)} \lesssim \beta_T \| \widetilde w \|_{H^1(\widetilde \omega_T)}
$$
provided $\lambda_T \leq \lambda_*$ and $\widetilde \omega_T$ is a patch in
$\gamma$ around $\widetilde{T}$.
\end{proposition}
\begin{proof}
We proceed in several steps.

\noindent
{\it Step 1: Reduction to $\mathbb{R}^n$.}
Fix $T \in \T$ and recall that $\bchi_T = \bP \circ \bX_T$ satisfies
\eqref{e:shape_reg_exact} and maps the reference patch $\widehat{\omega}_T$
into $\widetilde{\omega}_T$. For notational ease, let
\[
\psi=\bP_d \circ \bP^{-1}:\gamma\to\gamma,
\qquad
\widehat{\psi}=\bchi_T^{-1} \circ \psi \circ \bchi_T:\widehat{\omega}_T\to\widehat{\omega}_T.
\]
Given $\widetilde{w} \in H^1(\gamma)$, let
$\widehat{w}=\widetilde{w} \circ \bchi_T:\widehat{\omega}_T\to\mathbb{R}$,
and note that $\widehat{w}\in H^1(\widehat{\omega}_T)$ because $\bchi_T$ is Lipschitz.
We change variables via $\bchi_T$ to $\widehat{T}$ and invoke the non-degeneracy
property \eqref{q:nondegen} to obtain
\[
\|\widetilde{w}-\widetilde{w}\circ \psi\|_{L_2(\widetilde{T})}
\lesssim h_T^{n/2} \|\widehat{w}-\widehat{w} \circ \widehat{\psi}\|_{L_2(\widehat{T})}.
\]
The assumption $\bP_d \circ \bP^{-1} (\widetilde{T}) \subset \widetilde{\omega}_T$ given in \eqref{P_Pd:mismatch} is equivalent to $\widehat{\psi}(\widehat{T}) \subset \widehat{\omega}_T$ and is sufficient to ensure that the quantity on the right hand side is well-defined.

Since $\widehat{w}$ is defined on $\widehat{\omega}_T\subset\mathbb{R}^n$, and
its boundary is Lipschitz, there is a universal extension operator
$E:H^1(\widehat{\omega}_T)\to H^1(\mathbb{R}^n)$ which is bounded both in $L_2$ and in the $H^1$-seminorm \cite{Stein:70}; this is the so-called Calder\'on operator. 
We relabel $E \widehat{w}$ to be $\widehat{w}$, and thus assume it is bounded in
$H^1(\mathbb{R}^n)$ while satisfying
\[
|\widehat{w}|_{H^1(\mathbb{R}^n)} \lesssim |\widehat{w}|_{H^1(\widehat{\omega}_T)}.
\]

\noindent
{\it Step 2: Mollification.}
We now regularize $\widehat{w}$ by convolution with a standard smooth mollifier
supported in the ball $B(0,\varepsilon)$ centered at $0$ with radius $\varepsilon>0$
to be determined. If $\Omega\subset\mathbb{R}^n$ is an arbitrary domain,
it is well known that
\begin{gather*}
\|\widehat{w}-\widehat{w}_\varepsilon\|_{L_2(\Omega)} \lesssim
\varepsilon |\widehat{w}|_{H^1(\Omega + B(0,\varepsilon))},
\\
|\widehat{w}_\varepsilon|_{W^1_\infty(\Omega)} \lesssim
\varepsilon^{-n/2} |\widehat{w}|_{H^1(\Omega + B(0,\varepsilon))}.
\end{gather*}
We may now write, without restriction on $\varepsilon$, that
\[
\|\widehat{w}-\widehat{w} \circ \widehat{\psi}\|_{L_2(\widehat{T})} \lesssim
\|\widehat{w}-\widehat{w}_\varepsilon\|_{L_2(\widehat{T})}
+ \|\widehat{w}_\varepsilon-\widehat{w}_\varepsilon \circ \widehat{\psi}\|_{L_2(\widehat{T})}
+ \|\widehat{w}_\varepsilon\circ \widehat{\psi}-\widehat{w} \circ \widehat{\psi}\|_{L_2(\widehat{T})}.
\]
We estimate the first term using the first formula above for the mollifier
\[
\|\widehat{w}-\widehat{w}_\varepsilon \|_{L_2(\widehat{T})} \lesssim
\varepsilon |\widehat{w}|_{H^1(\mathbb{R}^n)} \lesssim
\varepsilon |\widehat{w}|_{H^1(\widehat{\omega}_T)}.
\]
Similarly, changing variables via the map $\widehat\psi$, which turns out to be
Lipschitz in view of \eqref{q:nondegen} and \eqref{e:shape_reg_exact}, and
applying the restriction $\widehat{\psi}(\widehat{T}) \subset \widehat{\omega}_T$
stated in \eqref{P_Pd:mismatch}, we find that
\[
\|(\widehat{w}_\varepsilon -\widehat{w}) \circ \widehat{\psi}\|_{L_2(\widehat{T})} \lesssim
\|\widehat{w}_\varepsilon-\widehat{w}\|_{L_2(\widehat{\omega}_T)} \lesssim
\varepsilon |\widehat{w}|_{H^1(\widehat{\omega}_T)}.
\]

\medskip\noindent
{\it Step 3: Estimate for
$\|\widehat{w}_\varepsilon-\widehat{w}_\varepsilon \circ \widehat{\psi}\|_{L_2(\widehat{T})}$.}
Let $\{\by_i\}$ be a lattice on $\mathbb{R}^n$ with minimum distance between $\by_i$ and $\by_j$ ($i \neq j$) proportional to $\varepsilon$ and such that $\{B(\by_i,\varepsilon)\}$ covers $\mathbb{R}^n$.  The set $\{B(\by_i,M\varepsilon)\}$ then has finite overlap for any $M\ge1$, with the maximum cardinality of the overlap depending on $M$. We choose
\[
\varepsilon = \sup_{\by\in\widehat{T}} |\by - \widehat\psi(\by)|
\quad\Rightarrow\quad
\|\widehat{w}_\varepsilon-\widehat{w}_\varepsilon \circ \widehat{\psi}\|_{L_\infty(B(\by_i,\varepsilon)\cap\widehat{T})}
 \lesssim \varepsilon |\widehat{w}_\varepsilon|_{W_\infty^1(B(\by_i,2\varepsilon))}
\]
Applying the second property of mollifiers given above yields
\[
|\widehat{w}_\varepsilon|_{W_\infty^1(B(\by_i,2\varepsilon))} \lesssim
\varepsilon^{-n/2} |\widehat{w}|_{H^1(B(\by_i,3\varepsilon))},
\]
whence
\begin{align*}
  \|\widehat{w}_\varepsilon  -\widehat{w}_\varepsilon \circ \widehat{\psi}\|_{L_2(\widehat{T})}^2  &\lesssim \varepsilon^{n} \sum_{i} \|\widehat{w}_\varepsilon-\widehat{w}_\varepsilon \circ \widehat{\psi}\|_{L_\infty(B(\by_i,\varepsilon) \cap \widehat{T})}^2
\\  
& \lesssim \varepsilon^2 \sum_{i} |\widehat{w}|_{H^1(B(\by_i,3\varepsilon))}^2
 \lesssim \varepsilon^2 |\widehat{w} |_{H^1(\mathbb{R}^n)}^2
\lesssim \varepsilon^2 |\widehat{w}|_{H^1(\widehat{\omega}_T)}^2.
\end{align*}

\medskip\noindent
{\it Step 4: Bound on $\varepsilon$.}
Making use of the bi-Lipschitz character \eqref{e:shape_reg_exact}, we get
\begin{align*}
  \big| \by - \widehat\psi (\by) \big| & =
  \big|\bchi_T^{-1} \big(\bchi_T(\by)) - \bchi_T^{-1}\big(\psi(\bchi(\by))\big) \big|
  \\
  & \le L h_T^{-1} \big| \bchi_T(\by) -  \psi(\bchi(\by)) \big|
  = L h_T^{-1} \big| \widetilde{\bx} - \bP_d\circ\bP^{-1}( \widetilde{\bx} ) \big|,
\end{align*}
where $\widetilde{\bx} = \bchi_T(\by)$.
Recalling \eqref{error-mismatch} and the definition of $\varepsilon$, we thus obtain
\[
\varepsilon \le 2L h_T^{-1} \beta_T.
\]
We now gather the estimates of Steps 2 and 3.
Mapping from $\widehat{T}$ to $\widetilde{T}$ and back via $\bchi_T$, and utilizing
\eqref{q:nondegen} and \eqref{e:shape_reg_exact}, yields
\begin{align*}
  \|\widetilde{w}-\widetilde{w} \circ \psi \|_{L_2(\widetilde{T})}^2
  &\lesssim h_T^n \|\widehat{w}-\widehat{w} \circ \widehat{\psi}\|_{L_2(\widehat{T})}^2
  \lesssim h_T^n \varepsilon^2 |\widehat{w}|_{H^1(\widehat{\omega}_T)}^2 
\\ & \lesssim h_T^n h_T^{-2} \beta_T^2 h_T^{2-n} |\widetilde{w}|_{H^1(\widetilde{\omega}_T)}^2 =\beta_T^2 |\widetilde{w}|_{H^1(\widetilde{\omega}_T)}^2.  
\end{align*}
This completes the proof.
\end{proof}

We conclude this section with a variant of Proposition \ref{l:convolution}
(mismatch between $\bP$ and $\bP_d$) which turns out to be instrumental for the
study of the Narrow Band method discussed later in Section \ref{sec:narrow}.
\begin{proposition}[Lipschitz perturbation]\label{p:mol_bulk}
Let $\Omega_1,\Omega_2\subset\subset\Omega\subset\mathbb R^{n+1}$ be Lipschitz
bounded domains and $\bL:\Omega_1\to\Omega_2$
be a bi-Lipschitz isomorphism. If \looseness=-1
\[
r := \max_{\bx\in\Omega_1} |\bL(\bx) - \bx|
\]
is sufficiently small so that $(\Omega_1\cup\Omega_2)+B(0,r)\subset\Omega$
then for all $g\in H^1(\Omega)$ we have
$$
\| g - g \circ \bL \|_{L^2(\Omega_1)} \lesssim r \| g \|_{H^1(\Omega)}.
$$
\end{proposition}

\begin{proof}
We now proceed as in Proposition \ref{l:convolution}:
let $\eps=r >0$ and $g_\eps$ be a regularization of $g$ by convolution with a
standard smooth mollifier supported in the ball $B(0,\eps)$. We write
$$
\| g - g \circ \bL \|_{L^2(\Omega_1)}  \leq \| g - g_\eps \|_{L^2(\Omega_1)}
+ \| g_\eps - g_\eps \circ \bL \|_{L^2(\Omega_1)}
+ \| g_\eps \circ \bL - g \circ \bL \|_{L^2(\Omega_1)} 
$$
and note that
$$
\| g - g_\eps \|_{L_2(\Omega_1)} \lesssim \eps \| g\|_{H^1(\Omega)},
\quad
\| g_\eps \circ \bL - g \circ \bL \|_{L^2(\Omega_1)} \lesssim \eps \| g\|_{H^1(\Omega)}
$$
because $\bL^{-1}$ is Lipschitz. 
To estimate $\| g_\eps - g_\eps \circ \bL \|_{L^2(\Omega_1)}$, we argue as in Step 3 of
Proposition \ref{l:convolution} (mismatch between $\bP$ and $\bP_d$). This
completes the proof.
\end{proof}

\subsection{A-Priori Error Analysis}\label{S:a-priori}

In this section we derive {\it a-priori} error estimates in $H^1$ and $L_2$, namely
estimates expressed in terms of regularity of the exact solution $\wu$ of \eqref{e:weak}.
Compared to the existing literature these estimates
involve two lifts: $\bP_d$ and $\bP$. The former, based on the distance function $d$, is only used theoretically or to define a notion of error when comparing $U$ with $\widetilde u$. 
The latter is generic and used in practice to define the finite element method, i.e., by setting $F = \widetilde f \circ \bP \frac{q}{q_\Gamma}$ and the discrete parametrization
$\bX$ to be the interpolant of the continuous one $\chi=\bP\circ\bX$.
Optimal orders of convergence are derived without the need to access the
distance function.

We also address a gap in the literature.  Existing proofs of optimal a priori estimates for surface FEMs employ the distance function lift $\bP_d=\bx-d(\bx) \nabla d(\bx)$.  However, when $\gamma$ is $C^2$, this map is only $C^1$ because of the presence of $\nabla d$ in its definition.  Thus given $\wv \in H^2(\gamma)$, its extension $\tv = \wv \circ \bP_d$ to $\Gamma$ is only in $H^1$ and not piecewise in $H^2$ as is needed to prove optimal approximation order.  Thus existing proofs that only employ the distance function lift require the assumption that $\gamma$ be of class $C^3$ in order to obtain optimal order error estimates in the standard way; cf. the work of Dziuk in \cite{Dz88} in which such error estimates were originally obtained.   

As pointed out already in Theorem~\ref{t:C1_implies_C2} ($C^1$ distance function implies $C^{1,1}$ surface), the distance function $d$ to a $C^{1,\alpha}$ surface is no better than Lipschitz in general. Therefore, the aforementioned strategy does not extend to
$C^{1,\alpha}$ surfaces. 
However, the best approximation property of the Galerkin method together with the geometric consistency estimates of Section~\ref{S:geom_consistency} yields a-priori error estimates in $H^1$. We present this discussion after that for $C^2$ surfaces.

\medskip\noindent
{\bf A-Priori Error Estimates for $C^2$ Surfaces.}
The following lemma will be instrumental to prove optimal a priori error estimates for $\gamma$ of class $C^2$. It states that a function
$\nabla_\Gamma(\wu\circ\bP_d)$ can be approximated in $H^1(\Gamma)$ to first order
for a function $\wu\in H^2(\gamma)$. The
difficulty is that the composite function $\wu\circ\bP_d\notin H^2(\Gamma)$
whereas $\nabla_\gamma \wu \circ\bP_d \in H^1(\Gamma)$. The proof exploits this property
to restore optimal approximability of $\nabla_\Gamma(\wu\circ\bP_d)$ in $H^1(\Gamma)$.

\begin{lemma}[approximability in $H^1(\Gamma)$]\label{L:approxH1}
Let $\gamma$ be a surface of class $C^2$ and $\wu \in H^2(\gamma)$. Let
$K_\infty$ be defined in \eqref{K:def} and $\beta_\T(\Gamma)$ be given in \eqref{d:beta}.
Then we have
\begin{equation}
\label{opt_approx}
\inf_{\tV \in \V(\T)} \|\nabla_\Gamma(\wu \circ \bP_d -\tV)\|_{L_2(\Gamma)}
\lesssim h_\T |\wu|_{H^2(\gamma)}+\beta_\T(\Gamma) K_\infty \|\nabla_\gamma \wu\|_{L_2(\gamma)}.
\end{equation}
\end{lemma}
\begin{proof}
We know from Veeser \cite{Vee15} that continuous and discontinuous
piecewise polynomial approximations in $H^1$ are equivalent.
 Even though this crucial result was originally
proved for Euclidean domains, it proofs carries over with essentially no changes
to the case of surface meshes
\begin{equation}
\label{local_gradient}
\inf_{\tV \in \V(\T)} \|\nabla_\Gamma(\wu \circ \bP_d -\tV)\|_{L_2(\Gamma)}^2 \lesssim \sum_{T \in \T} \inf_{\tV_T \in \V(T)} \|\nabla_\Gamma(\wu \circ \bP_d-\tV_T)\|_{L_2(T)}^2.
\end{equation}
We refer to \cite{CD15} for related results on surfaces. We thus fix $T\in\T$
and argue over this element hereafter; recall that $\widetilde{T}=\bP_d(T)$.

Applying the triangle inequality yields
$$
\big|\nabla_\Gamma(\wu \circ \bP_d-\tV_T) \big| \le
\big| \nabla_\Gamma(\wu \circ \bP_d)-\Pi_\Gamma (\nabla_\gamma \wu \circ \bP_d) \big|
+ \big| \Pi_\Gamma (\nabla_\gamma \wu \circ \bP_d) -\nabla_\Gamma \tV_T \big|.
$$
Using \eqref{e:tang_exact_to_discrete}, we next find that
$$|\nabla_\Gamma(\wu \circ \bP_d)-\Pi_\Gamma (\nabla_\gamma \wu \circ \bP_d)|=
\big|\Pi_\Gamma [d \bW  (\nabla_\gamma \wu \circ \bP_d)]\big| \le
K_\infty|d| \, \big|(\nabla_\gamma \wu) \circ \bP_d \big|,$$
which along with \eqref{H1:equiv} yields
$$
\| \nabla_\Gamma(\wu \circ \bP_d)-\Pi_\Gamma (\nabla_\gamma \wu \circ \bP_d)\|_{L_2(T)}
\lesssim \beta_T \, K_\infty \, \|\nabla_\gamma \wu\|_{L_2(\widetilde{T})}.
$$

Next note that $\Pi_\Gamma = \bI - \bnu_\Gamma\otimes\bnu_\Gamma$ is constant
over $T$. Therefore, $\Pi_\Gamma (\nabla_\gamma \wu \circ \bP_d) \in [H^1(T)]^{n+1}$
in $T$ because $\wu \in H^2(\gamma)$ implies $\nabla_\gamma \wu \in [H^1(\gamma)]^{n+1}$
and $\bP_d$ is $C^1$.  In addition, $\Pi_\Gamma (\nabla_\gamma \wu \circ \bP_d)$
is a tangent vector field on $\Gamma$.  On the other hand, $\nabla_\Gamma$
maps the affine functions $\mathbb{P}^1$ onto the subspace of $[\mathbb{P}^0]^{n+1}$
tangent to $\Gamma$, so standard approximation theory leads to
$$
\inf_{\tV_T \in \V(T)} \|{\bf w}-\nabla_\Gamma \tV_T\|_{L_2(T)} \lesssim h_T |{\bf w}|_{H^1(T)}
$$
for any tangent vector field ${\bf w} \in [H^1(T)]^{n+1}$ to $\Gamma$.  Using that
$\nabla \bP_d = \Pi-d \bW$ and $\bW$ is bounded because $\gamma$ is of class $C^2$,
together with the fact that $\Pi_\Gamma$ is constant in $T$, we deduce
\begin{align*} 
\inf_{\tV_T \in \V(T)}  \|\Pi_\Gamma (\nabla_\gamma \wu \circ \bP_d)&-\nabla_\Gamma \tV_T\|_{L_2(T)} \lesssim h_T |\Pi_\Gamma (\nabla_\gamma \wu \circ \bP_d)|_{H^1(T)}
\\ &  \lesssim h_T \|\Pi-d\bW\|_{L_\infty(T)} \|D_\gamma^2 \wu \circ \bP_d\|_{L_2(T)} 
 \lesssim h_T |\wu|_{H_2(\widetilde{T})},
\end{align*}
where we used the notation $D^2_\gamma  \wu := \nabla_\gamma \nabla_\gamma  \wu$.
This completes the proof.
\end{proof}

This proof reveals that \eqref{opt_approx} can in fact be written locally:   
\begin{align*}
\inf_{\tV \in \V(T)} \|\nabla_\Gamma(\wu \circ \bP_d -\tV)\|_{L_2(T)}^2
&\lesssim \beta_T^2 \, K_\infty^2 \|\nabla_\gamma \wu\|_{L_2(\widetilde{T})}^2
\\
& + \inf_{\bV \in \mathbb{P}^0(T)}  \|\Pi_\Gamma (\nabla_\gamma \wu \circ \bP_d)
- \bV \|_{L_2(T)}^2.
\end{align*}

We now apply Lemma \ref{L:approxH1} (approximability in $H^1(\Gamma)$) to derive an
a-priori error estimate. We present two proofs. The first one is very compact and
relies on Lemmas \ref{L:perturbation_bound} and \ref{L:perturbation_bound_dist}
(perturbation error estimate).
The second proof is selfcontained and paves the way to the $L_2$ error estimate
that follows. In both cases we rely on Lemma \ref{L:regularity} (regularity)
for $\gamma$ of class $C^2$ and $\wf \in L_{2,\#}(\gamma)$:
$$
\| \wu \|_{H^2(\gamma)} \lesssim \| \wf \|_{L_2(\gamma)}.
$$

\begin{theorem}[$H^1$ a-priori error estimate for $C^2$ surfaces] \label{t:H1error}
Let $\gamma$ be of class $C^2$,  $\wf \in L_{2,\#}(\gamma)$ and $\wu\in H^2(\gamma)$ 
be the solution of \eqref{e:weak}.
Let $U \in \V_\#(\T)$ be the solution to \eqref{e:galerkin} with $F =\wf \circ \bP \frac{q}{q_\Gamma}$ defined via the lift $\bP$.
If the geometric assumptions \eqref{bi_lipschitz},
\eqref{beta-small}, and \eqref{P_Pd:mismatch} are valid,
then
$$
\| \nabla_\Gamma (\wu \circ \bP -U)\|_{L_2 (\Gamma)}\lesssim \big(h_\T+ \lambda_\T(\Gamma) \big)
\| \wf \|_{L_{2}(\gamma)} \lesssim  h_\T \| \wf \|_{L_{2}(\gamma)}
$$
as well as
$$
\| \nabla_\Gamma (\wu \circ \bP_d-U)\|_{L_2 (\Gamma)}\lesssim \big(h_\T + \mu_\T(\Gamma) \big)
\| \wf \|_{L_{2}(\gamma)} \lesssim  h_\T \|\wf \|_{L_{2}(\gamma)}.
$$
\end{theorem}
\noindent
{\it Proof 1.}
We prove the second estimate.
Let $f_\Gamma=F$ and $u_\Gamma \in H^1_\# (\Gamma)$ solve \eqref{Gamma:LBproblem}
$$
\int_\Gamma \nabla_\Gamma u_\Gamma \nabla_\Gamma \tv = \int_\Gamma f_\Gamma \tv
\quad \forall \, \tv \in H^1_\#(\Gamma).
$$
Since $U \in \V_\#(\T)$ is the Galerkin approximation
to $u_\Gamma$ on $\Gamma$, we infer that
$$
\|\nabla_\Gamma(u_\Gamma-U)\| = \inf_{\tV \in \V(\T)} \|\nabla_\Gamma(u_\Gamma-\tV)\|.
$$
This combined with the triangle inequality yields
\[
\|\nabla_\Gamma (\wu\circ\bP_d -U)\|_{L_2(\Gamma)} \le
2 \|\nabla_\Gamma(\wu\circ\bP_d-u_\Gamma)\|_{L_2(\Gamma)}
+  \inf_{\tV \in \V(\T)} \|\nabla_\Gamma(\wu\circ\bP_d-\tV)\|_{L_2(\Gamma)}.
\]
Applying Lemma \ref{L:approxH1} (approximability of $H^1(\Gamma)$), together with
$\| \wu \|_{H^2(\gamma)}\lesssim \| \wf \|_{L_2(\gamma)}$,
readily gives
\[
\inf_{\tV \in \V(\T)} \|\nabla_\Gamma(\wu\circ\bP_d-\tV)\|_{L_2(\Gamma)}
\lesssim \big(h_\T +\beta_\T(\Gamma) \big) \|\wf\|_{L_2(\gamma)}.
\]
To estimate the remaining term, we resort to Lemma \ref{L:norm-equiv} (norm equivalence),
Lemma \ref{L:perturbation_bound_dist}
(perturbation error estimate) along with Corollary \ref{C:lambda-2} (geometric
consistency errors for $C^2$ surfaces) to obtain
\[
\|\nabla_\Gamma(\wu\circ\bP_d - u_\Gamma)\|_{L_2(\Gamma)}
\lesssim \mu_\T(\Gamma) \|F\|_{H_\#^{-1}(\Gamma)} + \|f q_d q_\Gamma^{-1}-F\|_{H_\#^{-1}(\Gamma)},
\]
where $q_d$ denotes the area element induced by the parametrization
$\bchi=\bP_d\circ\bX$ of $\gamma$.
We denote by $\bP_d^{-1}$ the inverse of $\bP_d$ restricted to $\Gamma$, and
use Proposition \ref{l:convolution} (mismatch between $\bP$ and $\bP_d$),
with $\widetilde{w}=\tv \circ \bP_d^{-1}$ and $\tv\in H^1_\#(\Gamma)$, to get
\begin{align*}
\|f q_d \, q_\Gamma^{-1}-F\|_{H^{-1}(\Gamma)} & = \sup_{\|\nabla_\Gamma \tv\|_{L_2(\Gamma)}=1} \int_\Gamma \Big(\wf\circ \bP_d \frac{q_d}{q_\Gamma}-\wf \circ \bP \frac{q}{q_\Gamma}\Big) \tv
\\ &  = \sup_{\|\nabla_\Gamma \tv\|_{L_2(\Gamma)}=1} \int_\gamma \wf \big(\tv \circ \bP_d^{-1} - \tv \circ \bP^{-1} \big)
\lesssim \beta_\T(\Gamma) \|\wf\|_{L_2(\gamma)}.
\end{align*}
Combining the previous inequalities with
$\|F\|_{H^{-1}(\Gamma)} \lesssim \|\wf\|_{L_2(\gamma)}$ completes the proof of the
second assertion. The proof of the first one proceeds along the same lines
but using Lemma \ref{L:perturbation_bound} (perturbation error estimate
for $C^{1,\alpha}$ surfaces)
and Corollary \ref{C:lambda} (geometric consistency for $C^{1,\alpha}$ surfaces)
instead.
\qed

\medskip\noindent
{\it Proof 2.}
We closely mimic the proof of Lemmas \ref{L:perturbation_bound} and \ref{L:perturbation_bound_dist} (perturbation error estimate) for the solution to the Laplace-Beltrami problem on nearby surfaces, with an additional step needed due to the Galerkin approximation. In addition, the fact that $F=\wf\circ\bP \frac{q}{q_\Gamma}$ is defined using the map $\bP$ while all other quantities are lifted using the closest point projection $\bP_d$ adds a twist to our proof as compared with standard proofs of such error estimates. We let $u=\wu\circ\bP_d(\bx)$ for all $\bx\in\Gamma$ for notational convenience, and focus on the second assertion.  

\medskip\noindent
{\it Step 1: Error representation}.
For $V \in \V(\T)$ arbitrary, we let $W=V-U$ 
to arrive at
\begin{align*} 
\| \nabla_\Gamma (V-U)\|_{L_2(\Gamma)}^2 = \int_\Gamma \nabla_\Gamma (u-U)\cdot \nabla_\Gamma W + \int_\Gamma \nabla_\Gamma (V-u) \cdot \nabla_\Gamma W.
\end{align*}
We now invoke Lemma \ref{L:GO} (Galerkin quasi-orthogonality) to rewrite the first
term as follows:
\[
\int_\Gamma \nabla_\Gamma (u-U)\cdot \nabla_\Gamma W
= \int_\Gamma \Big(\wf\circ\bP_d \frac{q_d}{q_\Gamma}-F \Big) W
+ \int_\Gamma \nabla_\Gamma u \cdot \bE_\Gamma \nabla_\Gamma W,
\]
where the area element $q_d$ over $\gamma$ is induced by the parametrization
$\bchi = \bP_d\circ\bX$. We thus have the error representation formula
\begin{align*}
\| \nabla_\Gamma (V-U)\|_{L_2(\Gamma)}^2 & = \int_\Gamma \nabla_\Gamma u \cdot \bE_{\Gamma} \nabla_\Gamma W + \int_\Gamma \nabla_\Gamma (V-u) \cdot \nabla_\Gamma W \\
& + \int_\Gamma \Big(\wf \circ \bP_d \frac{q_d}{q_\Gamma}-F\Big) W := I+II+III,
\end{align*}
and estimate the three terms on the right hand side separately.

\medskip\noindent
{\it Step 2: Geometric and interpolation errors.}
According to Corollary \ref{C:lambda-2} (geometric consistency errors for $C^2$ surfaces),
the error matrix satisfies $\|\bE_\Gamma\|_{L_\infty(\Gamma)} \lesssim \mu_\T(\Gamma)$. This,
together with Lemma \ref{L:norm-equiv} (norm equivalence) and the a priori
bound $\|\nabla_\gamma \wu\|_{L_2(\gamma)} \le \|\wu\|_{H^2(\gamma)} \lesssim \|\wf\|_{L_2(\gamma)}$, yields
\[
I \lesssim \mu_\T(\Gamma) \|\wf\|_{L_2(\gamma)} \|\nabla_\Gamma W\|_{L_2(\Gamma)}.
\]
On the other hand, we can choose $V\in\V(\T)$ so that Lemma \ref{L:approxH1}
(approximability in $H^1(\Gamma)$) holds, whence
\[
II \lesssim \big(h_\T + \beta_\T(\Gamma) \big) \|\wf\|_{L_2(\gamma)}
\|\nabla_\Gamma W\|_{L_2(\Gamma)}.
\]

\medskip\noindent
{\it Step 3: Final estimates}.
We recall that the discrete forcing is given by $F = \wf \circ \bP \frac{q}{q_\Gamma}$,
where $q$ is the area element in $\gamma$ induced by the parametrization
$\bchi = \bP \circ \bX$. Changing variables to $\gamma$ via the lifts
$\bP_d$ and $\bP$ for each integral in $III$ gives
\[
III = \int_\Gamma \Big(\wf \circ \bP_d \frac{q_d}{q_\Gamma}
- \wf \circ \bP \frac{q}{q_\Gamma} \Big) W
= \int_\gamma \wf \Big( W\circ \bP_d^{-1} - W \circ \bP^{-1} \Big),
\]
where again $\bP_d^{-1}$ denotes the inverse of $\bP_d$ restricted to $\Gamma$.
Since $\wf$ has vanishing mean over $\gamma$, we can assume that so does $W$
over $\Gamma$. This allows us to invoke \eqref{poin-unif} (uniform Poincar\'e-Friedrichs
constant) to deduce $\|W\|_{H^1(\Gamma)} \lesssim \|\nabla_\Gamma W\|_{L_2(\Gamma)}$
and thus apply Proposition \ref{l:convolution} (mismatch between $\bP$ and $\bP_d$)
to obtain
\[
III \lesssim  \beta_\T(\Gamma) \| f \|_{L_2(\Gamma)} \|\nabla_\Gamma W\|_{L_2(\Gamma)}.
\]

Collecting the previous estimates, and using that
$\beta_\T(\Gamma)\le\mu_\T(\Gamma)$, leads to
\[
\|\nabla_\Gamma (U-V)\|_{L_2(\Gamma)} \lesssim \big(h_\T + \mu_\T(\Gamma) \big)
\|f\|_{L_2(\Gamma)} \lesssim h_\T \|\wf\|_{L_2(\gamma)}
\]
because $\mu_\T(\Gamma) \lesssim h_\T^2 |d|_{W^2_\infty(\Gamma)}$ according to the
definition \eqref{d:mu} of $\mu_\T(\Gamma)$ and Corollary \ref{C:lambda-2}
(geometric consistency for $C^2$ surfaces). Invoking again Lemma \ref{L:approxH1}
(approximability in $H^1(\Gamma)$) yields the second assertion.
    
The first statement follows similarly upon replacing $\widetilde u\circ\bP_d$ by
$\widetilde u \circ \bP$, $\bP_d$ by $\bP$ and invoking Corollary \ref{C:lambda}
(geometric consistency errors for $C^{1,\alpha}$ surfaces)
$$
\| \bE_{\Gamma} \|_{L_\infty(\Gamma)}  \lesssim \lambda_\T(\Gamma)
\lesssim h_\T |\bP|_{W^2_\infty (\Gamma)}.
$$
This concludes the proof.
\qed

Comparing Corollary \ref{C:lambda} (geometric consistency errors for $C^{1,\alpha}$
surfaces) with Corollary \ref{C:lambda-2} (geometric consistency errors for
$C^{2}$ surfaces) ones sees that using the distance function lift $\bP_d$
for error representation gives rise to a quadratic geometric error estimator for surfaces
$\gamma$ of class $C^2$
\[
\mu_\T(\Gamma)  \lesssim h_\T^2 |d|_{W^2_\infty(\Gamma)} \, ,
\]
even though the FEM is designed in terms of a generic lift $\bP$ also of class $C^2$.
Meanwhile the geometric estimator
$\lambda_\T(\Gamma)\lesssim h_\T |\bP|_{W^2_\infty (\Gamma)}$ is linear for this
regularity class.
This does not affect the $H^1$ a-priori error analysis for piecewise linear
approximations of $\gamma$ and $u$, which is first order, but it is crucial to derive
optimal second-order $L_2$ error estimates by a duality argument. We present next
such estimates for surfaces of class $C^2$ and a FEM based on a generic lift $\bP$ also
of class $C^2$, a result that seems to be new in the literature.

\begin{theorem}[$L_2$ a-priori error estimate for $C^2$ surfaces]\label{t:L2_apriori}
Let $\gamma$ be of class $C^2$ and be described by a generic lift  $\bP$ of class
$C^2$. Let the geometric conditions \eqref{bi_lipschitz},
\eqref{beta-small}, and \eqref{P_Pd:mismatch} be satisfied.
Let $\wu\in H^1_\#(\gamma)$ solve \eqref{e:weak_relax} and $U \in \V_\#(\T)$ solve
\eqref{e:galerkin} with $F= \wf \circ \bP \frac{q}{q_\Gamma}$. Then
\begin{equation}\label{e:L2error}
  \| \wu \circ \bP - U \|_{L_2 (\Gamma)}
  \lesssim h_{\T}^2  \| \wf \|_{L_{2}(\gamma)},
\end{equation}
provided $\lambda \leq \lambda_*$, where $\lambda_*$ is as in Proposition~\ref{l:convolution}.
\end{theorem} 
\begin{proof}
We employ a standard duality argument, but enforcing compatibility (mean-value-zero)
conditions. We use the lift $\bP_d$ and its inverse $\bP_d^{-1}$ when restricted to $\Gamma$ to switch from $\gamma$ to $\Gamma$ back and
forth. To this end we use the notation $\widetilde{w} = w\circ\bP_d^{-1}:\gamma\to\mathbb{R}$
and $\tv = \ttv\circ\bP_d:\Gamma\to\mathbb{R}$ for functions $w:\Gamma\to\mathbb{R}$
and $\ttv:\gamma\to\mathbb{R}$. We denote $q_d$ the area element induced by
$\bP_d$. We finally observe that if $\bP$ is of class $C^2$ then
\[
\beta_\T(\Gamma) \le \mu_\T(\Gamma) \lesssim h_\T^2 |\bP|_{W^2_\infty(\Gamma)},
\]
where $\beta_\T(\Gamma)$ and $\mu_\T(\Gamma)$ are defined in \eqref{d:beta}
and \eqref{d:mu}. We split the proof into several steps.

\medskip\noindent
{\it Step 1: Duality argument.}
We associate with $U\in\V_\#(\T)$ the function $\widetilde{U}_\# = \frac{q_\Gamma}{q} \widetilde U \in L_{2,\#}(\gamma)$ with vanishing mean over $\gamma$ and let $\wz \in H^1_\#(\gamma)$ satisfy
$$
\int_\gamma \nabla_\gamma \wz \cdot \nabla_\gamma  \ww = \int_\gamma \big(\wu - \widetilde{U}_\#\big) \ww
\quad \forall \, \ww \in H^1_\#(\gamma).
$$
Observe that the Lax-Milgram lemma and Lemma \ref{L:Poincare} (Poincar\'e-Friedrichs inequality) guarantee existence and uniqueness of $\wz\in H^1_\#(\gamma)$.
Let also $Z\in\V_\#(\T)$ be the Galerkin approximation to $\wz$ over $\Gamma$, that is
$$
\int_\Gamma \nabla_\Gamma Z \cdot \nabla_\Gamma W =
\int_\Gamma \big( u_\# - U \big) W, \quad \forall \, W \in \V(\T),
$$
where $u_\# := \frac{q}{q_\Gamma} u$
has vanishing mean over $\Gamma$. Note also that $u_\# - U=(\wu-\wU_\#)\circ\bP_d \frac{q}{q_\Gamma}$ is a compatible right-hand side for Theorem \ref{t:H1error} ($H^1$
a-priori error estimate).
We thus have 
$$
\| \wu - \widetilde{U}_\# \|_{L_2(\gamma)}^2  = 
\int_{\gamma} \nabla_\gamma \big(\wu- \widetilde U \big) \cdot \nabla_\gamma (\wz- \widetilde Z)
+  \int_{\gamma} \nabla_\gamma \big(\wu- \widetilde U \big) \cdot \nabla_\gamma \widetilde Z.
$$
Applying Lemma \ref{L:GO} (Galerkin orthogonality) the second integral becomes
\[
\int_{\gamma} \nabla_\gamma \big(\wu- \widetilde U\big) \cdot \nabla_\gamma \widetilde Z
= \int_\gamma \Big(\wf - \widetilde F \frac{q_\Gamma}{q_d}  \Big) \widetilde Z
+ \int_\gamma \nabla_\gamma \widetilde U \cdot \bE \, \nabla_\gamma \widetilde Z,
\]
with $\widetilde F = F\circ\bP_d$. Changing variables first via the
lift $\bP_d$ and next via $\bP$, we get
\[
\int_\gamma \widetilde F \widetilde Z \frac{q_\Gamma}{q_d} = \int_\Gamma F Z
= \int_\gamma F\circ \bP^{-1} ~ Z\circ\bP^{-1} ~ \frac{q_\Gamma}{q} =\int_\gamma \wf \, Z\circ\bP^{-1}.
\]
Consequently, we have derived the following error representation:
\begin{equation}\label{error-rep}
\begin{aligned}
\| \wu - \widetilde{U}_\# \|_{L_2(\gamma)}^2  &= 
\int_{\gamma} \nabla_\gamma \big(\wu- \widetilde U \big) \cdot \nabla_\gamma (\wz- \widetilde Z)
\\
& + \int_\gamma \wf \Big( Z\circ\bP_d^{-1} - Z \circ \bP^{-1} \Big)
\\
& + \int_\gamma \nabla_\gamma \widetilde U \cdot \bE \, \nabla_\gamma \widetilde Z.
\end{aligned}
\end{equation}
The first term is standard and the next two account for the mismatch between
$\bP$ and $\bP_d$ and geometric consistency. We examine them separately now.

\medskip\noindent
{\it Step 2: Bounds.}
Since $\gamma$ is of class $C^2$, Lemma \ref{L:regularity} (regularity)
gives for $z$
\[
\|\wz\|_{H^2(\gamma)} \lesssim \| \wu - \widetilde{U}_\# \|_{L_2(\gamma)}.
\]
Combining Theorem \ref{t:H1error} ($H^1$ a-priori error estimate) for $\wz$
with Lemma \ref{L:norm-equiv} (norm equivalence) yields the following estimate
in $L_2(\gamma)$ instead of $L_2(\Gamma)$
\[
\|\nabla_\gamma (\wz-\widetilde Z)\|_{L_2(\gamma)} \lesssim
h_\T \| \wu - \widetilde{U}_\# \|_{L_2(\gamma)}.
\]
Applying Theorem \ref{t:H1error} again, this time for $u$, implies
\[
\int_{\gamma} \nabla_\gamma \big(\wu- \widetilde U \big) \cdot \nabla_\gamma (\wz- \widetilde Z)
\lesssim h_\T^2 \|\wf\|_{L_2(\gamma)} \| \wu - \widetilde{U}_\# \|_{L_2(\gamma)}.
\]

On the other hand, Proposition \ref{l:convolution} (mismatch between $\bP$
and $\bP_d$) with $\widetilde{w} = Z\circ\bP_d^{-1}$ leads to
\begin{align*}
\int_\gamma \wf \Big( Z\circ\bP_d^{-1} - Z \circ \bP^{-1} \Big) & \lesssim
\beta_\T(\Gamma) \|\wf\|_{L_2(\gamma)} \|Z \circ \bP^{-1}\|_{H^1(\gamma)}
\\
& \lesssim \beta_\T(\Gamma) \|\wf\|_{L_2(\gamma)} \|\nabla_\Gamma Z\|_{L_2(\Gamma)},
\end{align*}
because $Z$ has a zero mean on $\Gamma$. Since $\bP$ is of class $C^2$, one sees
that $\beta_\T(\Gamma)\lesssim h_\T^2 |\bP|_{W^2_\infty(\Gamma)}$. Hence the
a-priori bound $\|\nabla_\Gamma Z\|_{L_2(\Gamma)} \lesssim \|u-U\|_{L_2(\Gamma)}$
implies
\[
\int_\gamma \wf \Big( Z\circ\bP_d^{-1} - Z \circ \bP^{-1} \Big) \lesssim
h_\T^2 |\bP|_{W^2_\infty(\Gamma)} \|\wf\|_{L_2(\gamma)} \|u_\#-U\|_{L_2(\Gamma)}.
\]
Finally, Corollary \ref{C:lambda-2} (geometric consistency error for $C^2$ surfaces),
in conjunction with Lemma \ref{L:norm-equiv} (norm equivalence),
allows us to tackle the geometric error
\[
\int_\gamma \nabla_\gamma \widetilde U \cdot \bE \, \nabla_\gamma \widetilde Z \lesssim
\|\nabla_\Gamma U\|_{L_2(\Gamma)} \|\nabla_\Gamma Z\|_{L_2(\Gamma)} \|\bE\|_{L_\infty(\gamma)}
\lesssim h_\T^2 \|\wf\|_{L_2(\gamma)} \|u_\#-U\|_{L_2(\Gamma)},
\]
where again we have used a priori bounds for $\nabla_\Gamma U$ and $\nabla_\Gamma Z$. Lemma \ref{L:norm-equiv} (norm equivalence) and the nondegeneracy property \eqref{q:nondegen:assume} of $\frac{q}{q_\Gamma}$ imply that $\|u_\#-U\|_{L_2(\Gamma)} \lesssim \|\wu-\widetilde{U}_\#\|_{L_2(\gamma)}$.
Collecting the previous estimates and dividing through by $\|\wu-\widetilde{U}_\#\|_{L_2(\gamma)}$, we thus arrive at
\[
\| \wu - \widetilde{U}_\# \|_{L_2(\gamma)} \lesssim
h_\T^2 \|\wf\|_{L_2(\gamma)}. 
\]

\smallskip\noindent
{\it Step 3: Discrepancy between $\widetilde{U}$ and $\widetilde{U}_\#$ and final estimates.}  We still need to deal with the discrepancy between $\widetilde{U}$ and $\widetilde{U}_\#=\frac{q_\Gamma}{q} \widetilde{U}$.  Using Lemma \ref{C:lambda-2} (geometric consistency errors for $C^2$ surfaces) and Lemma \ref{L:norm-equiv} again, we find that
\[ \|\widetilde{U} -\widetilde{U}_\#\|_{L_2(\gamma)} \le \|1-q_\Gamma q^{-1}\|_{L_\infty(\gamma)} \|\widetilde{U}\|_{L_2(\gamma)} \lesssim h_\T^2 \|\widetilde{f}\|_{L_2(\gamma)}.\]
Applying the triangle inequality followed by Lemma \ref{L:norm-equiv}
gives the intermediate estimate
\[
\| \wu\circ\bP_d - U \|_{L_2(\Gamma)} \lesssim \| \wu - \widetilde U \|_{L_2(\gamma)} \lesssim
h_\T^2 \|\wf\|_{L_2(\gamma)}.
\]
To conclude the proof, we simply note that
\[
\| \wu\circ\bP_d - \wu\circ\bP \|_{L_2(\Gamma)} \approx
\| \wu - \wu\circ\bP_d\circ\bP^{-1} \|_{L_2(\gamma)} \lesssim
\beta_\T(\Gamma) \|\wu\|_{H^1(\gamma)} \lesssim h_T^2 \|\wf\|_{L_2(\gamma)},
\]
according to Proposition \ref{l:convolution} (mismatch between $\bP$ and $\bP_d$)
and the estimate $\beta_\T(\Gamma)\lesssim h_\T^2 \|\bP\|_{W^2_\infty(\Gamma)}$
for $\bP$ of class $C^2$ (see definition \eqref{d:beta} of $\beta_\T(\Gamma)$).
Finally, the triangle inequality leads to the asserted estimate.
\end{proof}

The estimate \eqref{e:L2error} is known for surfaces $\gamma$ of class $C^3$ and
the distance function lift $\bP_d$ \cite{Dz88}.
We insist that \eqref{e:L2error} appears to
be new even for $\bP=\bP_d$ for surfaces of class $C^2$ and is optimal both in
terms of regularity of $u$ and $\gamma$ as well as order.

The $C^2$ regularity of $\gamma$ enters in three distinct places in Step 2
of the proof to tackle the right hand side of \eqref{error-rep}
as well as in Step 3. The first instance is via Lemma \ref{L:regularity}
(regularity) to handle the $H^2$ regularity of both $u$ and $z$ in terms of the
$L_2$ norm of the forcing terms: it turns out that \eqref{regularity} becomes
\[
|\wu|_{H^2(\gamma)} \lesssim |d|_{W^2_\infty(\mathcal{N})} \|\wf\|_{L_2(\gamma)},
\]
whence the factor $|d|_{W^2_\infty(\mathcal{N})}^2$ appears. The same happens with the
term involving $\|\bE\|_{L_\infty(\gamma)}$ in view of \eqref{est-E-EG}, whereas
a factor $|\bP|_{W^2_\infty(\Gamma)}$ shows up for the middle term in \eqref{error-rep}
and the end of the proof due to Proposition \ref{l:convolution} (mismatch
between $\bP$ and $\bP_d$). The complete estimate thus reads
\begin{equation}\label{complete-est}
  \|\wu\circ\bP - U\|_{L_2(\Gamma)} \lesssim
  h_\T^2 |d|_{W^2_\infty(\mathcal{N})}^2 \|\wf\|_{L_2(\gamma)}.
\end{equation}

\medskip\noindent
{\bf A-Priori Error Estimates for $C^{1,\alpha}$ Surfaces.}
We end this section proving $H^1$ error estimates for surfaces $\gamma$
of class $C^{1,\alpha}$ and solutions $\wu$ of class $H^{1+s}(\gamma)$ for $0<s\le 1$.
We recall Lemma \ref{L:regularity-W2p} (regularity for $W^2_p$ surfaces) that establishes
this regularity for $s=1$, provided $n< p \le \infty$, along with
\[
\|\wu\|_{H^2(\gamma)} \lesssim \|\wf\|_{L_2(\gamma)}.
\]
In general, however,
the relation between $\alpha$ and $s$ is not well understood; we refer to
\cite{BDO19} where it is proved the existence of $s=s(\alpha)>0$ such that
$\wu \in H^{1+s}(\gamma)$. We start with a variant of Lemma \ref{L:approxH1}
(approximability in $H^1(\Gamma)$).

\begin{lemma}[approximability in $H^1(\Gamma)$]\label{L:approxH1-C1a}
Let $\gamma$ be a surface of class $C^{1,\alpha}$ and $\wu \in H^{1+s}(\gamma)$,
where $0<s<\alpha<1$ or $0<s\le \alpha=1$.
Then we have
\begin{equation}
\label{opt_approx2}
\inf_{\tV \in \V(\T)} \|\nabla_\Gamma(\wu \circ \bP -\tV)\|_{L_2(\Gamma)}
\lesssim h_\T^s |\wu|_{H^{1+s}(\gamma)}.
\end{equation}
\end{lemma}
\begin{proof}
We recall that $u=\wu\circ\bP$ and $\nabla_\Gamma u \circ \chi_\Gamma = D \bchi_\Gamma \bg_\Gamma^{-1} D\bchi^t \nabla_\gamma \widetilde u \circ \chi$, according to \eqref{tan-grads-lip}, and that $D \bchi_\Gamma, \bg_\Gamma^{-1}$ and $D\bchi$ are uniformly of class $C^{0,\alpha}$; here $\bchi_\Gamma=\bX$. Given $T\in\mathcal T$, a direct calculation using the definition of the seminorm $|\cdot|_{H^s(T)}$ shows that the composition of a Lipschitz map with a $H^s$ function as well as the product of a $C^{0,\alpha}$ function with a $H^s$ function belong to $H^s$ provided $s<\alpha$ or $s\le\alpha=1$. Consequently, we infer that $\nabla_\Gamma u \in H^{1+s}(T)$ for all $T\in\mathcal T$ along with
\[  
| u |_{H^{1+s}(T)} \lesssim | \wu |_{H^{1+s}(\widetilde T)}.
\]
A scaling argument guarantees that the constant hidden in this inequality is independent of $T\in \mathcal T$. We next apply the localized interpolation estimate of
Vesser \cite{Vee15} to deduce
\[
\inf_{V \in \mathbb V(\mathcal T)} \| \nabla_\Gamma (u - V)\|_{L_2(\Gamma)}^2 \lesssim 
\sum_{T\in \mathcal T} \inf_{V \in \mathbb V(T)} \| \nabla_\Gamma (u - V)\|_{L_2(T)}^2
\lesssim h_{\mathcal T}^{2s} |\wu|_{H^{1+s}(\gamma)}^2,
\]
which is the asserted estimate.
\end{proof}

We now compare Lemma \ref{L:approxH1-C1a} with Lemma \ref{L:approxH1} (approximability in $H^1(\Gamma)$). We stress that the lift $\bP=\bchi\circ\bX^{-1}$ is of class $C^{1,\alpha}$ for surfaces of class $C^{1,\alpha}$, whereas the distance lift $\bP_d$ is just of class $C^1$ for surfaces of class $C^2$. This is why the proof of Lemma \ref{L:approxH1-C1a} is considerably simpler than that of Lemma \ref{L:approxH1}. The virtue of $\bP_d$ is reflected in a higher order geometric error $\mu_\T(\Gamma)$ in Theorem \ref{t:H1error} ($H^1$ a-priori error estimate for $C^2$ surfaces) relative to the next $H^1$ error estimate. This is also responsible for the optimal Theorem \ref{t:L2_apriori} ($L_2$ a-priori error estimate for $C^2$ surfaces) which does not have a counterpart in this context.

\begin{theorem}[$H^1$ a-priori error estimate for $C^{1,\alpha}$ surfaces] \label{t:H1errorC1a}
  Let $\gamma$ be of class $C^{1,\alpha}$, $0<\alpha\leq 1$,  and assume that the geometric assumptions \eqref{bi_lipschitz},
\eqref{beta-small}, and \eqref{P_Pd:mismatch} are valid.
 Let $\wf \in L_{2,\#}(\gamma)$ and $\wu\in H^{1+s}(\gamma)$ be the solution of \eqref{e:weak} and satisfy
\[
\|\wu\|_{H^{1+s}(\gamma)} \lesssim \|\wf\|_{L_2(\gamma)},
\]
provided $0<s<\alpha<1$ or $0<s\le\alpha=1$.
If $U \in \V_\#(\T)$ is the solution to \eqref{e:galerkin} with $F =\wf \circ \bP \frac{q}{q_\Gamma}$ defined via the lift $\bP$, then
$$
\| \nabla_\Gamma (\wu \circ \bP -U)\|_{L_2 (\Gamma)}\lesssim h_\T^s \|\wu\|_{H^{1+s}(\gamma)}
+ \lambda_\T(\Gamma) \| \wf \|_{L_{2}(\gamma)} \lesssim  h_\T^s \| \wf \|_{L_{2}(\gamma)}.
$$
\end{theorem}
\begin{proof}
  We proceed along the lines of Proof 1 of Theorem \ref{t:H1error} ($H^1$ a-priori error estimate for $C^2$ surfaces), which splits the error into an approximation and a perturbation term. For the former we simply resort to Lemma \ref{L:approxH1-C1a} instead of
  Lemma \ref{L:approxH1} (approximability in $H^1(\Gamma)$). For the latter we argue exactly as in Theorem \ref{t:H1error} and thus employ \eqref{poin-unif} (uniform Poincar\'e-Friedrichs constant), Lemma \ref{L:perturbation_bound} (perturbation error estimate for $C^{1,\alpha}$ surfaces) and Corollary \ref{C:lambda} (geometric consistency for $C^{1,\alpha}$ surfaces). This shows the first asserted estimate.
The second bound follows from the standard interpolation estimate
$$
\lambda_\T(\Gamma) \lesssim h_\T^{\alpha} \| \bchi \|_{C^{1,\alpha}(\mathcal{V})}
$$
and the condition $\alpha \ge s$. This ends the proof.
\end{proof}

\subsection{A-Posteriori Error Analysis}\label{S:a-posteriori}

In contrast to the previous section, we now derive error estimates in $H^1$
which rely on information extracted from the computed solution $U$ of \eqref{e:galerkin}
and data, but do not make use of the exact solution $\wu$ of \eqref{e:weak}. They are
{\it a-posteriori} estimates of residual type, are fully computable, and are
instrumental to drive adaptive procedures. In this vein, we mention
\cite{BCMMN16,BCMN:Magenes} but we do not elaborate on this issue any longer.

The a-posteriori analysis requires a quasi-interpolation operator acting on
$H^1(\Gamma)$ functions, i.e. functions without point values. 
We use the Scott-Zhang operator $\interp^{\textrm{sz}}:H^1(\Gamma) \rightarrow \V(\T)$ and recall its local approximability and stability properties for all $T\in\T$
\begin{equation}\label{e:sz_interp}
\| \tv- \interp^{\textrm{sz}} \tv \|_{L^2(T)} \lesssim  h_T \| \nabla_\Gamma \tv \|_{L^2(\omega_T)}, \quad \| \nabla_\Gamma \interp^{\textrm{sz}} \tv\|_{L^2(T)} \lesssim  \| \nabla_\Gamma \tv \|_{L^2(\omega_T)} ,
\end{equation}
where $\omega_T$ is a macro patch defined in \eqref{patch} associated with $T$.
We do not require that $\interp^\textrm{sz} \tv\in\V_\#(\T)$ even if
$\tv\in H^1_\#(\Gamma)$, as it happened earlier in the a-priori error analysis
of Section \ref{S:a-priori}.

In order to derive a posteriori error estimates, we first introduce the
{\it interior} and {\it jump residuals} for any $V\in \V(\T)$:
\begin{align*}
R_T(V) &:= F\mid_T +\Delta_\Gamma V  \mid_T\quad\forall \, T\in {\T}\\
J_S(V) &:= \nabla_\Gamma V^+\mid_S \cdot \bmu^+_S +\nabla_\Gamma V^-\mid_S \cdot \bmu^-_S \quad \forall \, S\in S_{\T}
\end{align*}
where for $S=\overline{T}^+\cap \overline{T}^-$ is the face shared by $T^\pm\in\T$
and $\bmu^\pm_S:= \bmu_{T^\pm}$  are pointing outward co-normals to the elements $T^\pm$
(see Section~\ref{S:diver-thm}).
We point that when using piecewise affine functions $V=\widehat{V}\circ\bX^{-1}$
on polyhedral surfaces
$\Gamma$, the Laplace-Beltrami operator \eqref{lap-bel-def} vanishes within elements 
$$
\Delta_\Gamma V = \frac 1 {q_\Gamma}  \div{q_{\Gamma}\bg_{\Gamma}^{-1}\nabla \widehat V} = 0
\quad\forall \, T\in\T,
$$
and that, in contrast to the flat case, $\bmu^+_S\ne -\bmu^-_S$ in general.
If $J_{\partial T}(V)$ denotes the jump residual on $\partial T$, then
we define the {\it element indicator} to be
\[
\eta_{\T}(V,T)^2 := h_T^2 \| R_T(V)\|_{L^2(T)}^2
+  h_T \| J_{\partial{T}}(V) \|^2_{L^2(\partial{T})}
\quad\forall \, T\in\T,
\]
and the {\it error estimator} to be
\[
\eta_{\T}(V)^2 := \sum_{T\in \T}\eta_{\T}(V,T)^2.
\]

\begin{theorem}[a-posteriori upper bound for $C^{1,\alpha}$ surfaces]\label{t:posteriori_generic}
  Let $\gamma$ be of class $C^{1,\alpha}$, be parametrized by $\chi = \bP \circ \bX$
  and satisfy the geometric assumption~\eqref{bi_lipschitz}.
Let $\wu \in H^1_\#(\gamma)$ be the solution to \eqref{e:weak} and $U \in \V_\#(\T)$ be the solution to \eqref{e:galerkin} with $F = \wf \circ \bP { q\over q_\Gamma} \in L_{2,\#}(\Gamma)$.
Then, for $\widetilde{U}:=U\circ\bP^{-1}:\gamma\to\mathbb{R}$ we have
$$
\| \nabla_\gamma (\wu- \widetilde U) \|_{L^2(\gamma)}^2 \lesssim \eta_{\T}(U)^2
+ \lambda^2_{\T}(\Gamma) \| \wf \|_{L_2(\gamma)}^2.
$$
\end{theorem}
\begin{proof}
Using definitions \eqref{e:weak_relax} and \eqref{e:galerkin_relax},
along with the consistency relation \eqref{consistency},
enables us to write for any $\widetilde{\tv}\in H^1(\gamma),
\tv=\widetilde\tv\circ\bP\in H^1(\Gamma)$ and $V\in \mathbb V(\T)$
\begin{equation}\label{e:post_split}
\int_\gamma \nabla_\gamma (\wu-\widetilde U) \cdot \nabla_\gamma \widetilde\tv = I_1+I_2+I_3
\end{equation}
with
\begin{align*}
I_1 & = -\int_\Gamma \nabla_\Gamma U \cdot \nabla_\Gamma (\tv-V) +\int_\Gamma F(\tv-V),
\\
I_2 & = \int_\gamma \nabla_\gamma \widetilde{U} \cdot \bE \, \nabla_\gamma \widetilde{\tv},
\\
I_3 & = \int_\gamma \wf \, \widetilde\tv-\int_\Gamma F \tv.
\end{align*}
Employing the definition $F =  \wf\circ\bP { q \over q_\gamma }$ and changing
variables we deduce $I_3 = 0$.

On the one hand, decomposing $I_1$ over elements $T\in \T$, and resorting to
Corollary \ref{C:int-parts} (integration by parts) on $T$, leads to
\begin{equation}\label{e:I1}
I_1 = \sum_{T\in \T}\int_T R_T(U)(\tv-V) + \sum_{S \in\mathcal{S}}\int_{S} J_S(U)(\tv-V)
\end{equation}
and so
$$
I_1 \lesssim \sum_{T\in \T} \eta_{\T}(U,T)\left( h^{-1}_T \| \tv-V\|_{L^2(T)} + \|\nabla_\Gamma (\tv-V) \|_{L^2(T)}\right),
$$
because of the scaled trace inequality
\[
\|w\|_{L_2(\partial T)} \lesssim h_T^{-\frac12} \|w\|_{L_2(\partial T)}
+ h_T^{\frac12} \|\nabla_\Gamma w\|_{L_2(\partial T)}
\quad\forall \, w\in H^1(T).
\]
We now choose $V= \interp^{\textsc{sz}} \tv$ to be the Scott-Zhang
quasi-interpolant of $\tv$.
The local approximability and stability properties \eqref{e:sz_interp} imply
\begin{equation}\label{e:post_I1}
  I_1\lesssim \eta_{\T}(U) \| \nabla_\Gamma \tv \|_{L^2(\Gamma)}
  \lesssim \eta_{\T}(U) \| \nabla_\gamma \widetilde \tv \|_{L^2(\gamma)},
\end{equation}
where we have used the finite ovelapping properties of the patches $\{\omega_T\}_{T\in\T}$
and Lemma \ref{L:norm-equiv} (norm equivalence).
Regarding term $I_2$ we apply Corollary \ref{C:lambda} (geometric consistency errors
for $C^{1,\alpha}$ surfaces) to arrive at
$$
I_2 \lesssim \lambda_\T(\Gamma) \,
\| \nabla_{\gamma}\widetilde U \|_{L^2(\gamma)} \| \nabla_\gamma \widetilde \tv \|_{L^2(\gamma)}
\lesssim \lambda_\T(\Gamma) \, \| \wf \|_{L^2(\gamma)}
\, \| \nabla_\gamma \widetilde \tv \|_{L^2(\gamma)},
 $$
because of the estimates
$$
\| \nabla_{\gamma}\widetilde U \|_{L^2(\gamma)} \lesssim
\| \nabla_{\Gamma}  U \|_{L^2(\Gamma)}  \lesssim \| F \|_{H^{-1}_\#(\Gamma)}
\lesssim \| \wf \|_{L^2(\gamma)}
$$
which are a consequence of Lemma \ref{L:norm-equiv} (norm equivalence), $F = \wf\circ\bP \frac{q}{q_\Gamma}$, Lemma~\ref{L:Poincare} (Poincar\'e-Friedrich inequality) and $\int_\gamma \widetilde f = 0$.
Combining the above estimates, we end up with the assertion.
\end{proof}

To assess the tightness of the upper bound in Theorem \ref{t:posteriori_generic}
it is customary to show a lower bound. To this end, we introduce the so-called
{\it data oscillation}
\[
\mathrm{osc}_{\T}(F,T)^2 := h_T^2 \| F - \overline{F} \|_{L^2(T)}^2,
\quad
\mathrm{osc}_{\T}(F)^2:= \sum_{T\in \T} \mathrm{osc}_{\T}(F,T)^2,
\]
where $\overline F$ is the piecewise average of $F$. This quantity
accounts for the fact that the residual is evaluated in a weighted
$L_2$-norm rather than the natural $H^{-1}$-norm. This in turn makes the estimator
$\eta_\T(U)$ computable but perhaps at the expense of overestimation. This is
the subject of our next estimate, proved in \cite{BCMN:Magenes}.
We recall that for $T\in \mathcal T$, $\omega_T$ denotes the union of elements
in $\T$ that intersect $T$ and $\widetilde \omega_T$ stands for the lift of
$\omega_T$  to $\gamma$ via $\bP$. Moreover, we set
\[
\osc_{\T}(F,\omega_T)^2 := \sum_{T'\subset\omega_T} \osc_{\T}(F,T')^2,
\qquad
\lambda^2_\T(\omega_T) := \max_{T'\subset\omega_T} \lambda^2_T.
\]

\begin{theorem}[a-posteriori lower bound for $C^{1,\alpha}$ surfaces]\label{T:apost-lower}
Under the same conditions of Theorem \ref{t:posteriori_generic} (a-posteriori
upper bound for $C^{1,\alpha}$ surfaces), we have
\[
\eta_{\T}(U,T)^2 \lesssim \| \nabla_\gamma (\widetilde u- \widetilde U) \|_{L^2(\widetilde{\omega}_T)}^2
+ \osc_{\T}(F,\omega_T)^2 + \lambda^2_\T(\omega_T).
\]
\end{theorem}  
\begin{proof}
The proof of the lower bound is standard and is only sketched here.
It relies on an argument due to Verf\"urth \cite{MR3059294}.
The starting point is the error relation~\eqref{e:post_split} localized to $T\in \mathcal T$ via the test function $\tv = \overline{F} b_T$, where $b_T \in H^1_0(T)$ is the cubic bubble taking value $1$ at the element barycenter. Employing the norm equivalence \eqref{H1:equiv} (valid elementwise), we realize that
$$
\| \nabla_\gamma \ttv \|_{L_2(\widetilde{T})} \lesssim \| \nabla_\Gamma \tv \|_{L_2(T)} \lesssim h_T^{-1} \| \overline{F}\|_{L_2(T)},
$$
whence taking  $V=0$  in \eqref{e:post_split} yields
$$
\| \overline{F} \|_{L_2(T)}^2 \lesssim \int_T \overline{F} \tv \lesssim h_T^{-1} \left( \| \nabla_\gamma (\wu - \widetilde U)\|_{L_2(\widetilde T)} + \osc_\T(F,T) + \|\bE\|_{L_\infty(\widetilde T)} \right) \| \overline{F}\|_{L_2(T)}
$$
upon recalling that $I_3=0$ with our choice of $F$ and the expression \eqref{e:I1} for $I_1$.
Corollary \ref{C:lambda} (geometric consistency errors for $C^{1,\alpha}$ surfaces),
combined with a triangle inequality, then leads to the desired estimate for the bulk term
$$
h_T^2\| F \|_{L^2(T)}^2 \lesssim \| \nabla_\gamma (\wu - \widetilde U)\|_{L_2(\widetilde T)}^2 + \mathrm{osc}_{\T}(F,T)^2 + \lambda_T^2. 
$$
As for the jump term, we define for a side $S \in \mathcal S$ with adjacent elements $T^\pm$,  $b_S \in H^1_0(\omega_S)$ as the  quadratic bubble taking value $1$ at the barycenter of $S$ and $0$ at all other quadratic nodes in $\omega_S:= T^+\cup T^-$.
We also let $\widetilde{\omega}_S:= \bP(\omega_S)$ be the lift of $\omega_S$ to
$\gamma$ by the map $\bP$.
Taking $\tv=J_S(U) b_S$ and $V=0$ in \eqref{e:post_split}, and recalling the expression \eqref{e:I1} for $I_1$ and that $I_3=0$, yields
\begin{align*}
&\| J_S(U)\|_{L_2(S)}^2  \lesssim \int_S J_S(U)\tv \\
&\quad \lesssim  \left( \| \nabla_\gamma (\widetilde u - \widetilde U)\|_{L_2(\widetilde{\omega}_S))} + h_T \| F\|_{L_2(\omega_S)} + \max(\lambda_{T^+},\lambda_{T^-}) \right) \| \nabla_\gamma \ttv\|_{L_2(\widetilde{\omega}_S))}.
\end{align*}
Finally, it suffices to use the preceding estimate for $h_T \| F\|_{L_2(T)}$,
together with 
$$
\| \nabla_\gamma \ttv\|_{L_2(\widetilde{\omega}_S))} \lesssim \| \nabla_\Gamma \tv\|_{L_2(\omega_S)} \lesssim h_T^{-1/2} \| J_S(U)\|_{L_2(S)},
$$
to conclude the proof.
\end{proof}  

One important observation to make is that $\mathrm{osc}_\T(F)$ is generically
of higher order than $\eta_{\T}(U)$ for $\wf\in L_2(\gamma)$, whence
this term can be ignored relative to $\eta_{\T}(U)$ asymptotically. However,
the geometric estimator $\lambda_\T(\Gamma)$ is linear and thus of the same
order as $\eta_{\T}(U)$, thereby making the lower bound of Theorem \ref{T:apost-lower}
questionable. This estimator comes from the estimate \eqref{e:estim_consist}
of Corollary \ref{C:lambda} (geometric consistency errors for $C^{1,\alpha}$ surfaces), which cannot obviously be improved for surfaces
of class $C^{1,\alpha}$. However, Corollary \ref{C:lambda-2}
(geometric consistency errors for $C^2$ surfaces) shows that this 
effect becomes of second order for surfaces of class $C^2$. Practically, the estimator
$\lambda_\T(\Gamma)$ is pessimistic and leads
to unnecessary and thus suboptimal refinements for $C^2$ surfaces \cite{BD:18}.
We discuss the impact of this superconvergence estimate next following
\cite{BD:18}.

\begin{theorem}[a-posteriori upper bound for $C^2$ surfaces]\label{T:apost-upper-dist} 
Let $\gamma$ be of class $C^2$ and \eqref{e:shape_reg_init},
\eqref{q:nondegen:assume}, \eqref{beta-small},  and \eqref{P_Pd:mismatch} hold.
Let $\wu$ be the solution of \eqref{e:weak} with $\wf \in L_{2,\#}(\gamma)$ and
$U \in \V(\T)$ be the solution to \eqref{e:galerkin} with
$F = \wf \circ \bP \frac{q}{q_\Gamma}$, where $q$ corresponds to the parametrization $\bchi = \bP \circ \bX$ of $\gamma$. Then
$$
\| \nabla_\gamma (\wu - U \circ \bP_d^{-1}) \|_{L_2(\gamma)}^2 \lesssim \eta_\T(U)^2
+ \mu_\T^2(\Gamma) \, \| \wf \|_{L_2(\gamma)}^2.
$$
\end{theorem}
\begin{proof}
We proceed as in the proof of Theorem~\ref{t:posteriori_generic} (a-posteriori upper
bound for $C^{1,\alpha}$ surfaces) but using the distance function lift to represent
the errors. We denote
$\widetilde{U} = U \circ \bP_d^{-1}$, $\tv = \widetilde\tv \circ \bP_d$ for a generic
$\widetilde\tv\in H^1(\gamma)$ and get for any $V\in \mathbb V(\T)$
\begin{equation}\label{e:error_rep_dist}
\int_\gamma \nabla_\gamma \big(\wu- \widetilde{U}\big) \cdot \nabla_\gamma \wv
= I_1+I_2+I_3
\end{equation}
with
\begin{align*}
I_1 & = -\int_\Gamma \nabla_\Gamma U \cdot \nabla_\Gamma (v-V) +\int_\Gamma F(v-V),
\\
I_2 & = \int_\gamma \nabla_\gamma \widetilde{U} \cdot \bE \, \nabla_\gamma \wv,
\\
I_3 & = \int_\gamma \wf \, \wv - \int_\Gamma F \tv,
\end{align*}
where we have used again \eqref{consistency} but with the error matrix $\bE$ now defined with respect to $\bP_d$ and 
given by \eqref{error-matrix-gamma} of Lemma \ref{L:geom_consist_dist}
(geometric consistency). We tackle $I_1$ and $I_2$ exactly as in
Theorem \ref{t:posteriori_generic}, thereby obtaining
\[
I_1 \lesssim \eta_\T(U) \| \nabla_\gamma \ttv \|_{L_2(\gamma)},
\quad
I_2 \lesssim \mu_\T(\Gamma) \, \| \wf \|_{L^2(\gamma)} \| \nabla_\gamma \widetilde \tv \|_{L_2(\gamma)},
\]
except that we resort to \eqref{est-E-EG} of Corollary \ref{C:lambda-2} (geometric
consistency errors for $C^2$ surfaces) to estimate $\bE$.

On the other hand, $I_3$ no longer vanishes because
$F=\wf\circ\bP \frac{q}{q_\Gamma}$ is defined via $\bP$
and the function $\tv$ via $\bP_d$. Using $\bP$ to change variables back to
$\gamma$ we obtain
\[
\int_\Gamma F \tv = \int_\Gamma (\wf\circ\bP) ~ (\wv\circ\bP_d) \frac{q}{q_\Gamma}
= \int_\gamma \wf ~ (\wv\circ\bP_d \circ\bP^{-1}),
\]
whence $I_3$ becomes
\begin{equation}\label{e:I3-dist}
I_3 = \int_\gamma \wf ~ \big( \wv  -  \wv \circ \bP_d \circ \bP^{-1} \big). 
\end{equation}
Invoking Proposition \ref{l:convolution} (mismatch between $\bP$ and $\bP_d$) yields
\[
I_3 \lesssim \beta_\T(\Gamma) \, \| \wf \|_{L_2(\gamma)}
\| \nabla_\gamma \wv \|_{L_2(\gamma)}
\]
and concludes the proof because $\beta_\T(\Gamma) \le \mu_\T(\Gamma)$.
\end{proof}

We conclude with a lower bound for $C^2$ surfaces.
We point out that, compared with the existing results in the literature, see e.g. \cite{DemlowDziuk:07}, we account for the mismatch between the two lifts $\bP$ and $\bP_d$. 

\begin{theorem}[a-posteriori lower bound for $C^{2}$ surfaces]\label{T:apost-lower-dist}
Under the same conditions as Theorem \ref{T:apost-upper-dist} (a-posteriori
upper bound for $C^{2}$ surfaces), we have
\[
\eta_{\T}(U,T)^2 \lesssim \| \nabla_\gamma (\wu- \widetilde U) \|_{L^2(\widetilde{\omega}_T)}^2
+ \osc_\T(F,\omega_T)^2 + \mu_\T(\omega_T)^2,
\]
where $\mu_\T(\omega_T) = \max_{T' \subset \omega_T} \mu_{T'}$.
\end{theorem}  
\begin{proof}
The proof follows along the lines of Theorem~\ref{T:apost-lower} (a-posteriori lower bound for $C^{1,\alpha}$ surfaces) with the following variants.
We use Corollary \ref{C:lambda-2} instead of Corollary \ref{C:lambda} in
the error representation~\eqref{e:error_rep_dist} to tackle $I_2$
and account for the fact that $I_3\ne0$ via~\eqref{e:I3-dist} for a generic
lift $\bP$.
\end{proof}


\section{Trace Method}\label{sec:trace}

In this section we present a class of methods which are known as {\it trace finite element methods} or {\it cut finite element methods} \cite{ORG09, BHL15, Re15}.  The setting for these methods is situations in which a PDE posed on an $n$-dimensional hypersurface $\gamma$ embedded in $\mathbb{R}^{n+1}$ must be solved numerically, and a bulk or volume background mesh of some domain $\Omega \subset \mathbb{R}^{n+1}$ is present with $\gamma \subset \Omega$.  It is often more convenient to describe $\gamma$ and solve associated PDE employing the background mesh instead of independently meshing $\gamma$.  A paradigm physical example is a two-phase flow problem.  There $\Omega$ is subdivided into subdomains $\Omega_1$ and $\Omega_2$ (one for each phase) and $\gamma$ is the interface between $\Omega_1$ and $\Omega_2$.  In simulations $\Omega$ is typically meshed in order to solve equations of fluid dynamics (e.g., Stokes or Navier-Stokes), while accounting for interfacial effects such as surface tension also requires solving a surface PDE on $\gamma$.  It can be particularly inconvenient to independently mesh $\Omega$ and $\gamma$ in dynamic simulations in which $\gamma$ evolves as either a specified or free boundary.   In addition to the overhead associated with transferring information between unrelated bulk and surface meshes,  remeshing is generally necessary from time to time when parametric methods are used to describe dynamic interfaces because mesh degeneracies may occur as the surface deforms.

Trace and cut FEMs were introduced by Olshanskii et al \cite{ORG09} and have been further developed over the past decade as one option for circumventing these difficulties.  In order to describe them more precisely, first let $\T:=\T_\Omega$ be a  simplicial decomposition of $\Omega \subset \mathbb{R}^{n+1}$, $n\ge1$.
We let $h_T:=| T |^{\frac 1 {n+1}}$ for any $T\in\T$ and set $h:=\max_{T\in \mathcal T} h_T$ for the mesh-size of $\T$.
We will omit to mention the explicit dependence on the shape regularity constant of $\T$
$$
\sigma := \max_{T\in \mathcal T} \frac{\textrm{diam}(T)}{h_T}
$$
in most estimates below.
Assume that $\gamma \subset \Omega$ is a closed, $C^2$ $n$-dimensional surface.  As outlined in Section \ref{S:distance-function}, $\gamma$ is then the zero level set of a $C^2$ distance function $d$ defined on a tubular neighborhood $\mathcal{N}$ of $\gamma$. Let $\V(\mathcal T) \subset H^1(\Omega)$ consist of the continuous piecewise linear functions over $\T$.   In order to fix thoughts, let $d_h \in \V(\mathcal T)$ be the Lagrange interpolant $\interp d$ of $d$ satisfying
\[
\|d-d_h\|_{L_\infty(\mathcal{N})} + h \|d-d_h\|_{W_\infty^1(\mathcal{N})}
\lesssim h^2 |d|_{W^2_\infty(\mathcal{N})}.
\]
The discrete computational surface $\Gamma$ is then defined by 
$$
\Gamma := \{\bx \in \Omega : d_h(\bx)=0\}.
$$
Below we also discuss how to derive $\Gamma$ from more general implicit representations of $\gamma$. 
Because $d_h$ is piecewise linear, $\Gamma$ consists of intersections of hyperplanes with simplices and is thus a polyhedron having triangular and quadrilateral faces for $n=2$ (see Figure \ref{f:tracemeshes}).
We denote by $\mathcal{F}$ the collection of faces of $\Gamma$.
In addition, the conditions placed on $d_h$ ensure that $\|d\|_{L_\infty(\Gamma)} + h \|\bnu-\bnu_\Gamma\|_{L_\infty(\Gamma)} \lesssim h^2,$
so the perturbation results for $C^2$ surfaces outlined in Section \ref{S:perturb-C2} hold on $\Gamma$ with order $h^2$ geometric perturbation error.

The surface finite element space $\V(\mathcal F)$ is simply the restriction of $\V(\mathcal T)$ to $\Gamma$:  
\[
\V(\mathcal F) := \big\{ V|_\Gamma:  V\in \V(\mathcal T)    \big\}.
\]
By its definition $\V(\mathcal F)\subset H^1(\Gamma) $ consists of the continuous functions which are affine over each face $F \in \mathcal{F}$.
We also denote by $\V_\#(\mathcal F) := \V(\mathcal F) \cap L_{2,\#}(\Gamma)$ its subspace consisting of functions with vanishing mean values.
In order to approximate the solution $\wu$ to the Laplace-Beltrami problem $-\Delta_\gamma \wu =\widetilde f$ on $\gamma$, we first define a suitable approximation $F_\Gamma$ to $f$ and then seek $U \in \V_\#(\mathcal F)$ such that 
\begin{align}
\label{def:trace_fem}\int_\Gamma \nabla_\Gamma U \cdot \nabla_\Gamma V = \int_\Gamma F_\Gamma V, \qquad \forall V \in \V_\#(\mathcal F).
\end{align}

\begin{figure}[ht!]
\centerline{\includegraphics[width=0.51\textwidth]{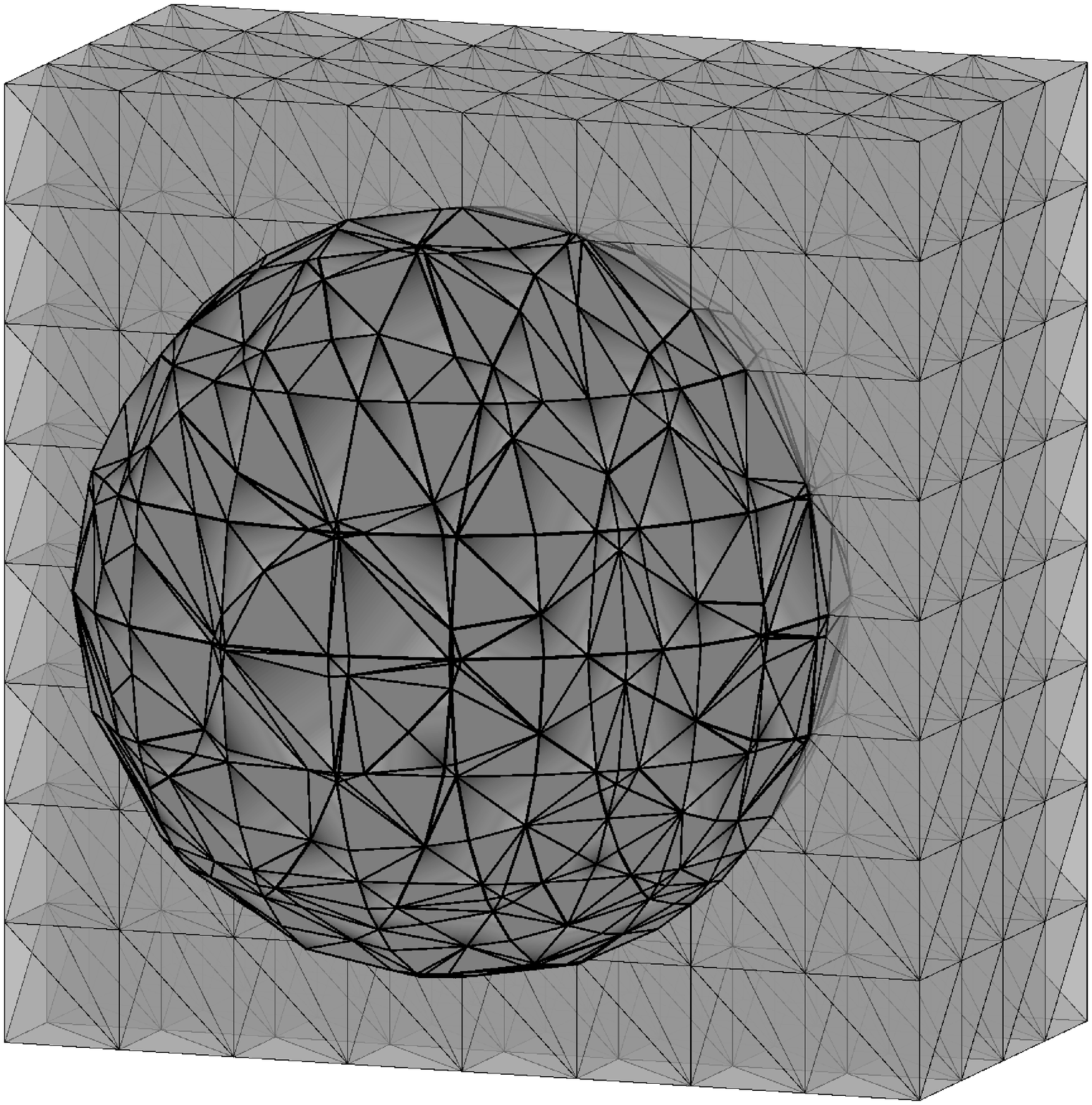}
\includegraphics[width=0.46\textwidth]{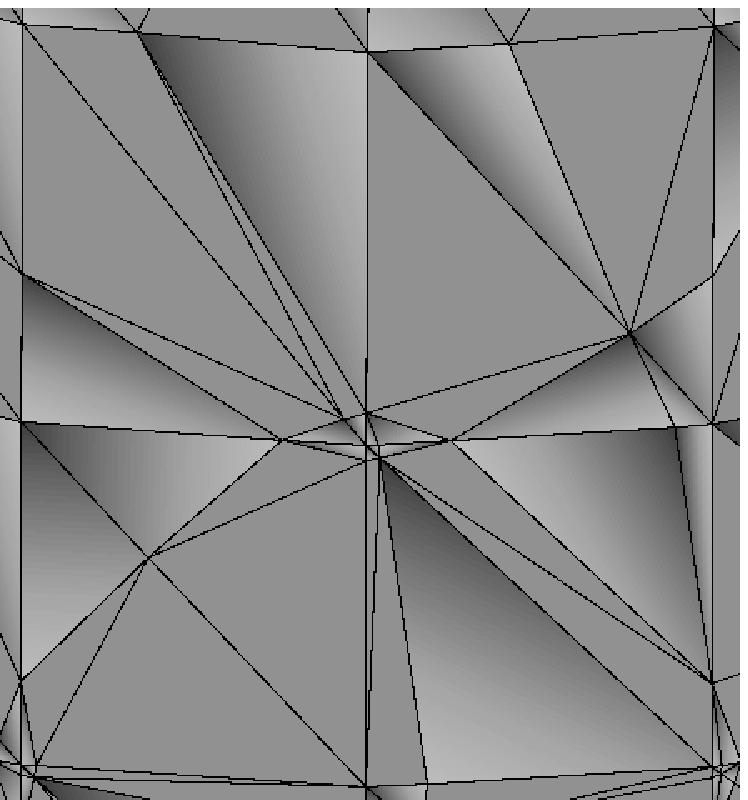}}
\caption{Bulk mesh cutaway with associated trace mesh (left); blowup of a trace mesh showing small and narrow elements (right).  
} \label{f:tracemeshes}
\end{figure}

This is the trace method and has two notable advantages:

\begin{itemize}
\item
{\it Only single mesh}: The main advantage is that both bulk and
interfacial effects can be computed using the same mesh.

\item
{\it Error estimates}: 
Optimal-order and regularity error estimates hold in the $H^1$ and $L_2$ norms.
\end{itemize}

\noindent
On a practical and theoretical levels the method exhibits three main challenges:

\begin{itemize}
\item
  {\it Implicit surface representation}: The simplest option of taking the distance function $d$ to define $\gamma$ and its Lagrange interpolant of $d_h$ to define $\Gamma$ is not generally practical as $d$ is rarely available in applications. It is generally more practical to assume that the discrete surface $\Gamma$ is derived from a more general level set representation $\phi$ of $\gamma$.   We provide a brief discussion of general level set representations below.  
\item
{\it Surface integration}: Computing the finite element system is more cumbersome than in standard parametric surface FEMs since both the mesh $\mathcal{F}$ and the finite element space $\V(\mathcal F)$ are derived from their corresponding bulk counterparts. These difficulties are manageable in the case of the piecewise linear method presented here, but become significantly more cumbersome when a higher-order surface approximation is used.
\item
{\it Linear algebra and stabilization}: In contrast to parametric surface FEMs there is no obvious practical basis for $\V(\mathcal F)$, only spanning sets derived from subsets of the bulk space $\V(\mathcal T)$.  In practice such a spanning set is derived from the degrees of freedom for $\V(\mathcal T)$ corresponding to elements touching $\Gamma$.  Degenerate modes arise from this procedure.  These are either handled at the linear algebra level or by various stabilization procedures.

\end{itemize}

Theoretical study of trace FEMs is also more involved than for parametric surface FEMs.   One prominent issue is that the surface mesh $\mathcal{F}$ does not consist of shape regular elements, as is documented in Figure \ref{f:tracemeshes}.  This is because the faces in $\mathcal{F}$ consist of {\it arbitrary} intersections of hyperplanes with simplices (planes and tetrahedra for $n=2$).  Thus elements may be arbitrarily small with respect to the bulk mesh-size $h$ or fail to satisfy a minimum angle condition, and it is not possible to directly employ standard error estimation techniques.  Properties of the ``high-quality'' bulk mesh $\T$ and finite element space $\V(\mathcal T)$ must be invoked instead, which in turn requires careful use of extensions and restrictions of functions to and from $\gamma$ and $\Gamma$.   For purposes of intuition, it is however useful to note that the surface mesh $\mathcal{F}$ does inherit some structure from the regularity of the bulk mesh $\T$.  Elements in $\mathcal{F}$ for example satisfy a maximum-angle condition \cite{ORX12}, and each element in $\mathcal{F}$ also shares a vertex with a shape-regular element of diameter equivalent to $h$ \cite{DO12}.

Below we prove a priori and a posteriori error estimates for a piecewise linear trace FEM.  In keeping with the previous section, we concentrate on surface representations and regularity in our discussion.  In particular, we only assume that $\gamma$ is $C^2$, whereas previous approaches require that $\gamma$ be $C^3$.  The recent article \cite{OR17} provides a broader survey of trace FEMs, including discussion of topics such as higher-order versions, stabilization procedures, and space-time trace FEMs that we omit here.  

\subsection{Preliminaries}\label{S:trace-prelim}

{\bf Bulk and Surface Meshes.}
Below we need to carefully distinguish between mesh structures defined relative to the surface mesh $\mathcal{F}$ and those defined relative to the volume mesh $\T$.  First note that we shall consistently denote by $F$ ($n$-dimensional) surface elements lying in $\mathcal{F}$, which as we have noted above may not be shape-regular.  In addition, $T$ will be used to denote $(n+1)$-simplices lying in $\T$.  Given a face $F \in \mathcal{F}$, we denote by $T_F$ the simplex in which $F$ lies (or one of them if $F$ is a face shared by two bulk elements).  In addition, given $T \in \T$ we denote by $\omega_\T^1(T)$ the patch of elements of $\T$ surrounding $T$ (first ring)
\[
\omega_\T^1(T) := \bigcup \big\{T'\in\T: ~~ T'\cap T \neq \emptyset  \big\},
\]
and by $\omega_\T^2(T)$ the patch of elements of $\T$ surrounding $\omega_\T^1(T)$ (second ring). We also define 
$$h_F={\rm diam}(T_F) \quad F \in \mathcal{F},$$
whence the local mesh size of the face element $F$ is taken to be the diameter of the corresponding bulk element.   Note that it is possible that ${\rm diam}(F) << h_F.$  We will also denote by $h_T$ the diameter of elements $T\in\T$. We finally let
$$\T_\Gamma := \big\{T \in \T: ~~ \Gamma \cap T \neq \emptyset \hbox{ or } \gamma \cap T \neq \emptyset \big\}$$
be the set of elements of $\T$ touching either $\Gamma$ or $\gamma$.  

\medskip\noindent
{\bf Geometric Assumptions.}  
Above we described $\Gamma$ as the zero level set of an approximate distance function $d_h$.  In this section we first place abstract requirements on $\Gamma$ that are sufficient to obtain optimal-order and regularity a priori error estimates and then prove that these requirements are satisfied on sufficiently fine bulk meshes $\T$ when $\Gamma$ is built from a suitably general level set description of $\gamma$.   We now list three main geometric assumptions.

\begin{itemize} 
\item {\bf Description of $\Gamma$.}  We assume that $\Gamma$ is a polyhedral surface whose faces $F\in\mathcal{F}$ consist of the intersection of hyperplanes with simplices $T\in\T$. We further assume that $\Gamma \subset \mathcal{N}$ with $\mathcal{N}$ the tubular neighborhood defined in \eqref{N:def}.
\item {\bf Geometric resolution of $\gamma$.}  Let $d$ be the distance function to $\gamma$, $\bnu = \nabla d$ and $\bnu_\Gamma$ be the outward unit normal on $\Gamma$.  We assume that
\begin{align}
\label{trace:geo_res_assumption}
\|d\|_{L_\infty(F)} + h_F \|\bnu-\bnu_\Gamma\|_{L_\infty(F)}
\lesssim h_F^2 |d|_{W^2_\infty(\mathcal{N})}
\quad F \in \mathcal{F}.
\end{align}
This assumption is sufficient to ensure optimal decay of the geometric consistency error in a priori error estimates.  
\item {\bf Local flattening.}  We assume that for each $T \in \T$ with $T \cap \gamma \neq \emptyset$, there is a ball $B_R$ of radius $R$ with $R \simeq 1$ (independent of $h_T$) such that $T \subset B_{R/2}$ and there is a uniformly bi-$C^2$ map
\begin{align}
\label{ass:flattening}
\Phi:B_R \rightarrow \mathbb{R}^3, \quad \Phi(\gamma \cap B_R) \hbox{ lies in a hyperplane.}
\end{align}
\end{itemize} 

The flattening assumption \eqref{ass:flattening} follows from the $C^2$ nature of the surface $\gamma$ provided elements $T\in\T$ intersecting $\gamma$ are sufficiently fine with respect to the inverse of the maximum principal curvature. The flattening map $\Phi$ may be constructed by expressing $\gamma$ as a $C^2$ graph over tangent hyperplanes of $\gamma$, with the radius of the domain of these graphs bounded by the inverse of the maximum principal curvature of $\gamma$ (cf. \cite[Appendix C]{Ev98} for the construction of $\Phi$; the bound for $R$ follows from the definition of curvature).

\medskip\noindent
{\bf Level set representations.}
While we prove our results below under the abstract geometric resolution assumption \eqref{trace:geo_res_assumption} involving the distance function, in practice trace methods often build the discrete surface $\Gamma$ from a more general implicit representation of $\gamma$.   Such a representation may be obtained by assuming that $\gamma$ is the zero level set of a {\it level-set function} $\phi:\mathcal{N} \rightarrow \mathbb{R}$
$$\gamma=\{\bx \in \mathcal{N}: ~~ \phi(\bx)=0\}.$$
Broadening our assumptions concerning implicit representation of $\gamma$ is important in many practical applications.  Because the distance function $d$ has a closed form expression only if $\gamma$ is a sphere or a torus, there are many settings where $\gamma$ may easily be represented as a level set even if $d$ is not available. A simple example is the ellipsoid given by $\gamma=\left \{\bx \in \mathbb{R}^3: \frac{x^2}{a^2}+\frac{y^2}{b^2}+\frac{z^2}{c^2}-1=0 \right \}$.  {\it Level set methods} in which an evolving free boundary is computationally approximated by the level set of a discrete function are also popular in many applications.  In this case it is also natural to define $\gamma$ via a generic level set function $\phi$ rather than restrict attention to the distance function $d$.  

Our essential assumptions concerning $\phi$ are that $\phi \in C^2(\mathcal{N})$ and 
\begin{align}
\label{eq:phi_assumption}
\nabla \phi(\bx) \cdot \bnu(\bx) \ge c_\phi >0 \quad \forall \bx \in \gamma.
\end{align}
Because $\gamma$ is a level set of $\phi$, $|\nabla \phi|=|\nabla \phi \cdot \nu|$ on $\gamma$, so the assumption \eqref{eq:phi_assumption} is equivalent to assuming that $\nabla \phi$ is nondegenerate on $\gamma$ and points in the same direction as $\bnu=\nabla d$. Let $\phi_h \in \V(\mathcal T)$ be an approximation to $\phi$ satisfying 
\begin{align}
\label{eq:phih_assumption}
\|\phi-\phi_h\|_{L_\infty(T)} + h_T\|\phi-\phi_h\|_{W_\infty^1(T)}
  \lesssim h_T^2 \|\phi\|_{W^2_\infty(\mathcal{N})}
\quad T \in \T, 
\end{align}
and define the discrete surface $\Gamma$ by 
$$\Gamma : = \big\{\bx \in \mathcal{N}: ~~ \phi_h(\bx)=0 \big\}.$$

\begin{lemma}[geometric resolution]\label{L:geo-res}
Let $\gamma$ be $C^2$. Under the above assumptions, the inequality \eqref{trace:geo_res_assumption} holds for $h:=\max_{T \in \T} h_T$ sufficiently small, namely
\begin{align}
\label{eq:phi_properties}
\|d\|_{L_\infty(F)}+ h_F \|\bnu-\bnu_\Gamma\|_{L_\infty(F)}
  \lesssim h_F^2 \|\phi\|_{W^2_\infty(\mathcal{N})}
\quad F \in \mathcal{F}.
\end{align}
\end{lemma}
\begin{proof}

First let $\bx \in \mathcal{N}$, for which the projection $\bP_d(\bx)$ on $\gamma$ is uniquely defined. Let $\zeta(s) := \nabla \phi\big( s\bx + (1-s) \bP_d(\bx) \big)\cdot\bnu(\bx)$ and compute
$$
\begin{aligned}
|(\nabla \phi \cdot \bnu) (\bx) &- (\nabla \phi \cdot \bnu) (\bP_d(\bx) )| =
|\zeta(1)-\zeta(0)| = \left |\int_0^1 \zeta'(s) ds \right |
\\
&= \left | \int_0^1 \nabla \big(\nabla\phi\big( s\bx + (1-s) \bP_d(\bx) \big)\cdot\bnu(\bx) \big) \cdot \big(\bx-\bP_p(\bx) \big) \right |
\\ & \le  |\bx- \bP_d(\bx)| \, \|\nabla (\nabla \phi \cdot \nu)\|_{L_\infty([\bP_d(x), \bx])}.
\end{aligned}
$$
Since $(\nabla \phi \cdot \bnu) (\bP_d(\bx)) \ge c_\phi>0$, $\phi \in C^2(\mathcal{N})$, and $\bnu \in C^1(\mathcal{N})$, there thus exists a constant $C_\phi\le\frac{1}{2K_\infty}$ (depending on $\|\phi\|_{W_\infty^2(\mathcal{N})}$ and $|d|_{W_\infty^2(\mathcal{N})}$) such that
\[
(\nabla \phi \cdot \bnu)(\bx) \ge \frac{c_\phi}{2}
\quad \forall \,\bx \in \mathcal{N}_{\phi}:=\big\{\by\in\Omega: ~ |d(\by)|\le C_\phi  \big\} \subset \mathcal{N},
\]
according to \eqref{N:def}. Therefore, for any $\bx \in \mathcal{N}_\phi$ we have $\phi(\bP_d(\bx))=0$ and
$$
|\phi(\bx)| = \left | \int_0^1 \nabla \phi(s\bx+(1-s)\bP_d(\bx)) \cdot (\bx-\bP_d(\bx)) \right | \simeq |\bx - \bP_d(\bx)| = |d(\bx)|,
$$
because $\bx-\bP_d(\bx) = |\bx-\bP_d(\bx)| \bnu(\bx)$. Given any face $F \in \mathcal{F}$ of $\Gamma$, we realize that $\phi_h(\bx)=0$ for all $\bx\in F$ and
\[
|\phi(\bx)| = |\phi(\bx)-\phi_h(\bx)| \lesssim h_F^2 \|\phi\|_{W^2_\infty(\mathcal{N})}.
\]
If $h\ge h_F$ is sufficiently small, then $\bx\in\mathcal{N}_\phi$ and $|d(\bx)| \simeq |\phi(\bx)| \lesssim h_F^2|\phi|_{W^2_\infty(\mathcal{N})}$. This is the desired bound for the first term on the left hand side of \eqref{eq:phi_properties}.

To prove the remaining bound in \eqref{eq:phi_properties}, we note that for $\bx \in F \in \mathcal{F}$, we have $\bnu_\Gamma(\bx)= \frac{ \nabla \phi_h(\bx)}{|\nabla \phi_h(\bx)|}$ and $\bnu(\bx) = \nu(\bP_d(\bx)) = \frac{\nabla \phi(\bP_d(\bx))}{|\nabla \phi (\bP_d(\bx))|}$. Consequently, for such $\bx \in \Gamma$, we use \eqref{eq:phih_assumption}, the bound $|d(\bx)| \lesssim h_F^2$ already proved, and the $C^2$ nature of $\phi$ to obtain
\begin{align*}
\begin{aligned}
\left |(\bnu_\Gamma-\bnu)(\bx) \right| &= \left |\frac{ \nabla \phi_h(\bx)}{| \nabla \phi_h(\bx)|}-\frac{\nabla \phi(\bP_d(\bx))}{|\nabla \phi (\bP_d(\bx))|} \right| 
\\ & \le \left |\frac{ \nabla \phi_h(\bx)}{|\nabla \phi_h(\bx)|}-\frac{ \nabla \phi(\bx)}{|\nabla \phi(\bx)|} \right | + \left | \frac{ \nabla \phi(\bx)}{|\nabla \phi(\bx)|}-\frac{\nabla \phi(\bP_d(\bx))}{|\nabla \phi (\bP_d(\bx))|}\right | \\ & \lesssim \big(h_F + h_F^2\big) \|\phi\|_{W^2_\infty(\mathcal{N})} \lesssim h_F \|\phi\|_{W^2_\infty(\mathcal{N})}.
\end{aligned}
\end{align*}  
This completes the proof.
\end{proof}
 
Thus we have shown that it is possible to define the discrete surface $\Gamma$ using a generic level set representation of $\gamma$ in such a way that $\Gamma$ has the same geometric approximation properties as if it were derived more directly from the distance function $d$.  Below we assume practical access to the distance function $d$ and associated geometric properties (curvatures and normal vectors) in two further places: the first one is the definition of the right hand side $F_\Gamma$ in formulating the trace FEM and the second one is the definition of geometric a posteriori error estimators.  As outlined in \cite{DemlowDziuk:07}, it is computationally feasible to accurately approximate $d(\bx)$, $\bP_d(x)$, and $\nu(\bx)$ for $\bx \in \Gamma$ under the assumption that we have access to a level set function $\phi$ with the properties assumed above.  In outline, the foundational building block of this procedure is a numerical approximation to $\bP_d(\bx)$.  Two such algorithms are proposed in \cite{DemlowDziuk:07}, one being Newton's method and the other an ad hoc first order method; cf. \cite{Gr17} for generalizations and analysis of these methods.  Once $\bP_d(\bx)$ is computed, we then have
\[
|d(\bx)| = |\bx - \bP_d(\bx)|,
\quad
\nu(\bx)= \frac{\nabla \phi(\bP_d(\bx))}{|\nabla \phi(\bP_d(\bx))|},
\quad\bW(\bP_d(\bx))= \nabla \frac{\nabla \phi(\bP_d(\bx))}{|\nabla \phi(\bP_d(\bx))|}.
\]
These relationships allow for the computation of all geometric information required to bound geometric errors in the trace method a posteriori.  In addition, because we may reasonably assume access to $\bP_d$ it is in turn reasonable to assume a consistent definition of the right hand side $F_\Gamma$, that is, $F_\Gamma = \frac{q}{q_\Gamma} f \circ \bP_d$.  A different definition of $F_\Gamma$ would lead to an additional consistency term in the results below.

\medskip\noindent
{\bf Harmonic Extension and Traces.}
Here we collect instrumental results for our proofs of a priori and a posteriori error estimates. For the latter we use the fractional-order space $H^{3/2}(\Omega)$, so we first define the seminorm of $H^{1+s}(\Omega)$
  $$|\tv|_{H^{1+s}(\Omega)}^2 := \sum_{|\alpha|=1} \iint_{\Omega \times \Omega} \frac{|D^\alpha \tv(x)-D^\alpha \tv(y)|^2}{|x-y|^{n+2s}} \hspace{1pt} {\rm d}x \hspace{1pt}{\rm d}y$$
for a Lipschitz domain $\Omega\subset\mathbb{R}^n$ and $0<s<1$, and corresponding norm
$$\|\tv\|_{H^{1+s}(\Omega)}^2 =\|\tv\|_{H^1(\Omega)}^2 + |\tv|_{H^{1+s}(\Omega)}^2.$$

Our first lemma is a standard extension result which may for example be found in \cite[Theorem 1.4.3.1]{G85}.  

\begin{lemma}[$H^{1+s}$ extension]
\label{lem:standard_extension}
Let $D$ be a bounded Lipschitz domain in $\mathbb{R}^n$, $n\ge2$.  Then there is an extension operator $E: H^{1+s}(D) \rightarrow H^{1+s}(\mathbb{R}^n)$ such that 
\begin{align}
\label{extension}
\|E\tv\|_{H^{1+s}(\mathbb{R}^n)} \lesssim \|\tv\|_{H^{1+s}(D)} \quad \forall \,s\in [0,1), \quad \forall \tv \in H^{1+s}(D).
\end{align}
\end{lemma}

We also state a trace result relating $H^1(\mathbb{R}^2)$ and $H^{3/2}(\mathbb{R}^3)$; this is a special case of \cite[Theorem 7.58]{Ad75}.   
\begin{lemma}[trace]
\label{lem:special_trace}
If $\tv \in H^{3/2}(\mathbb{R}^n)$, $n\ge2$, and $\mathbb{P}$ is any $(n-1)$-dimensional hyperplane in $\mathbb{R}^n$, then
\begin{align}
\label{special_trace}
\|\tv\|_{H^1(\mathbb{P})}  \lesssim \|\tv\|_{H^{3/2}(\mathbb{R}^n)}.
\end{align}
\end{lemma}

The following is an important technical lemma which expresses traces relationships between norms on surface elements (flat or curved) and corresponding norms on bulk elements.  An essential component of these estimates is that they allow for surfaces to cut through bulk elements in an arbitrary fashion.   Such estimates were essential in the proof of the first a posteriori estimates for trace methods in \cite{DO12}.  In the context of a priori error estimates for trace methods, these provide a substantially simplified proof of error bounds when compared with the original proofs given in \cite{ORG09}; cf. \cite{HH02, HH04, BHL15, Re15}.

\begin{lemma}[trace estimates for cut elements]
\label{L:trace-est}
Let $D\subset \mathbb{R}^n$ ($n \ge 2$) be a (not necessarily bounded) Lipschitz domain, and let $D_{n-1}$ be the intersection of $D$ with an arbitrary hyperplane of dimension $n-1$.  Then
\begin{align}
\label{general_trace_est}
\|\tv \|_{L_2(D_{n-1})} \lesssim \|\tv\|_{H^1(D)} \quad \forall \tv \in H^1(D),
\end{align}
where the hidden constant depends on the Lipschitz nature of $D$ but not on the orientation or size of $D_{n-1}$.  In particular, let $F \in \mathcal{F}$ with $F \subset T \in \T$. Then
\begin{align}
\label{trace_est}
\|\tv\|_{L_2(F)} \lesssim h_T^{-1/2} \|\tv\|_{L_2(T)} + h_T^{1/2}\|\nabla \tv \|_{L_2(T)}
\quad \forall \tv \in H^1(T).
\end{align}
In addition, given $T \in \T$ there hold
\begin{align}
\label{curved_trace_est} 
\|\tv\|_{L_2(T \cap \gamma)} \lesssim h_T^{-1/2} \|\tv\|_{L_2(T)}+ h_T^{1/2} \|\nabla \tv\|_{L_2(T)} \quad \forall \tv \in H^1(T),
\end{align}
and
\begin{align}
\begin{aligned}
\label{H32_curved_trace}
h_T^{-1} & \|\tv\|_{L_2(T \cap \gamma)}  + \|\nabla_\gamma \tv\|_{L_2(T \cap \gamma)} 
\\ & \lesssim h_T^{-3/2} \|\tv\|_{L_2(T)} + h_T^{-1/2} \|\nabla \tv\|_{L_2(T)} + |\tv|_{H^{3/2}(T)} \quad\forall \tv \in H^{3/2}(T).
\end{aligned}
\end{align}
\end{lemma}

\begin{proof}
The estimate $\eqref{general_trace_est}$ is a special case of \cite[Lemma 5.19]{Ad75}.  The scaled result \eqref{trace_est} follows by a standard scaling argument.  

To prove \eqref{curved_trace_est} and \eqref{H32_curved_trace} we employ a flattening argument.  First let $\hat{K}$ be the unit reference simplex in $\mathbb{R}^n$ with standard affine reference mapping $\varphi: \hat{K} \rightarrow T$ satisfying $\|\nabla \varphi\|_{L_\infty(\hat K)} \lesssim h_T$ and $\|(\nabla \varphi)^{-1}\|_{L_\infty(T)} \lesssim h_T^{-1}$.  Let now $\Phi$ be the flattening map in assumption \eqref{ass:flattening}.  It is possible to extend $\Phi$ to all of $\mathbb{R}^n$ so that the resulting extension is also $C^2$, still flattens $T \cap \gamma$, and has derivative bounded above and below away from 0.  To see this, take a smoothly weighted average of $\Phi$ and the identity with weight 1 for $\Phi$ on $B_{R/2}$ and weight 0 outside of $B_R$.  Having thus extended $\Phi$, we define $\widetilde\Phi:=\varphi^{-1} \circ \Phi \circ \varphi$.  It is easy to check that $\widetilde\Phi$ and $\widetilde\Phi^{-1}$ are uniformly bounded in $C^2$ and that $\widetilde\Phi(\varphi^{-1} (T \cap \gamma))$ lies in some $(n-1)$-dimensional hyperplane $\mathbb{P}$.

For $\tv \in H^1(T)$ with $T \in \mathcal T$ satisfying $|T\cap \gamma|>0$, let now $\widehat{\tv}=\tv \circ \varphi$. We first prove \eqref{H32_curved_trace} upon transforming to the reference element back and forth. We start with a simple scaling argument
\[
\frac{|\varphi^{-1}(T\cap\gamma)|}{|T\cap\gamma|} \approx h_T^{1-n},
\]
regardless of the actual size and orientation of $T\cap\gamma$ relative to $T\in\T$. Hence, applying a standard change of variables involving $\varphi$ yields
\[
h_T^{1-n} \Big( \|\tv\|_{L_2(T \cap \gamma)}^2  + h_T^2\|\nabla_\gamma \tv \|_{L_2(T \cap \gamma)}^2 \Big) \approx \|\widehat{\tv}\|_{H^1(\varphi^{-1}(T \cap \gamma))}^2.
\]
We next resort to the extension operator $E:H^{3/2}(\hat K)\to H^{3/2}(\mathbb{R}^n)$ in Lemma \ref{lem:standard_extension} ($H^{1+s}$ extension), the smoothness of $\widetilde\Phi^{-1}$, the fact that $\widetilde\Phi( \varphi^{-1}(T \cap \gamma))\subset\mathbb{P}$, the trace inequality \eqref{special_trace}, the smoothness of $\widetilde\Phi^{-1}$ again, and the boundedness \eqref{extension} of $E$ in $H^{3/2}(\hat K)$, in this order, to arrive at
\begin{align*}
\|\widehat{\tv}\|_{H^1(\varphi^{-1}(T \cap \gamma))} &=
\|E \widehat{\tv}\|_{H^1(\varphi^{-1} (T \cap \gamma))}
\\ & \lesssim \|E \widehat{\tv} \circ \widetilde\Phi^{-1}\|_{H^1\big(\widetilde\Phi( \varphi^{-1}(T \cap \gamma))\big)}
\\ & \lesssim \|E \widehat{\tv} \circ \widetilde\Phi^{-1}\|_{H^1(\mathbb{P})}
\\ & \lesssim \|E \widehat{\tv} \circ \widetilde\Phi^{-1}\|_{H^{3/2}(\mathbb{R}^n)}
\\ & \lesssim \|E \widehat{\tv}\|_{H^{3/2}(\mathbb{R}^n)}
\\ & \lesssim \|\widehat{\tv}\|_{H^{3/2}(\hat{K})}.
\end{align*}
The desired estimate \eqref{H32_curved_trace} finally follows from a scaling argument
from $\hat K$ to $T$ employing again the map $\varphi$:
\[
\|\widehat{\tv}\|_{H^{3/2}(\hat{K})}^2
\lesssim h_T^{-n} \Big(\|\tv\|_{L_2(T)}^2 + h_T^2 \|\nabla \tv\|_{L_2(T)}^2 + h_T^3 |\tv|_{H^{3/2}(T)}^2 \Big).
\]

To prove \eqref{curved_trace_est}, we argue similarly to above except that we now employ
$E:H^1(\hat K) \to H^1(\mathbb{R}^n)$ and \eqref{general_trace_est} instead of \eqref{special_trace}.  Doing so yields
$$\begin{aligned}
 h_T^{(1-n)/2} \|\tv\|_{L_2(T \cap \gamma)} & \lesssim \|\widehat{\tv}\|_{L_2(\varphi^{-1}(T \cap \gamma))}
\\ & = \|E \widehat{\tv}\|_{L_2(\varphi^{-1} (T \cap \gamma))}
\\ & \lesssim \|E \widehat{\tv} \circ \widetilde\Phi^{-1}\|_{L_2(\widetilde\Phi( \varphi^{-1}(T \cap \gamma)))}
\\ & \lesssim \|E \widehat{\tv} \circ \widetilde\Phi^{-1}\|_{L_2(\mathbb{P})}
\\ & \lesssim \|E \widehat{\tv} \circ \widetilde\Phi^{-1}\|_{H^1 (\mathbb{R}^n)}
\\ & \lesssim \|E \widehat{\tv}\|_{H^1(\mathbb{R}^n)}
\\ & \lesssim \|\widehat{\tv}\|_{H^1(\hat{K})}
\\ & \lesssim h_T^{-n/2} \|\tv\|_{L_2(T)} + h_T^{(2-n)/2} \|\nabla \tv\|_{L_2(T)}.
\end{aligned}$$
Multiplying both sides by $h_T^{(n-1)/2}$ gives the desired bound \eqref{curved_trace_est}.
\end{proof}

\subsection{A Priori Error Estimates}\label{S:trace-apriori}

We recall that we use the notation $ h := \max_{T \in \T} h_T$ and that we omit to mention the explicit dependence on the shape regularity constant of $\mathcal T$ in most estimates.

\medskip\noindent
{\bf Geometric resolution and extensions.}  
%
Given a surface $\gamma$ of class $C^2$ and $\wu \in H^2(\gamma)$, 
Proposition \ref{P:H2-extension} ($H^2$ extension) yields the existence of an extension $u$ of $\wu$ to a tubular neighborhood $\mathcal{N}(\delta)$ with $\delta$ sufficiently small with respect to $\frac{1}{2 K_\infty}$ lying in $H^2(\mathcal{N(\delta)})$ and satisfying
\begin{align}
\label{H2extend}
\|u\|_{H^2(\mathcal{N}(\delta))} \lesssim \delta^{1/2} |d|_{W_\infty^2(\mathcal{N})}\|\wu\|_{H^2(\gamma)}.
\end{align}

\begin{itemize}
\item {\bf First assumption on geometric resolution by the bulk mesh.} We assume
\begin{equation}\label{e:TGamma}
\bigcup \big\{\omega_\T^1(T): ~ T \in \T_{\Gamma}\big\} \subset \mathcal{N}(\delta)
\end{equation}
with $\delta \simeq h$ sufficiently small so that \eqref{H2extend} holds.
\item {\bf Second assumption on geometric resolution by the bulk mesh.} We assume that the layer $D_{\Gamma, \gamma}:=\{s\bx + (1-s) \bP_d(\bx): \, \bx \in \Gamma \hbox{ and } 0 \le s \le 1\}$ satisfies 
\begin{align}
\label{ass:skinlayer}
D_{\Gamma, \gamma} \subset \bigcup \big\{T: ~T\in\T_\Gamma\big\}.
\end{align}
This clearly holds for $h$ sufficiently small because the Hausdorff distance between $\gamma$ and $\Gamma$ satisfies $\dist_H(\gamma,\Gamma)\lesssim h^2$ according to \eqref{trace:geo_res_assumption}.

\item
{\bf Uniform Poincar\'e-Friedrichs estimate on $\Gamma$.} We assume that
\begin{align} \label{trace-poin-unif}
\|\tv\|_{L_2(\Gamma)} \lesssim \|\nabla \tv\|_{L_2(\Gamma)}
\quad\forall \tv \in H^1_\#(\Gamma)
\end{align}
holds with uniform constant. According to the discussion below \eqref{poin-unif} (uniform Poincar\'e-Friedrichs constant), this only requires that $\Gamma \subset \mathcal{N}(1/2K_\infty)$ and that $\bnu\cdot \bnu_\Gamma \ge c >0$ on $\Gamma$. These conditions are easily checkable and valid asymptotically.
\end{itemize}

\medskip\noindent
{\bf Approximation properties of trace finite element space.}
%
We next state a fundamental approximation bound for the trace FEM, which we prove under the regularity assumption that $\gamma$ is of class $C^2$. We emphasize that this assumption is less restrictive than the hypotheses of previous approximation bounds for trace estimates, which assume that $\gamma$ is of class $C^3$.  

\begin{lemma}[trace approximation]
\label{lem:trace_approx}
Let $\gamma$ be of class $C^2$ and the geometric resolution assumptions \eqref{trace:geo_res_assumption}, \eqref{ass:flattening}, \eqref{e:TGamma}, and \eqref{ass:skinlayer} hold.  Then
\begin{equation}\label{e:trace_approx}
  \inf_{\tV \in \V(\mathcal F)} \|\nabla_\Gamma (\wu\circ\bP_d-V)\|_{L_2(\Gamma)}
  \lesssim h\|\wu\|_{H^2(\gamma)}.
  \end{equation}
\end{lemma}
\begin{proof}
Let $\delta \simeq h$ be sufficiently small so that \eqref{e:TGamma} is valid.
Let $\interp^{\textrm{sz}}$ be the standard Scott-Zhang interpolation operator on $\T$, and take $V=\interp^{\textrm{sz}} u$ with $u\in H^2(\Nd)$ given by Proposition \ref{P:H2-extension} ($H^2$ extension) and satisfying \eqref{H2extend}. We then denote $u_d=\wu\circ\bP_d, V_d=V\circ\bP_d$, add and subtract multiple terms, and apply the triangle inequality to find that
\[
\| \nabla_\Gamma(u_d-V)\|_{L_2(\Gamma)}  \lesssim  \sum_{i=1}^7 I_i
\]
where
\begin{align*}
I_1 &:= \|\nabla_\Gamma(u_d-V_d)\|_{L_2(\Gamma)},
\\ I_2 &:=   \|\Pi_\Gamma [\nabla V_d-(\nabla V) \circ \bP_d]\|_{L_2(\Gamma)} ,
\\ I_3 &:= \|\Pi_\Gamma [\nabla V\circ \bP_d - \nabla u \circ \bP_d]\|_{L_2(\Gamma)} ,
\\ I_4 &:= \|\Pi_\Gamma [\nabla u \circ \bP_d - (\interp^{\textrm{sz}} \nabla u) \circ \bP_d]\|_{L_2(\Gamma)},
\\ I_5 &:= \|\Pi_\Gamma [(\interp^{\textrm{sz}} \nabla u) \circ \bP_d - \interp^{\textrm{sz}} \nabla u] \|_{L_2(\Gamma)} ,
\\ I_6 &:= \|\Pi_\Gamma [\interp^{\textrm{sz}} \nabla u - \nabla u]\|_{L_2(\Gamma)},
\\ I_7 &:= \|\Pi_\Gamma [ \nabla u -\nabla V]\|_{L_2(\Gamma)}.
\end{align*}
Here we have applied the interpolation operator $\interp^{\textrm{sz}}$ componentwise to the $(n+1)$-vector $\nabla u$ and used that $\nabla_\Gamma = \Pi_\Gamma \nabla$. We next estimate each term separately.

In order to bound terms $I_1$ and $I_3$, we employ Lemma \ref{L:norm-equiv} (norm equivalence) between $\gamma$ and $\Gamma$ and recall that $|\nabla_\gamma \tv| \le |\nabla \tv|$ pointwise to find that
\begin{align*}
  I_1+I_3 \lesssim \Big (\sum_{T \in \T_\Gamma} \|\nabla (u-V)\|_{L_2(T \cap \gamma)} ^2 \Big)^{1/2}.
\end{align*}  
We next apply the trace estimate \eqref{curved_trace_est}, utilize standard approximation properties of $\interp^{\textrm{sz}}$, and finally use the bound \eqref{H2extend} to obtain
\begin{align*}
I_1+I_3 & \lesssim \Big(\sum_{T\in \T_\Gamma} h_T^{-1} \|\nabla (u-V)\|_{L_2(T)}^2 + h_T\|D^2 u\|_{L_2(T)}^2 \Big)^{1/2}
\\ & \lesssim h^{1/2} \|u\|_{H^2(\mathcal{N}(\delta))} \lesssim h \|\wu\|_{H^2(\gamma)}.
\end{align*}
Here we have used that $\nabla \nabla V=0$ elementwise since $V$ is piecewise linear. Similar arguments lead to the following estimate for $I_4$
\[
I_4 \lesssim \Big(\sum_{T\in \T_\Gamma} h_T^{-1} \|\nabla u-\interp^{\textrm{sz}} \nabla u\|_{L_2(T)}^2 + h_T\|\nabla(\nabla u-\interp^{\textrm{sz}} \nabla u)\|_{L_2(T)}^2 \Big)^{1/2},
\]
as well as $I_4 \lesssim h \|\wu\|_{H^2(\gamma)}$ provided $\|\nabla \interp^{\textrm{sz}} \nabla u\|_{L_2(T)} \lesssim \|D^2 u\|_{L_2(\omega_\T^1(T))}$. To show this estimate we let $\overline{\nabla u}_T := |\omega_\T^1(T)|^{-1} \int_{\omega_\T^1(T)} \nabla u$ be the meanvalue of $\nabla u$ in $\omega_\T^1(T)$ and exploit the stability of $\interp^{\textrm{sz}}$ in $H^1(T)$
\begin{align*}
  \|\nabla \interp^{\textrm{sz}} \nabla u\|_{L_2(T)}
  &= \|\nabla \interp^{\textrm{sz}} [\nabla u - \overline{\nabla u}_T]\|_{L_2(T)}
  \lesssim h_T^{-1} \|\interp^{\textrm{sz}}[\nabla u - \overline{\nabla u}_T]\|_{L_2(T)}
  \\ &
  \lesssim h_T^{-1} \|\nabla u - \overline{\nabla u}_T\|_{L_2(\omega_\T^1(T))}
  \lesssim \|D^2 u \|_{L_2(\omega_\T^1(T))}.
\end{align*}
Moreover, applying the trace estimate \eqref{trace_est} directly to the terms $I_6$ and
$I_7$ yields
\begin{align*}
I_6 & \lesssim \Big(\sum_{F \in \mathcal{F}} \|\interp^{\textrm{sz}} \nabla u-\nabla u\|_{L_2(F)}^2 \Big)^{1/2} 
\\ & \lesssim \Big(\sum_{T \in \T_\Gamma} h_T^{-1} \|\interp^{\textrm{sz}} \nabla u -\nabla u\|_{L_2(T)}^2 + h_T \|\nabla[\interp^{\textrm{sz}} \nabla u -\nabla u]\|_{L_2(T)}^2 \Big)^{1/2}
\lesssim h \|\wu\|_{H^2(\gamma)},
\end{align*}
and
\begin{align*}
I_7 & \lesssim \Big(\sum_{F \in \mathcal{F}} \|\nabla(u-V)\|_{L_2(F)}\Big)^{1/2}
\\ &
\lesssim \Big(\sum_{T \in \T_\Gamma} h_T^{-1} \|\nabla (u-V)\|_{L_2(T)}^2 + h_T \|D^2 u\|_{L_2(T)}^2\Big)^{1/2}
\lesssim h \|\wu\|_{H^2(\gamma)}. 
\end{align*}

In order to bound term $I_2$, we first note that
\[
\Pi_\Gamma [\nabla V_d-(\nabla V) \circ \bP_d] = \Pi_\Gamma (\Pi-d D^2d-\bI) (\nabla V) \circ \bP_d.
\]
An easy computation using the assumption \eqref{trace:geo_res_assumption} yields
$$|\Pi_\Gamma(\Pi-d D^2d - \bI)|  \lesssim |\Pi_\Gamma \Pi-\Pi_\Gamma| + |d| =|(\bnu \cdot \bnu_\Gamma) \bnu_\Gamma \otimes \bnu -\bnu \otimes \bnu|+ h^2 \lesssim h.$$
Thus employing the equivalence of norms on $\gamma$ and $\Gamma$, the trace estimate \eqref{curved_trace_est}, the $H^1$ boundedness of $\interp^{\textrm{sz}}$, and the boundedness \eqref{H2extend} of the extension yields
$$
\begin{aligned}
I_2 & \lesssim h \|\nabla \tV \circ \bP_d\|_{L_2(\Gamma)} \lesssim h \|\nabla V\|_{L_2(\gamma)} 
\\& \lesssim h^{1/2} \|\nabla V\|_{L_2(\T_\Gamma)} 
 \lesssim h^{1/2} \|u\|_{H^1(\mathcal{N}(\delta))}
 \lesssim h \|\widetilde u\|_{H^2(\gamma)}.
\end{aligned}
$$
We finally bound term $I_5$.  Given $\bx=\bP_d(\bx) + d(\bx)\nabla d(\bP_d(\bx)) \in \Gamma$, we infer that
\[
|\interp^{\textrm{sz}} \nabla u(\bx)-\interp^{\textrm{sz}} \nabla u(\bP_d(\bx))| \le \int_0^{d(\bx)} \Big| \nabla \big[\interp^{\textrm{sz}} \nabla u\big( \bP_d(\bx)+s\nabla d(\bP_d(\bx))  \big) \big] \Big| ds
\]
and $|d(\bx)|\lesssim h^2$ according to \eqref{trace:geo_res_assumption}, whence 
\[
I_5^2 \lesssim h^2 \int_\Gamma \int_0^{d(\bx)} \Big| \nabla \big[\interp^{\textrm{sz}} \nabla u\big( \bP_d(\bx)+s\nabla d(\bP_d(\bx))  \big) \big]\Big|^2 ds d\sigma(\bx) \lesssim h^2 \int_{D_{\Gamma,\gamma}} |\nabla \interp^{\textrm{sz}} \nabla u|^2.
\]
In view of assumptions \eqref{ass:skinlayer} and \eqref{e:TGamma}, and the
bound $\|\nabla \interp^{\textrm{sz}}\nabla u\|_{L_2(T)} \lesssim \|D^2 u\|_{L_2(\omega_\T^1(T))}$, we deduce
\[
I_5^2 \lesssim h^2 \|D^2 u\|_{L_2(\Nd)}^2 \lesssim h^3 \|D^2 \wu\|_{H^2(\gamma)}^2,
\]
and conclude the proof.
\end{proof}

\begin{theorem}[a-priori error estimates]
Let $\gamma$ be of class $C^2$ and let $\Gamma$ be so that the geometric assumptions \eqref{trace:geo_res_assumption}, \eqref{ass:flattening}, \eqref{e:TGamma}, \eqref{ass:skinlayer}, and \eqref{trace-poin-unif} are satisfied. Let $\wf\in L_{2,\#}(\gamma)$ and $\wu\in H^2(\gamma)$ solve \eqref{e:weak_relax}. If $U \in \V_\#(\mathcal F)$ is the finite element solution of \eqref{def:trace_fem} with $F_\Gamma =\frac{q}{q_\Gamma} \wf \circ \bP_d $, then
$$
\|\wu\circ\bP_d-U \|_{L_2(\Gamma)} + h \|\nabla_\Gamma (\wu\circ\bP_d-U)\|_{L_2(\Gamma)} \lesssim h^2 \|\wf\|_{L_2(\gamma)}.
$$
\end{theorem}

\begin{proof}
With the geometric resolution estimate~\eqref{trace:geo_res_assumption} and Lemma \ref{lem:trace_approx} (trace approximation) in hand, the proof is nearly identical to those of Theorem \ref{t:H1error} ($H^1$ a-priori error estimate) and Theorem \ref{t:L2_apriori} ($L^2$ a-priori error estimate) for parametric surface FEM. We thus sketch the proof without details.

\medskip\noindent
{\it Step 1: $H^1$ error estimate.}
Let $V \in\V(\mathcal F)$ achieve the infimum in \eqref{e:trace_approx}, $W:=V-U$, $u=\wu\circ\bP_d$, and write the error representation formula as
\[
\|\nabla_\Gamma (V-U)\|_{L_2(\Gamma)}^2 = \int_\Gamma \nabla_\Gamma u\cdot\bE_\Gamma \nabla_\Gamma W + \int_\Gamma \nabla_\Gamma (V-u)\cdot\nabla_\Gamma W,
\]
because $F_\Gamma = \frac{q}{q_\Gamma} \wf \circ \bP_d$.
In view of Lemma \ref{L:geom_consist_dist} (geometric consistency) and the geometric resolution estimate~\eqref{trace:geo_res_assumption} we deduce $|\bE_\Gamma| \lesssim h^2 |d|_{W^2_\infty(\mathcal{N})}$ and
\[
\Big| \int_\Gamma \nabla_\Gamma u\cdot\bE_\Gamma \nabla_\Gamma W  \Big|
\lesssim h^2 |d|_{W^2_\infty(\mathcal{N})} \|\wu\|_{H^1(\gamma)} \|\nabla_\Gamma W\|_{L_2(\Gamma)} \lesssim h \|\widetilde f\|_{L_2(\gamma)} \|\nabla_\Gamma W\|_{L_2(\Gamma)}.
\]
On the other hand, Lemma \ref{lem:trace_approx} (trace approximation) yields
\[
\Big|  \int_\Gamma \nabla_\Gamma (V-u)\cdot\nabla_\Gamma W \Big|
\lesssim h \|\wu\|_{H^2(\gamma)} \|\nabla_\Gamma W\|_{L_2(\Gamma)}.
\]
The desired estimate follows from Lemma \ref{L:regularity} (regularity).

\medskip\noindent
{\it Step 2: $L_2$ error estimate.} Let $\bP_d^{-1}$ denotes the inverse of $\bP_d$ restricted to $\Gamma$. 
Let $\widetilde U:=U\circ\bP_d^{-1}:\gamma\to\mathbb{R}$ and
$\widetilde U_\# := \frac{q_\Gamma}{q}\widetilde U \in H^1_\#(\gamma)$; likewise,
let $u_\# := \frac{q}{q_\Gamma} u \in H^1_\#(\Gamma)$. We now solve dual problems on $\gamma$
\[
\wz\in H^1_\#(\gamma): \quad
\int_\gamma \nabla_\gamma \wz \cdot \nabla_\gamma w = \int_\gamma (\wu - \wU_\#) w
\quad \forall \, w\in H^1_\#(\gamma)
\]
and on $\Gamma$
\[
Z\in\V_\#(\mathcal F): \quad
\int_\Gamma \nabla_\Gamma Z \cdot \nabla_\Gamma W = \int_\Gamma (u_\#-U)W
\quad\forall \, W\in\V_\#(\mathcal F).
\]
Note that the right-hand sides $u_\#-U = \frac{q}{q_\Gamma}(\wu - \wU_\#)\circ\bP_d$ are compatible and Step 1 applies.
We set $\widetilde Z = Z \circ \bP_d$ and proceed as in Theorem \ref{t:L2_apriori} ($L_2$ a-priori error estimate) to deduce the error representation
\[
\|\wu-\wU_\#\|_{L_2(\gamma)}^2 =
\int_\gamma \nabla_\gamma (\wu-\widetilde U)\cdot\nabla_\gamma(\widetilde z-\widetilde Z)
+ \int_\gamma \nabla_\gamma \widetilde U \cdot \bE \, \nabla_\gamma \widetilde Z,
\]
because $F_\Gamma =\frac{q}{q_\Gamma} \wf \circ \bP_d $.
Applying Lemma \ref{L:regularity} (regularity) to $\wz$ yields $\|\wz\|_{H^2(\gamma)} \lesssim \|\wu-\wU_\#\|_{L_2(\gamma)}$. This together with Step 1 implies
\[
\Big| \int_\gamma \nabla_\gamma (\wu-\widetilde U)\cdot\nabla_\gamma(\wz-\widetilde Z) \Big|
\lesssim h^2 \|\wf\|_{L_2(\gamma)} \|\wu-\wU_\#\|_{L_2(\gamma)}.
\]
Making use again of Lemma \ref{L:geom_consist_dist} (geometric consistency) and the geometric resolution estimate~\eqref{trace:geo_res_assumption} we deduce $|\bE| \lesssim h^2 |d|_{W^2_\infty(\mathcal{N})}$, whence
\[
\Big| \int_\gamma \nabla_\gamma \widetilde U \cdot \bE \, \nabla_\gamma \widetilde Z  \Big|
\lesssim h^2 \|\wf\|_{L_2(\gamma)} \|u_\# -U\|_{L_2(\Gamma)}.
\]
Consequently,
\[
\|\wu-\wU_\#\|_{L_2(\gamma)}^2 \lesssim h^2 \|\wf\|_{L_2(\gamma)}
\Big(\|\wu-\wU_\#\|_{L_2(\gamma)} + \|u_\#-U\|_{L_2(\Gamma)} \Big)
\]
and the asserted bound follows from Lemma \ref{L:norm-equiv} (norm equivalence)
and the auxiliary estimate
\[
\|\widetilde{U}-\widetilde{U}_\#\|_{L_2(\gamma)}
\lesssim
h^2  \|\wf\|_{L_2(\gamma)}.
\]
The latter hinges on Corollary \ref{C:lambda-2} (geometric consistency errors for $C^2$ surfaces) and the geometric resolution estimate \eqref{trace:geo_res_assumption}, as in the proof of Theorem \ref{t:L2_apriori}. This completes the proof.
\end{proof}

\subsection{A Posteriori Error Estimates}

A posteriori error estimates for the trace FEM were first proved in \cite{DO12}, while a posteriori estimates for a trace FEM based on octree meshes were proved in \cite{CO15}.    The proof of the estimates given in \cite{DO12} is significantly different than that of the a priori estimates given above.  A main reason for the difference is that, in contrast to the framework above that deals with quasi-uniform meshes, we assume that the bulk mesh $\T$ is merely shape-regular.  This is necessary to allow for meaningful mesh grading in adaptive algorithms. Moreover, the extension used in Proposition \ref{P:H2-extension} ($H^2$ extension) is not immediately useful here because the parameter $\delta$ specifying the width of the tubular neighborhood about $\gamma$ is taken to be proportional to $h$ when $\T$ is quasi-uniform; such a global mesh size parameter is no longer meaningful on graded meshes.  A local counterpart of Proposition \ref{P:H2-extension} on graded meshes, that uses the normal extension instead of the regularized normal extension, is employed in \cite{CO15} to prove a posteriori bounds, but with the drawback that the constants in the estimates depend on the difference in refinement depth between the largest and smallest elements in the bulk mesh.  We thus present here the framework of \cite{DO12}, which relies on the harmonic extension of $\tv \in H^1(\gamma)$ into $H^{3/2}(\mathbb{R}^3)$ instead of either the normal extension $\tv_d$ or the extension of Proposition \ref{P:H2-extension}.

\medskip\noindent
{\bf Notation and surface resolution assumptions.} 
%
We make the following two assumptions concerning resolution of $\gamma$ by the bulk mesh $\T$:
\begin{itemize}
\item {\bf Resolution of skin layer between $\gamma$ and $\Gamma$.}  Given a discrete surface element $F \in \mathcal{F}$, let 
$$D_F=\{{\bf y}\in\Omega: ~{\bf y}=t\bx+(1-t) \bP_d(\bx) \hbox{ for some } 0 \le t \le 1 \hbox{ and some } \bx \in F\}.$$
The set $D_F$ is the collection of all points lying on line segments connecting points in $\bx \in F$ and their images $\bP_d(\bx) \in \gamma$.  We assume that
\begin{align}
\label{trace_apost_ass1}
D_F \subset \omega_\T^1(T_F),
\end{align}
that is, $D_F$ lies in the volume element patch $\omega_\T^1(T_F)$ (first ring) corresponding to the face element $F$, which is defined in section \ref{S:trace-prelim}.
\item{\bf Normal projections of elements have finite overlap.}
We assume that
\begin{align}
\label{trace_apost_ass2}
\bP_d \big(\omega_\T^1(T_F) \big) \subset \omega_\T^2(T_F)
\quad\forall \, F \in \mathcal{F},
\end{align}
where the second ring $\omega_\T^2(T_F)$ is also defined in  section \ref{S:trace-prelim}.
\end{itemize}

The above assumptions hold if $\gamma$ is sufficiently resolved by the bulk mesh $\T$.  To see this, note first that $\|d\|_{L_\infty(D_F)} \lesssim h_F^2$ by \eqref{trace:geo_res_assumption}, so that ${\rm dist}({\bf y},F) \lesssim h_F^2$ for all ${\bf y} \in D_F$.  On the other hand, ${\rm dist}(F, \partial \omega_\T^1(T_F)) \gtrsim h_F$.  Thus there is a constant $C$ such that the assumption \eqref{trace_apost_ass1} is satisfied when $h_F \le C$;  $C$ here depends on geometric properties of $\gamma$, the shape regularity constant of $\T$ and properties of the Lagrange interpolant.  In principle an upper bound for $C$ could be computed and this condition checked, but this has not been attempted in the literature and we do not do so here.  A similar but more involved argument holds for the assumption \eqref{trace_apost_ass2}.

\medskip\noindent
{\bf Extension for a posteriori error estimates.}
The next essential result states that a given a function $\wv \in H^1(\gamma)$
can be boundedly extended to $\tv \in H^{3/2}(\mathbb{R}^{n+1})$.
\begin{lemma}[harmonic extension]
\label{lem:harmonic_extension}
Let $\gamma$ be a closed surface of class $C^2$ and dimension $n$ embedded in $\mathbb{R}^{n+1}$ for $n\ge1$. Given $\wv \in H^1(\gamma)$, there is $\tv \in H^{3/2}(\mathbb{R}^n)$ such that ${\rm trace}(\tv)=\wv$ and
\begin{align}
\label{harmonic_ext_bound}
\|\tv\|_{H^{3/2}(\mathbb{R}^{n+1})} \lesssim \|\wv \|_{H^1(\gamma)}.
\end{align}
\end{lemma}
\begin{proof}
First let $\tv\in H^1(D)$ solve $\Delta \tv=0$ on the bulk domain $D$ comprising the interior of $\gamma$, with $\tv=\wv$ on $\gamma$.  By \cite[Theorem 5.15]{JK95}, we have that $\tv \in H^{3/2}(D)$, ${\rm trace}(\tv)=\wv$, and $\|\tv\|_{H^{3/2}(D)} \lesssim \|\wv \|_{H^1(\gamma)}$.  Boundedly extending $\tv$ to $H^{3/2}(\mathbb{R}^{n+1})$ via the extension operator $E$ defined in Lemma \ref{lem:standard_extension} ($H^{1+s}$ extension) completes the proof.
\end{proof}

\medskip\noindent
{\bf Preliminary results.}
We now give a technical lemma that quantifies the evaluation mismatch between $\Gamma$ and $\gamma$ for a discrete function.
\begin{lemma}[evaluation mismatch between $\gamma$ and $\Gamma$]\label{L:eval-mismatch}
Let $\tV \in \V(\mathcal T)$, and let the conditions \eqref{trace:geo_res_assumption}, \eqref{ass:flattening}, and \eqref{trace_apost_ass1} hold. For all $F \in \mathcal{F}$, we have
\begin{align}\label{discrete_trace_dif}
\|\tV-\tV \circ \bP_d \|_{L_\infty(F)} \lesssim h_F^2 |d|_{W^2_\infty(\mathcal{N})}\|\nabla \tV \|_{L_\infty(\omega_\T^1(T_F))}.
\end{align}
\end{lemma}
\begin{proof}
Fix $\bx \in F$, and let $g(t)=\tV (t\bx + (1-t) \bP_d(\bx))$, $0 \le t \le 1$.  Then $g(0)=\tV(\bP_d(\bx))$ and $g(1)=\tV(\bx)$. Since $\bP_d(\bx) = \bx - d(\bx) \nu(\bx)$, we see that
\[  
g'(t)=\nabla \tV(t\bx + (1-t) \bP_d(\bx))\cdot (\bx -\bP_d(\bx)) = d(\bx) \nabla \tV(t\bx + (1-t) \bP_d(\bx)) \cdot \nu(\bx),
\]
whence $\tV(\bx) -\tV(\bP_d(\bx))= g(0)-g(1) = \int_0^1 g'(t) dt$ and
\begin{align*}
\big| \tV(\bx) -\tV(\bP_d(\bx)) \big|
\lesssim |d(\bx)| \|\nabla \tV \|_{L_\infty(D_F)}.
\end{align*}
The assertion follows from assumptions \eqref{trace:geo_res_assumption} and \eqref{trace_apost_ass1}.
\end{proof}

\medskip\noindent
{\bf A posteriori upper bound.} 
First we define a residual error indicator
$$
\eta_\F(U,F) := h_F\|F_\Gamma+\Delta_\Gamma U\|_{L_2(F)} + h_F^{1/2}\|\llbracket \nabla_\Gamma U \rrbracket \|_{L_2(\partial F)} \quad F \in \mathcal{F},
$$
and corresponding estimator
\[
\eta_\F (U) := \left(\sum_{F\in\F} \eta_\F(U,F)^2 \right)^{1/2}.
\]
Here $\llbracket \cdot \rrbracket$ denotes the jump in the normal component of the argument over $\partial F$.  Because we have assumed access to the closest point projection $\bP_d$, we also employ a geometric indicator that directly accesses information from $\bP_d$
\begin{align*}
  \xi_F := \|d\|_{L_\infty(F)} \|K\|_{L_\infty(\bP_d(F))} + \|\bnu-\bnu_\Gamma\|_{L_\infty(F)}^2
  \quad F\in\F,
\end{align*}
and corresponding geometric estimator
\[
\xi_\F(\Gamma) :=\max_{F \in \mathcal{F}} \xi_F.
\]
\begin{theorem}[a-posteriori upper estimate]
Let $\gamma$ be of class $C^2$ and let $\Gamma$ be defined so that the geometric assumptions \eqref{trace:geo_res_assumption}, \eqref{ass:flattening}, \eqref{trace_apost_ass1}, and \eqref{trace_apost_ass2} hold. Let $\wf\in L_{2,\#}(\gamma)$ and $\wu\in H^1_\#(\gamma)$ solve \eqref{e:weak_relax}. If $U \in \V(\mathcal F)$ is the finite element solution of \eqref{def:trace_fem} with $F_\Gamma =\frac{q}{q_\Gamma} \wf \circ \bP_d $, and $U_d=U\circ\bP_d^{-1}$ where $\bP_d^{-1}$ is the inverse of $\bP_d$ restricted to $\Gamma$, then
\begin{align}
\label{trace_apost}
\|\nabla_\gamma(\wu-U_d)\|_{L_2(\gamma)} \lesssim \eta_\F(U) + \xi_\F(\Gamma) \|\nabla_\Gamma U\|_{L_2(\Gamma)}.
\end{align}
\end{theorem}

\begin{proof}
We proceed in several steps.

\medskip\noindent
{\it Step 1:  Error representation via the residual equation.}
First we note that
$$
\|\nabla_\gamma(\wu-U_d)\|_{L_2(\gamma)} = \sup_{\wv \in H^1(\gamma), \|\nabla_\gamma \wv\|_{L_2(\gamma)}=1} \int_\gamma \nabla_\gamma(\wu-U_d)\cdot \nabla_\gamma \wv
$$
and then write as in \eqref{e:post_split} that
$$
\int_\gamma \nabla_\gamma(\wu-U_d)\cdot \nabla_\gamma \wv = I_1+I_2+I_3
$$
with
\begin{align*}
I_1& :=-\int_\Gamma \nabla_\Gamma U \cdot \nabla_\Gamma(\tv_d-\tV)+ \int_\Gamma F_\Gamma (\tv_d-\tV),
\\ I_2 & := \int_\gamma \nabla_\gamma U_d \cdot \bE \, \nabla_\gamma \wv,
\\ I_3 & := \int_\gamma \wf \, \wv - \int_\Gamma F_\Gamma \tv_d.
\end{align*}
Here $\bE$ is as in \eqref{error-matrix-gamma}, $\tv_d=\wv\circ\bP_d$, and $V\in\V(\mathcal T)$ is a suitable approximation of the $H^{3/2}$ extension $\tv$ of $\wv$ given by Lemma \ref{lem:harmonic_extension} (harmonic extension). Note that $I_3=0$ because of the definition $F_\Gamma=\frac{q}{q_\Gamma} \wf \circ \bP_d$.

\medskip\noindent
{\it Step 2:  Bounding the geometric error terms.}
Using \eqref{est-E} (or more accurately the corresponding pointwise bound from which it is derived) directly yields
$$\|{\bf E}\|_{L_\infty(F)} \lesssim \xi_F
\quad E\in\mathcal{F}.$$
Thus making use of Lemma \ref{L:norm-equiv} (norm equivalence) implies
$$
|I_2| \lesssim \xi_\F(\Gamma) \|\nabla_\gamma U_d\|_{L_2(\gamma)} \|\nabla_\gamma \wv\|_{L_2(\gamma)} \le \xi_\F(\Gamma) \|\wf \|_{L_2(\gamma)} \|\nabla_\gamma \wv\|_{L_2(\gamma)}.
$$

\noindent
{\it Step 3:  Bounding the residual term.}
In order to bound $I_1$, we first decompose the integrals over faces $F\in\F$
and then integrate by parts to arrive at
$$
|I_1| \lesssim \sum_{F \in \mathcal{F}} \eta_\F(U,F)\Big(h_F^{-1} \|\tv_d-\tV\|_{L_2(F)} + h_F^{-1/2} \|\tv_d-\tV\|_{L_2(\partial F)}\Big).
$$
We may thus complete the proof upon showing that
$$
\left ( \sum_{F \in \mathcal{F}} h_F^{-2} \|\tv_d-\tV\|_{L_2(F)}^2+ h_F^{-1} \|\tv_d-\tV\|_{L_2(\partial F)}^2 \right ) ^{1/2} \lesssim \|\wv \|_{H^1(\gamma)}.
$$

Given $F \in \mathcal{F}$, we begin by considering the quantity $h_F^{-1} \|\tv_d -\tV\|_{L_2(e)}$ for any edge $e \subset \partial F$.  We first use the triangle inequality to obtain
\[
h_F^{-1/2} \|\tv_d- \tV\|_{L_2(e)} \lesssim h_F^{-1/2} \|\tv_d-\tV_d\|_{L_2(e)}
+ h_F^{-1/2}\|\tV_d-\tV\|_{L_2(e)} 
\]
with $V_d = V\circ\bP_d$, 
and examine the last term first. Since $e$ is an $(n-1)$-dimensional edge with diam$(e) \le h_F$, combining H\"older inequality and Lemma \ref{L:eval-mismatch} (evaluation mismatch between $\gamma$ and $\Gamma$) with an inverse estimate over the $(n+1)$-dimensional patch $\omega_\T^1(T_F)$ yields
\begin{align*}
h_F^{-1/2}\|\tV_d-\tV\|_{L_2(e)} &\lesssim h_F^{(n-2)/2}\|\tV_d-\tV\|_{L_\infty(F)}
\\ &
\lesssim h_F^{(n+2)/2} |d|_{W^2_\infty(\mathcal{N})} \|\nabla V\|_{L_\infty(\omega_\T^1(T_F))}
\\&
\lesssim h_F^{1/2} |d|_{W^2_\infty(\mathcal{N})} \|\nabla V\|_{L_2(\omega_\T^1(T_F))}.
\end{align*}

For the first term $h_F^{-1/2} \|\tv_d-\tV_d\|_{L_2(e)}$ we argue as follows. Let $\mathbb{P}$ be the $n-$dimensional hyperplane containing $F$.  While it may be that ${\rm diam}(e) <\!<h_F$, the shape regularity of $\T$ implies that 
$${\rm dist}(e, \partial \omega_{\T}^1(T_F)) \ge {\rm dist}(T_F, \partial \omega_\T^1(T_F) \simeq h_F.$$  
Thus there exists an $n-$dimensional ball $B \subset \mathbb{P} \cap \omega_{\T}^1(T_F) \subset \mathbb{P}$ so that $e \subset B$ and ${\rm diam}(B) \simeq h_F$.  This ball $B$ is the candidate for applying the $h_F$-scaled version of \eqref{general_trace_est} of Lemma \ref{L:trace-est} (trace estimates for cut elements), namely
\[
h_F^{-1/2} \|\tv_d-\tV_d\|_{L_2(e)}\lesssim h_F^{-1} \|\tv_d-\tV_d\|_{L_2(B)} + \|\nabla_\mathbb{P} (\tv_d-\tV_d)\|_{L_2(B)}.
\]
Since $\tv_d-\tV_d = (\wv-V)\circ\bP_d$, we change variables from $B$ to $\gamma$ while employing Lemma \ref{L:norm-equiv} (norm equivalence) to get
\[
h_F^{-1/2} \|\tv_d-\tV_d\|_{L_2(e)}\lesssim
h_F^{-1} \|\wv-\tV\|_{L_2(\bP_d(B))} + \|\nabla_\gamma(\wv-\tV)\|_{L_2(\bP_d(B))}.
\]
We observe that $\bP_d(B)\subset \bP_d(\omega_\T^1(T_F)) \subset \omega_\T^2(\T_F)$ in light of \eqref{trace_apost_ass2}, whence
\[
h_F^{-1/2} \|\tv_d-\tV_d\|_{L_2(e)}\lesssim
h_F^{-1} \|\tv-\tV\|_{L_2(\omega_\T^2(T_F)\cap \gamma)} + \|\nabla_\gamma(\tv-\tV)\|_{L_2(\omega_\T^2(T_F) \cap \gamma)}.
\]

We now carry out a similar but more direct computation for the term $h_F^{-1} \|\tv_d-\tV\|_{L_2(F)}$ appearing in $I_1$.   Again using Lemma \ref{L:eval-mismatch} we obtain
\begin{align*}
  h_F^{-1} \|\tv_d-\tV_d\|_{L_2(F)} &\lesssim h_F^{-1} \|\wv-V\|_{L_2(T\cap\gamma)}
  \\
  h_F^{-1} \|\tV_d-\tV\|_{L_2(F)} &
  \lesssim h_F^{(n+2)/2} \|\nabla V\|_{L_\infty(\omega_\T^1(T_F))}
  \lesssim \|\nabla V\|_{L_2(\omega_\T^1(T_F))}.
\end{align*}
Combining the previous estimates we end up with
\begin{align*}
 h_F^{-1}&  \|\tv_d-\tV\|_{L_2(F)}   + h_F^{-1/2} \|\tv_d-\tV\|_{L_2(\partial F)} 
 \\ &  \lesssim h_F^{-1} \|\wv-\tV\|_{L_2(\omega_\T^2(T_F)) \cap \gamma)} + \|\nabla_\gamma(\wv-\tV)\|_{L_2(\omega_\T^2(T_F)) \cap \gamma)} + \|\nabla \tV\|_{L_2(\omega_\T^1(T_F))}.
\end{align*}
Summing over $F \in \mathcal{F}$ while using finite overlap of the patches $\omega_\T^2(T_F)$ yields
\begin{align*} 
&  \left ( \sum_{F \in \mathcal{F}} h_F^{-2} \|\tv_d-\tV\|_{L_2(F)}^2  + h_F^{-1} \|\tv_d-\tV\|_{L_2(\partial F)}^2 \right )^{1/2} 
 \\ &  \qquad \lesssim \left ( \sum_{T \in \T_\Gamma} h_T^{-2} \|\wv -\tV\|_{L_2(T \cap \gamma)}^2 + \|\nabla_\gamma (\wv -\tV)\|_{L_2(T \cap \gamma)}^2 \right ) ^{1/2} + \|\nabla \tV\|_{L_2(\Omega)}.
 \end{align*} 

{\it Step 4: Interpolation.} 
We next apply \eqref{H32_curved_trace} to the function $\tv-\tV$ while realizing that $|\tV|_{H^{3/2}(T)}=0$.  Doing so yields $|\tv-\tV|_{H^{3/2}(T)}=|\tv|_{H^{3/2}(T)}$ and
$$ \begin{aligned}
h_T^{-1} \|\wv -\tV\|_{L_2(T \cap \gamma)} & + \|\nabla_\gamma (\wv -\tV)\|_{L_2(T \cap \gamma)}
 \\ &  \lesssim h_T^{-3/2} \|\tv-\tV\|_{L_2(T)} + h_T^{-1/2} \|\nabla (\tv-\tV)\|_{L_2(T)} + |\tv|_{H^{3/2}(T)}.
\end{aligned}
$$ 
Next let $\tV=I^{\textrm{sz}}_\T \tv$, where $I^{\textrm{sz}}_\T$ is the Scott-Zhang interpolation operator on the bulk space $\V(\mathcal T)$.  Standard approximation theory in  $\V(\mathcal T)$ then yields
$$ h_T^{-3/2} \|\tv-\tV\|_{L_2(T)} + h_T^{-1/2} \|\nabla (\tv-\tV)\|_{L_2(T)} + |\tv|_{H^{3/2}(T)} \lesssim \|\tv\|_{H^{3/2}(\omega_\T^1(T))}$$
and
$$\|\nabla \tV\|_{L_2(T)} \lesssim \|\nabla \tv\|_{L_2(\omega_\T^1(T))}$$
for every $T\in\T_\Gamma$.
Using the finite overlap of the patches $\omega_\T^1(T)$ and the bound
$\|\tv\|_{H^{3/2}(\mathbb{R}^{n+1})} \lesssim \|\wv\|_{H^1(\gamma)}$ of Lemma \ref{lem:harmonic_extension} (harmonic extension), we finally obtain
\begin{align*}
\left ( \sum_{F \in \mathcal{F}} h_F^{-2} \|\tv_d-\tV\|_{L_2(F)}^2  + h_F^{-1} \|\tv_d-\tV\|_{L_2(\partial F)}^2 \right )^{1/2} & \lesssim \|\tv\|_{H^{3/2}(\mathbb{R}^3)} \lesssim \|\wv\|_{H^1(\gamma)}.
\end{align*}
This completes the proof.
\end{proof}

\begin{remark}[efficiency]
 In a posteriori error analysis it is standard to prove lower (efficiency) bounds such as those in Theorems \ref{T:apost-lower} and \ref{T:apost-lower-dist}.  For trace methods such estimates would ideally take the form
\[ 
\begin{aligned}
   \eta_{\mathcal{F}}(U,F) &  \lesssim \|u_d-U\|_{H^1(\omega_\F^1(F))} +
   \osc_{\mathcal{F}} (F_\Gamma, \omega_{\mathcal{F}}^1(F))
   \\ & ~~~~~+ \xi_{\mathcal{F}} (\omega_{\mathcal{F}}^1(F)) \|\nabla_\Gamma U\|_{L_2(\omega_\F^1(F))}.
   \end{aligned}
\]   
where $\omega_{\mathcal{F}}^1(F)$ is the patch of elements about $F \in \mathcal{F}$ and  $\osc_{\mathcal{F}} (F_\Gamma, \omega_{\mathcal{F}}^1(F))$ is a heuristically higher-order term measuring the deviation of $F_\Gamma$ from the piecewise constants.  However, the standard proof of this result does not work for trace methods due to the irregular structure of the surface mesh $\mathcal{F}$.  The paper \cite{DO12} contains partial efficiency results for the volume residual but none for the jump residual term.  Numerical experiments suggest that a local efficiency result may hold, but also show a slight degeneration of the constant as the mesh is refined.  Thus it is not clear whether the estimators we have studied for the trace method are efficient, and if so what form an efficiency estimate would take.  
\end{remark}

\section{Narrow Band Method}\label{sec:narrow}

In the narrow band approach, the partial differential equation \eqref{e:weak_relax}
on $\gamma$
$$
-\Delta_\gamma \wu = \wf
$$
is extended to the tubular neighborhood $\Nd$ of $\gamma$ defined in \eqref{e:delta-tube}
$$
\mathcal N(\delta):= \left \lbrace  \bx  \in \mathbb R^{n+1}  : \   |d(x)| < \delta \right\rbrace \subset \mathbb R^{n+1};
$$
we refer to the original papers \cite{MR1868103,MR2485787}. The finite element method is then posed over a discrete approximation to $\mathcal{N}(\delta)$.
We assume that $\gamma$ is of class $C^2$ and $0<\delta < \frac 1{2 K_\infty}$ so that \eqref{e:Ntilde} holds, namely $\Nd\subset{\mathcal{N}}_\eps(\delta_\eps)$, and all the properties of the distance function detailed in Section~\ref{sec:preliminaries} are valid in $\mathcal N(\delta)$.

A natural / standard way to extend $\wu$ and $\wf$ to $\mathcal N(\delta)$ is to use the constant extensions along the normal direction
\[
u = \wu \circ \bP_d,
\quad
f = \wf \circ \bP_d.
\]
We use the latter to design the FEM. However, we need $u \in H^2(\Nd)$ to derive optimal a-priori $H^1$ error estimates for the FEM, which entails $\gamma \in  C^3$ when using the closest point projection $\bP_d$. We circumvent this extra regularity on $\gamma$ via Proposition~\ref{P:H2-extension} ($H^2$ extension), which defines $u$ as a normal extension relative to a perturbation $\gae$ of $\gamma$ constructed as a zero level set of a regularized distance function $\de$. We will show below in Lemma \ref{t:narrow:geom_consistency} (narrow band PDE consistency) that such a function $u$ satisfies
\begin{equation}\label{narrow-band-eq}
\left| \int_{\mathcal N(\delta)} \nabla u \cdot \nabla v - \int_{\mathcal N(\delta)} f v \right| \lesssim  \delta^{3/2} \| \wf \|_{L_2(\gamma)} \|  v \|_{H^1(\mathcal N(\delta))}. 
\end{equation}
The specific choice of $u$ adds several technicalities to the proof of \eqref{narrow-band-eq} but reduces the regularity of $\gamma$ to $C^2$. This seems to be a new result in the literature consistent with the underlying regularity $\wu\in H^2(\gamma)$. This also motivates the narrow band FEM as a straightforward (bulk) finite element approximation of \eqref{narrow-band-eq} upon replacing $\Nd$ by a polygonal approximation $\mathcal N_h(\delta)$ dictated by $d_h$, the Lagrange interpolant of $d$ in the bulk. We discuss this next.
We refer to \cite{MR3471100} for higher order FEMs and \cite{MR2608464,MR3249369} for an algorithm based on a level-set function, rather that the less practical distance function. The essential ideas, however, are similar to those below but are more technical.

\subsection{The Narrow Band FEM}\label{S:FE-narrow-band}

We assume that $\mathcal{N}$ is enclosed in a $n+1$ dimensional polyhedral domain $D$ and denote by $\mathcal T$ a partition of $D$ made of simplices. We omit to mention the explicit dependence on the shape regularity constant of $\T$
$$
\sigma := \max_{T\in \mathcal T} \frac{\textrm{diam}(T)}{h_T}
$$
in most estimates below;
we use the notation $h_T=| T |^{\frac 1 {n+1}}$ and $h=\max_{T\in \mathcal T} h_T$.
Let $d_h$ stand for the Lagrange interpolant of the
distance function $d$ by continuous piecewise linear functions over $\T$. The discrete distance function $d_h$ induces the discrete narrow band
$$
\mathcal N_h(\delta) := \left\lbrace \bx \in D \ : \  |d_h(\bx)| < \delta \right\rbrace.
$$ 
Notice that standard interpolation estimates imply
\begin{equation}\label{e:narrow:interp}
  \| d - d_h \|_{L_\infty(\mathcal N)} + h \| \nabla (d-d_h)\|_{L_\infty(\mathcal N)} \leq c_I h^2
  | d |_{W^2_\infty(\mathcal N)},
\end{equation}
where $c_I$ is a constant only depending on $\sigma$.
This implies the {\it non-degeneracy} property
\begin{equation}\label{e:narrow:nondegen}
| \nabla d_h| \geq \big| |\nabla d| - |\nabla(d-d_h)| \big| \geq  \big| 1 - |\nabla(d-d_h)| \big|  \geq  \frac 1 2,
\end{equation}
provided $h$ is sufficiently small so that
$
c_I h | d |_{W^2_\infty(\mathcal N)} \leq \frac 1 2. 
$
Combining estimates \eqref{e:narrow:interp} and \eqref{e:narrow:nondegen}
we deduce that the Hausdorff distance between
$\Nd$ and $\mathcal{N}_h(\delta)$ satisfies 
\begin{equation}\label{hausdorff}
\dist_H(\Nd,\mathcal{N}_h(\delta)) \le 2 c_I h^2 | d |_{W^2_\infty(\mathcal N)}.
\end{equation}
Moreover, to guarantee that $\mathcal N_h(\delta) \subset \mathcal N$, we observe
$$
|d(\bx)| \leq |d_h(\bx)| + |(d-d_h)(\bx)| \leq \delta + c_I | d |_{W^2_\infty(\mathcal N)} h^2
\quad\forall \, \bx \in \mathcal N_h(\delta).
$$
In view of \eqref{N:def}, it thus suffices to restrict $\delta$ and $h$ so that
\begin{equation}\label{e:narrow:delta_and_h}
\delta + c_I | d |_{W^2_\infty(\mathcal N)} h^2 \leq \frac{1}{2K_\infty}.
\end{equation}
Hereafter we make the structural assumption
\begin{equation}\label{delta-h}
C_1 h \le \delta \le C_2 h
\end{equation}
with $c_I \le C_1 \le C_2$ so that \eqref{e:narrow:delta_and_h} holds for $h$ sufficiently small.

We denote by  $\mathcal T_\delta$ the restriction of $\mathcal T$ to $\mathcal N_h(\delta)$  in the sense that
$$
\mathcal T_\delta := \big\lbrace T\in\T  \ : \ T\cap \mathcal N_h(\delta) \not = \emptyset \big\rbrace.
$$
The finite element space associated with $\mathcal T_\delta$ is then constructed in the usual way
$$
\mathbb V(\mathcal T_\delta) := \left\lbrace V \in C^0(\overline{\mathcal N_h(\delta)}) \ : \ V|_T \in \mathcal P, \ T \in \mathcal T_\delta \right\rbrace, \quad
$$
where we recall that $\mathcal P$ stands for the space of polynomials of degree 1.
The subspace of functions with vanishing mean value is denoted $\mathbb V_{\#}(\mathcal T_\delta)$.

With this notation at hand and inspired by \eqref{narrow-band-eq},  we define the narrow band finite element solution $U \in \mathbb V_{\#}(\mathcal T_\delta)$ to satisfy
\begin{equation}\label{e:discrete_nb}
\int_{\mathcal N_h(\delta)} \nabla U \cdot \nabla V = \int_{\mathcal N_h(\delta)}F  V, \qquad \forall V \in \mathbb V_\#(\mathcal T_\delta),
\end{equation}
where $F$ is an approximation to $f=\widetilde f \circ \bP_d$ satisfying $\int_{\mathcal N_h(\delta)} F = 0$. In order to make a convenient choice of $F$, we first define ${\bf M}_h:\mathcal{N}_h(\delta) \rightarrow \mathcal{N}(\delta)$ by
\[
  {\bf M}_h(\bx)=\bP_d(\bx) + d_h(\bx) \nabla d(\bx);
\]
the properties of ${\bf M}_h$ are explored thoroughly later in this section.  With this definition in hand, we let
\begin{equation}\label{e:narrow:F}
F = f \circ {\bf M}_h-  \frac{1}{  |\mathcal N_h(\delta)|}  \int_{\mathcal N_h(\delta)} f \circ {\bf M}_h.
\end{equation}
%
%
This requires having access to $d, d_h$ and $\bP_d$, which we assume hereafter.
Since $F$ has vanishing meanvalue, \eqref{e:discrete_nb} is also valid for all
$V\in \mathbb V(\mathcal T_\delta)$.
The existence and uniqueness of $U \in \mathbb V_\#(\mathcal T_\delta)$ follows directly from the Lax-Milgram lemma. 

\subsection{PDE Geometric Consistency}\label{s:narrow:consistency}

We intend to prove \eqref{narrow-band-eq} for the extension $u \in H^2(\Nd)$ in Proposition \ref{P:H2-extension} ($H^2$ extension) of $\wu\in H^2(\gamma)$. We recall Proposition \ref{P:BVP} (PDE satisfied by $u$)
\[
-\div { \mue \bBe \nabla u } = \fe \mue,
\]
multiply by a test function $\tv \in H^1(\mathcal N(\delta))$ and integrate by parts in $\Nd$ to obtain
\begin{equation}\label{e:narrow_band_exact}
\int_{\mathcal N(\delta)} \bB_\varepsilon \nabla u \cdot \nabla \tv ~\mu_\varepsilon = \int_{\mathcal N(\delta)} f_\varepsilon~ \tv ~\mu_\varepsilon
+ \int_{\partial \mathcal N(\delta)} \bB_\varepsilon \nabla u \cdot \nabla d ~ \tv ~\mu_\varepsilon.
\end{equation}
Notice that we have used that $\bnu=\nabla d$ is the outward pointing normal to $\partial \Nd$.
We start by estimating geometric quantities appearing in \eqref{e:narrow_band_exact}.

%
\begin{lemma}[properties of $\mue$ and $\bBe$]\label{e:narrow:consistency_alg}
Let $\gamma$ be of class $C^2$ and $C\delta \le \eps\le\frac{\delta}{2}$ be sufficiently small. Then for all $\bx\in\Nd$ we have
\begin{equation}\label{nb:mu_estim}
\|1-\mu_\varepsilon\|_{L_\infty(\mathcal N(\delta))} \lesssim   \delta |d|_{W^2_\infty(\mathcal{N})}
\end{equation}
and
\begin{equation}\label{nb:matrix_estim}
\|  \Pi_\eps - \bB_\varepsilon \mu_\varepsilon \|_{L_\infty(\mathcal N(\delta))} \lesssim    \delta |d|_{W^2_\infty(\mathcal{N})}.
\end{equation}
\end{lemma}
\begin{proof}
We recall the definitions of $\mue$ from Proposition \ref{P:BVP} (PDE satisfied by $u$)
and $\wmue$ from Lemma \ref{L:PDE-ue} (PDE satisfies by $\ue$)
\[
\mue = \frac{1}{\wmue\circ\bPe} \det\Big( \bI - \wde \, D^2\wde  \Big), \quad 
\wmue = \det\Big(\bI - \wde \, D^2\wde  \Big)
\big( \nabla d \cdot \nabla\wde  \big) \circ\bQe,
\]
where $\bP_\eps$ is the projection from $\mathcal N(\delta)$ onto $\gamma_\eps =\{ \wde (\bx) = 0\}$ and $\bQ_\eps$ is its inverse when restricted to $\gamma$. 
Note that in $\Nd$
\[
1-\mue = \Big(1 - \frac{1}{\wmue\circ\bPe}  \Big)
+ \frac{1}{\wmue\circ\bPe} \Big(1 - \det \big(\bI - \wde D^2 \wde \big) \Big).
\]
We thus need to examine the eigenvalues $(\zeta_i(\bx))_{i=0}^n$ of 
$$
\bI - \wde(\bx)D^2 \wde(\bx) \quad\forall \, \bx\in\Nd,
$$
with $\zeta_0(\bx)=1$ corresponding to the eigenvector $\nabla \wde$.
We infer that
$$
\zeta_i(\bx) = 1 - \eta_i(\bx)
$$
where 
\begin{equation*}\label{e:narrow:zeta}
  | \eta_i(\bx)| \lesssim  | \wde(\bx)| ~ | \wde |_{W^2_\infty(\Nd)} \lesssim
  \delta | d |_{W^2_\infty({\mathcal N})}
\end{equation*}
according to Lemma \ref{L:properties-Pe} (properties of $\wde$) and \eqref{e:Ntilde}
with $\delta_\varepsilon \le \frac32 \delta$. Hence
\begin{equation*}\label{e:narrow:estimr}
\Big| 1 - \det \big( \bI - \wde(\bx)D^2 \wde(\bx) \big) \Big|
= \left| 1- \prod_{i=1}^{n} \zeta_i(\bx)\right|
\lesssim \delta |d|_{W^2_\infty({\mathcal N})}
\end{equation*}
for all $\bx\in\Nd$. This takes care of the second term in the equation for $1-\mue$.
It remains to estimate $1-\wmue\circ\bPe$. Since $1-\wmue\circ\bPe$ reads as follows
on $\gamma$
\[
1-\wmue = \Big(1 - \det \big( \bI - \wde D^2 \wde \big) \Big)
+ \det \big( \bI - \wde D^2 \wde \big) \big(1 - \nabla d\cdot\nabla\wde \big),
\]
combining the previous estimate with Lemma \ref{L:properties-Pe} (properties
of $\wde$) yields
\[
\big| 1-\wmue(\bPe(\bx)) \big| \lesssim \delta  |d|_{W^2_\infty({\mathcal N})}
\quad\forall \, \bx\in\Nd.
\]
This implies $|\wmue(\bPe(\bx))|\ge\frac12$ for $\delta$ suficiently small and
thus leads to \eqref{nb:mu_estim}.

We now prove \eqref{nb:matrix_estim} which, in light of \eqref{nb:mu_estim}, reduces to
the estimate $\| \Pi_\eps - \bB_\varepsilon\|_{L_\infty(\mathcal N(\delta))} \lesssim \delta |d|_{W^2_\infty({\mathcal N})}$.
We recall from Proposition \ref{P:BVP} that in $\Nd$
\[
\bBe = \big(\bI - \wde \, D^2\wde \big)^{-1} \Pi_\eps \wbAe\circ\bPe \Pi_\eps
\big( \bI - \wde \, D^2\wde \big)^{-1}.
\]
Since $\wde(\bx) \le \widetilde \delta \leq \frac{3}{2}\delta$ for $\bx \in \Nd$ and $\| D^2 \wde \|_{L_\infty(\mathcal N(\delta))} \lesssim |d|_{W^2_\infty(\mathcal N)}$ thanks to Lemma~\ref{L:properties-Pe} (properties of $\wde$),  
the Taylor expansion of  $(\bI - t \wde D^2\wde)^{-1}$ centered at $t=0$ and computed
at $t=1$ converges for $\delta$ sufficiently small. It reads
$$
(\bI - \wde D^2 \wde)^{-1} = \bI +  \wde  (\bI- \xi \wde D^2 \wde)^{-2} D^2 \wde
$$  
for some $0< \xi <1$. The definition of $\wbAe$ given in Lemma~\ref{L:PDE-ue} yields
$$
\bBe =\Pi_\eps (\Pi\circ \bQ_\eps \circ \bP_\eps) \Pi_\eps +  \wde \bG,
$$
where $\bG:\mathcal N(\delta) \rightarrow \mathbb R^{(n+1)\times (n+1)}$ satisfies $\| \bG \|_{L_\infty(\mathcal N(\delta))} \lesssim 1$. 
Moreover,
\[
\Pi_\eps - \Pi_\eps (\Pi \circ \bQ_\eps \circ \bP_\eps) \Pi_\eps = \Pi_\eps \nabla d \circ (\bQ_\eps \circ \bP_\eps)  \otimes \Pi_\eps \nabla  d \circ (\bQ_\eps \circ \bP_\eps)
\]
whence for all $\bx\in\Nd$ we see that
\[
\Pi_\eps(\bx) \nabla d(\bQ_\eps(\bP_\eps(\bx))) =
\nabla d(\bQ_\eps(\bP_\eps(\bx))) -
\nabla\wde(\bx) \big(\nabla d(\bQ_\eps(\bP_\eps(\bx))) \cdot\nabla\wde(\bx)\big).
\]
Since
\begin{align*}
  \big| \nabla \wde(\bx) - \nabla d(\bQ_\eps(\bP_\eps(\bx))) \big| &\leq
  \big | \nabla (\wde(\bx) -d(\bx)) \big| + \big| \nabla d(\bx) - \nabla d(\bQ_\eps(\bP_\eps(\bx))) \big| \\
  &\lesssim \big(\delta+|\bx -\bQ_\eps(\bP_\eps(\bx))| \big) |d|_{W^2_\infty(\mathcal{N})}
  \lesssim \delta |d|_{W^2_\infty(\mathcal{N})}
\end{align*}
thanks to Lemma~\ref{L:properties-Pe} (properties of $\wde$), we get
\begin{equation*}\label{e:narrow_matrix_A}
|| \Pi_\eps - \bBe||_{L_\infty(\Nd)} \lesssim \delta |d|_{W^2_\infty(\mathcal{N})}
\end{equation*}
as asserted. This concludes the proof.
\end{proof}
\begin{remark}[estimate of $\mu$]\label{r:narrow:mu}
Lemma~\ref{L:area_ratio_distance} (relation between $q$ and $q_\Gamma$) gives
the expression $\mu(\bx):= \det(\bI-d(\bx) D^2 d(\bx))$ for the change of infinitesimal
area between $\gamma_s:=\{ d^{-1}(s)\}$ and $\gamma := \{ d^{-1}(0)\}$.
Proceeding as in the proof of the above lemma, we get
\begin{equation}\label{nb:mu_estim_d}
\|1-\mu\|_{L_\infty(\mathcal N(\delta))} \lesssim  \delta \, |d|_{W^2_\infty(\mathcal{N})}
\end{equation}
provided $\delta$ is sufficiently small so that
$\mathcal N(\delta) \subset \mathcal N$.
\end{remark}
We are now in position to prove a consistency estimate measuring the discrepancy between $f$ and $\Delta u$ in $\mathcal{N}(\delta)$.
\begin{lemma}[narrow band PDE consistency]\label{t:narrow:geom_consistency}
Let $\gamma$ be of class $C^2$ and $u$ be the extension of Proposition \ref{P:H2-extension} ($H^2$ extension) with $C\delta \le \eps \le \frac{\delta}{2}$
sufficiently small. If $\wf \in L_2(\gamma)$,
then for all $\tv \in H^1(\mathcal N(\delta))$, we have
\begin{equation}\label{e:narrow:geom_consistency}
\left| \int_{\mathcal N(\delta)} \nabla u \cdot \nabla \tv - \int_{\mathcal N(\delta)} f \tv \, \right| \lesssim  \delta^{3/2} |d|_{W^2_\infty(\mathcal{N})}^2 \| \wf \|_{L_2(\gamma)} \|  \tv \|_{H^1(\mathcal N(\delta))}. 	
\end{equation}
\end{lemma}
\begin{proof}
In view of \eqref{e:narrow_band_exact}, we deduce
$$
I(\tv) := \int_{\mathcal N(\delta)} \nabla u \cdot \nabla \tv - \int_{\mathcal N(\delta)} f \tv
= I_1(\tv) + I_2(\tv) + I_3(\tv)
\quad\forall \, \tv \in H^1(\Nd),
$$
where
\begin{align*}
 I_1(v) &:= \int_{\mathcal N(\delta)} (\bI - \bB_\varepsilon \mu_\varepsilon) \nabla u \cdot \nabla \tv,
  \\
  I_2(v) &:= \int_{\mathcal N(\delta)} (f_\varepsilon \mu_\varepsilon - f)  \tv,
  \\
  I_3(v) &:= \int_{\partial \mathcal N(\delta)} \bB_\varepsilon \nabla u \cdot \nabla d ~ \tv ~\mu_\varepsilon
\end{align*}
with $f_\varepsilon=\wf \circ \bQ_\varepsilon \circ \bP_\varepsilon$.
We now examine these three terms separately.

\medskip\noindent
{\it Step 1: Term $I_1(v)$.}
Since $u$ is constant along the direction $\nabla\wde$, we realize that
$\nabla u = \Pi_\eps\nabla u$ and
Lemma~\ref{e:narrow:consistency_alg} (properties of $\mue$ and $\bBe$) directly yields
$$
\big| I_1(\tv) \big| \lesssim \delta  |d|_{W^2_\infty(\mathcal{N})}  \| \nabla u \|_{L_2(\mathcal N(\delta))} \| \nabla \tv \|_{L_2(\mathcal N(\delta))}.
$$
\medskip\noindent
{\it Step 2: Term $I_2(\tv)$.} Let $-\delta<s<\delta$ and consider the isomorphisms
$$
\bR_s:=\bQ_\eps \circ \bP_{\eps} \circ \bQ_s: \gamma \to \gamma,
\quad
\bR^{-1}_s = \bP_d \circ \bQ_{\eps,s} \circ \bP_\eps : \gamma \to \gamma,
$$
where $\bQ_s:\gamma\to\gamma_s$ is the inverse of $\bP_d$ on $\gamma_s$ and $\bQ_{\eps,s}:\gae\to\gamma_s$ is the inverse of $\bP_\eps$ on $\gamma_s$. Using the coarea formula \eqref{e:coarea} together with $ |\nabla d|=1$ we write
\[
I_2(\tv) = \int_{-\delta}^\delta \int_{\gamma_s} (\fe\mue-f) \tv,
\]
and combining with Lemma \ref{L:area_ratio_distance} (relation between $q$ and
$q_\Gamma$), we obtain
\begin{align*}
  I_2(\tv) &= \int_{-\delta}^\delta \int_\gamma (\wf \circ \bRs)
  (\tv\circ\bQs) (\mue \circ \bQs) (\mu^{-1} \circ \bQs)
  \\
  & - \int_{-\delta}^\delta \int_\gamma \wf (\tv\circ\bQs) (\mu^{-1} \circ \bQs)
  = II_1(\tv) + II_2(\tv) + II_3(\tv),
\end{align*}
where
\begin{align*}
  II_1(\tv) &:= \int_{-\delta}^\delta \int_\gamma (\wf \circ \bRs) (\tv\circ\bQs)
  - \wf (\tv\circ\bQs),
  \\
  II_2(\tv) &:= \int_{-\delta}^\delta \int_\gamma \wf (\tv\circ\bQs) \big(1 - \mu^{-1}\circ\bQs  \big)
  \\
  II_3(\tv) &:= \int_{-\delta}^\delta \int_\gamma (\wf \circ \bRs) (\tv\circ\bQs)
  \Big( (\mue \circ \bQs) (\mu^{-1} \circ \bQs) - 1 \Big).
\end{align*}
We proceed to estimate each term separately. To manipulate $II_1(\tv)$ we first observe that
changing variables from $\gamma$ to $\gamma_s$, $\gamma_s$ to $\gae$,
and $\gae$ to $\gamma$ and invoking Lemma \ref{L:area_ratio_distance} (relation between $q$ and $q_\Gamma$) yields
\begin{equation}
\begin{split}
  \int_\gamma (\wf \circ \bRs) (\tv\circ\bQs) &=
  \int_\gamma (\wf \circ \bQe \circ \bPe \circ \bQs) (\tv\circ\bQs)
  \\ &=
  \int_{\gamma_s} (\wf \circ \bQe \circ \bPe) \tv \mu
  \\ &=
  \int_{\gae} (\wf \circ \bQe) (\tv \, \mu \, \mue^{-1}) \circ \bQes
  \\ &=
  \int_\gamma \wf \, (\tv \, \mu \, \mue^{-1} ) \circ \bQes\circ \bPe \, \mue,
\end{split}
\end{equation}
where
\[
\mu = \det \big(\bI-d  D^2 d \big),
\qquad \mu_{\eps} = \det \big(\bI-\wde  D^2 \wde \big) \, (\nabla d \cdot \nabla\wde).
\]
Therefore, denoting by $\mu_\bR$ the infinitesimal change in area induced by
$\bR_s^{-1}$ on $\gamma$
\[
\mu_\bR := (\mu \, \mue^{-1}) \circ \bQes\circ\bPe \, \mue,
\]
we infer again from the coarea formula \eqref{e:coarea} that
\begin{align*}
II_1(\tv) &= \int_{-\delta}^\delta \int_\gamma \wf \, \Big( (\tv \circ \bQes\circ \bPe) \mu_\bR - \tv\circ\bQs \Big)
\\ & =
\int_{-\delta}^\delta \int_{\gamma_s} \Big( f \, (\tv \circ \bQes\circ \bPe \circ \bP_d) \mu_\bR \, \mu - f  \tv \mu \Big) |\nabla d|
\\ &=
\int_\Nd f \, \big(\tv \circ \bL - \tv \big) \mu_\bR\mu
+ \int_\Nd f \, \tv \big(\mu_\bR -1 \big) \mu,
\end{align*}
where $\bL$ is defined on each $\gamma_s$ by $\bL|_{\gamma_s}  := \bQes\circ \bPe \circ \bP_d: \gamma_s \rightarrow \gamma_s$ .
Notice that the map $\bL :\Nd\to\Nd$ is a
bi-Lipschitz perturbation of the identity with perturbation constant
\[
r = \|\bI - \bL\|_{L_\infty(\Nd)} \lesssim \delta |d|_{W^2_\infty(\Nd)},
\]
because
\begin{align*}
\| \bI -\bP_d\|_{L_\infty(\mathcal N(\delta))} &+ \| \bI - \bQ_{s} \|_{L_\infty(\mathcal N(\delta))}\\& + 
\| \bI -\bP_\eps\|_{L_\infty(\mathcal N(\delta))} + \| \bI -\bQ_{\eps,s}\|_{L_\infty(\mathcal N(\delta))} \lesssim \delta  |d|_{W^2_\infty(\Nd)}.
\end{align*}
Moreover, since $\mu_\bR - 1 = (\mu\mue^{-1}-1) \circ \bQes\circ\bPe \, \mue +
(\mue-1)$, \eqref{nb:mu_estim} and \eqref{nb:mu_estim_d} imply
\[
\| \mu_\bR -1 \|_{L_\infty(\gamma)} \lesssim \delta  |d|_{W^2_\infty(\Nd)}.
\]
These estimates in conjunction in Proposition \ref{p:mol_bulk} (Lipschitz perturbation)
give
\[
\big| II_1(\tv) \big| 
\lesssim \delta  |d|_{W^2_\infty(\Nd)} \|f\|_{L_2(\Nd)} \|\tv\|_{H^1(\Nd)};
\]
we observe that to apply Proposition \ref{p:mol_bulk} we take
$\Omega_1=\Omega_2=\Nd$,
which are Lipschitz domains, and extend $\tv$ to $\Omega=\mathcal{N}$
so that $\|\tv\|_{H^1(\mathcal{N})} \lesssim \|\tv\|_{H^1(\Nd)}$.

Upon utilizing the coarea formula \eqref{e:coarea} once more,  we obtain for $II_2(v)$
\[
II_2(\tv) = \int_\Nd f \tv (1-\mu^{-1})\mu,
\]
so that
\eqref{nb:mu_estim_d} yields
\[
\big| II_2(\tv)  \big| \lesssim
\delta  |d|_{W^2_\infty(\Nd)} \|f\|_{L_2(\Nd)} \|\tv\|_{H^1(\Nd)}.
\]
We proceed similarly for $II_3(v)$ but using in addition that
$$
\int_{-\delta}^\delta \|  \wf \circ \bRs \|_{L_2(\gamma)}^2 ds
\lesssim \int_{-\delta}^\delta \| \wf \|_{L_2(\gamma)} ds \lesssim
\|f\|_{L_2(\Nd)}^2,
$$
and
$$
\|  (\mue \circ \bQs) (\mu^{-1} \circ \bQs) - 1\|_{L_\infty(\gamma)} \lesssim \delta  |d|_{W^2_\infty(\Nd)}
$$
thanks to \eqref{nb:mu_estim} and \eqref{nb:mu_estim_d} again. We thus obtain for
$II_3(v)$ an estimate similar to those for $II_1(\tv)$ and $II_2(\tv)$, whence
\[
\big| I_2(\tv)  \big| \lesssim
\delta  |d|_{W^2_\infty(\Nd)} \|f\|_{L_2(\Nd)} \|\tv\|_{H^1(\Nd)}.
\]

\noindent
 {\it Step 3: Term $I_3(\tv)$.} In view of $\Pi_\eps=\bI-\nabla\wde\otimes\nabla\wde$,
 we first note that
\[
\nabla d^T \bB_\varepsilon \mu_\eps = \nabla d^T \big( \bB_\varepsilon \mue - \Pi_\eps\big)
+ \nabla(d-\wde)^T + \nabla\wde^T \big( 1 - \nabla d \cdot \nabla\wde  \big).
\]
Invoking Lemma~\ref{e:narrow:consistency_alg} (properties of $\mue$ and $\bBe$)
and then Lemma~\ref{L:properties-Pe} (properties of $\wde$) yields
$$
\| \nabla d^T \bB_\varepsilon \mu_\eps \|_{L_\infty(\mathcal N(\delta))} \lesssim
\delta |d|_{W^2_\infty(\Nd)}.
$$
It remains to use trace inequalities  to obtain
\begin{align*}
I_3(\tv) &\lesssim \delta |d|_{W^2_\infty(\Nd)} \| \nabla u \|_{L_2(\partial N(\delta))} \| \tv \|_{L_2(\partial N(\delta))}
  \\&
\lesssim \delta |d|_{W^2_\infty(\Nd)} \| u \|_{H^2(\mathcal N(\delta))} \| \tv \|_{H^1(\mathcal N(\delta))}.
\end{align*}
\medskip\noindent
 {\it Step 4: Normal extension.} Gathering the above estimates we find that
$$
I(\tv) \lesssim \delta |d|_{W^2_\infty(\Nd)} \left(\| u \|_{H^2(\mathcal N(\delta))}+ \| f \|_{L_2(\mathcal N(\delta))}\right) \| \tv \|_{H^1(\mathcal N(\delta))}.
$$
We finally deduce 
$\|f\|_{L_2(\Nd)} \lesssim \delta^{\frac12} |d|_{W^2_\infty(\Nd)}\|\wf\|_{L_2(\gamma)}$
because $f$ is the normal extension of $\wf$ to $\gamma$, and
\[
\|\nabla u\|_{L_2(\Nd)} \lesssim \delta^{\frac12} |d|_{W^2_\infty(\mathcal{N})}
\|\wf\|_{L_2(\gamma)},
\]
upon combining Proposition \ref{P:H2-extension} ($H^2$ extension) with
Lemma \ref{L:regularity} (regularity). This leads to the desired estimate.
\end{proof}

\subsection{Properties of the Narrow Band FEM}

To begin with, we recall the definition of $\bM_{h}: \mathcal N_h(\delta) \rightarrow \mathcal N(\delta)$, that accounts for the mismatch between $\mathcal N_h(\delta)$ and $\mathcal N(\delta)$:
\begin{equation}\label{e:narrow:rewrite_M}
\bM_{h}(\bx) = \bP_d(\bx) + d_h(\bx) \nabla d(\bx) = \bx +(d_h(\bx)-d(\bx)) \nabla d(\bx).
\end{equation}
Note that if $\bx\in\Nhd\subset\mathcal{N}$ then
$\bP_d(\bM_h(\bx))=\bP_d(\bx)$, because this is what happens with all points in the line
$s\mapsto\bx+s\nabla d(\bx)$ within $\mathcal{N}$. Since $|d_h(\bx)|<\delta$,
\[
|d(\bM_h(\bx))| = | \bM_h(\bx) - \bP_d(\bM_h(\bx))| = |d_h(\bx)| \, |\nabla d(\bx)|
< \delta
\]
implies that $\bM_h(\bx)\in\Nd$ and the map $M_h$ is well defined.
Further properties of $\bM_h$ are discussed next.
Before doing so, we mention that the results provided below are not optimal (to avoid technicalities) but are sufficient for our analysis.
We refer to \cite{MR3249369,MR3471100} for higher order estimates.

\begin{lemma}[properties of $\bM_h$]\label{l:narrow:prop_Mh}
Let $\gamma$ be of class $C^2$ and $h$ be sufficiently small.
Then, the map $\bM_h: \mathcal N_h(\delta) \rightarrow \mathcal N(\delta)$ is bi-Lipschitz with
\begin{equation} \label{e:narrow:lipschitz}
\| D\bM_h \|_{L_\infty(\mathcal N_h(\delta))} + \| D\bM_h^{-1} \|_{L_\infty(\mathcal N(\delta))} \leq L
\end{equation}
for some constant $L$ independent of $h$ and $\delta$.
Moreover, there holds
\begin{equation}\label{e:narrow:estim_Mh}
\| \bI - \bM_h \|_{L_\infty(\mathcal N_h(\delta))} + h  \| \bI - D\bM_h \|_{L_\infty(\mathcal N_h(\delta))} \lesssim h^2 | d |_{W^2_\infty(\mathcal N)}
\end{equation}
and
\begin{equation}\label{e:narrow:det}
\| \det(D \bM_h) - 1 \|_{L_\infty(\mathcal N_h(\delta))}  \lesssim h | d |_{W^2_\infty(\mathcal N)}.
\end{equation}
\end{lemma}
\begin{proof}
From the definition~\eqref{e:narrow:rewrite_M} of $\bM_h$ and the interpolation estimate~\eqref{e:narrow:interp}, we find that
$$
| \bx - \bM_h(\bx) | \leq  | d(\bx) - d_h(\bx) | \leq c_I   h^2 | d |_{W^2_\infty(\mathcal N)}.
$$
Furthermore, we compute
$$
D\bM_h(\bx) = \bI + \nabla (d_h(\bx)-d(\bx)) \otimes \nabla d(\bx)
+ (d_h(\bx) - d(\bx)) D^2d(\bx)
$$
to deduce
$$
\| \bI - D\bM_h \|_{L_\infty(\mathcal N_h(\delta))} \leq c_I h | d |_{W^2_\infty(\mathcal N)}
+ c_I h^2 | d |_{W^2_\infty(\mathcal N)}^2.
$$
The above two estimates yield~\eqref{e:narrow:estim_Mh} because $c_Ih| d |_{W^2_\infty(\mathcal N)}\le\frac12$ for $h$ sufficiently small. Exploiting \eqref{e:narrow:estim_Mh}, we also deduce that $\bM_h$ is invertible, bi-Lipschitz and that~\eqref{e:narrow:lipschitz} holds for $h$ sufficiently small.

We are left to show~\eqref{e:narrow:det}.
This follows from $D\det\bA = (\det \bA) \bA^{-1}$ for any invertible matrix
$\bA$ and the first order Taylor expansion of 
$$
\psi(t) := \det\left(\bI -t \left( \nabla (d(\bx)-d_h(\bx)) \otimes \nabla d(\bx) - (d(\bx) - d_h(\bx)) D^2d(\bx)\right)\right)
$$ 
about $t=0$ and evaluated at $t=1$, along with \eqref{e:narrow:estim_Mh} and
the fact that $\psi(1)=\det(D\bM_h(\bx))$. This concludes the proof.
\end{proof}

The previous lemma is instrumental to estimate the consistency error 
\begin{equation}\label{e:narrow:Ih}
  E_h(V):=\int_{\mathcal N_h(\delta)} \nabla u \cdot \nabla V - \int_{\mathcal N_h(\delta)} F V \quad\forall \, V \in \mathbb V(\mathcal T_\delta),
\end{equation}
due to the approximation of the narrow band $\mathcal N(\delta)$ by $\mathcal N_h(\delta)$ and to the use of $F$ in the discrete formulation~\eqref{e:discrete_nb}. Since $\Nd\subset\mathcal{N}$ is of class $C^2$, we assume without loss of generality that the function $u:\Nd\to\mathbb{R}$ constructed in Proposition \ref{P:H2-extension} ($H^2$ extension) extends to $\mathcal{N}$ and satisfies $\|u\|_{H^2(\mathcal{N})} \lesssim \|u\|_{H^2(\Nd)}$. In light of $\Nhd\subset\mathcal{N}$, the consistency error \eqref{e:narrow:Ih} is well defined.

\begin{lemma}[narrow band geometric consistency]\label{l:narrow:consistency_Ih}
Let $\gamma$ be of class $C^2$ and $\delta$ and $h$ satisfy the structural
condition \eqref{delta-h} and be sufficiently small. Let $\wf \in L_{2,\#}(\gamma)$,
$\wu \in H^2(\gamma)$ solve~\eqref{e:weak}, and $u\in H^2(\Nd)$ be the $H^2$ extension of $\wu$ given by~\eqref{e:non_const_ext} with $C\delta \le \eps \le \frac{\delta}{2}$. Let also $F$ be given by \eqref{e:narrow:F}.
Then the consistency error \eqref{e:narrow:Ih} satisfies for all $V \in \mathbb V(\mathcal T_\delta)$
$$
\left| \int_{\mathcal N_h(\delta)} \nabla u  \cdot \nabla V - \int_{\mathcal N_h(\delta)} F V \,\right| \lesssim 
\delta^{3/2} |d|_{W^2_\infty(\mathcal N)} \| \wf \|_{L_2(\gamma)} \|  V \|_{H^1(\mathcal N_h(\delta))}.
$$
\end{lemma}
\begin{proof}
We compare the consistency errors \eqref{e:narrow:Ih} and \eqref{e:narrow:geom_consistency} term by term.

\medskip\noindent
{\it Step 1: Dirichlet integrals.}
Utilizing the change of variables induced by the map $\bM_h:\Nhd\to\Nd$ we end up with
\begin{align*}
  \int_{\Nhd} \nabla u  \cdot \nabla V -
  \int_{\Nd} \nabla u  \cdot \nabla (V\circ\bM_h^{-1}) = I_1(V) + I_2(V) + I_3(V),
\end{align*}    
where
\begin{align*}
I_1(V) &:= \int_{\Nhd} \big(\nabla u - \nabla u\circ\bM_h \big)  \cdot \nabla V
  \, \det\big( D\bM_h  \big)
\\
I_2(V) &:= \int_{\Nhd} \nabla u \cdot \nabla V \, \big( 1- \det(D\bM_h) \big)\\
I_3(V) &:= \int_{\Nd} \nabla u \cdot \left( \nabla V \circ \bM_h^{-1} - \nabla ( V \circ \bM_h^{-1})\right).
\end{align*}
In view of Proposition \ref{p:mol_bulk} (Lipschitz perturbation) and Lemma \ref{l:narrow:prop_Mh} (properties of $\bM_h$), we infer that
\[
\big| I_1(V) \big| , \, \big| I_2(V) \big|
\lesssim h |d|_{W^2_\infty(\mathcal{N})} \|u\|_{H^2(\Nd)} \|V\|_{H^1(\Nhd)}.
\]
Similarly for $I_3(V)$, we observe that
\[
\nabla V \circ \bM_h^{-1} - \nabla \big( V \circ \bM_h^{-1} \big)
= \big(\bI - D \bM_h^{-1} \big) \nabla V \circ \bM_h^{-1},
\]
so that employing Lemma \ref{l:narrow:prop_Mh} (properties of $\bM_h$) 
yields
\[
\big| I_3(V)|
\lesssim h |d|_{W^2_\infty(\mathcal{N})}
\|\nabla u\|_{L_2(\Nd)} \|\nabla V\|_{L_2(\Nhd)}.
\]
Recalling the structural assumption $C_1 h\leq \delta$,  Lemma \ref{L:regularity} (regularity) as well as $\|f\|_{L_2(\Nd)} \lesssim \delta^{\frac12} \|\wf\|_{L_2(\gamma)}$, the estimates for $I_1(V)$, $I_2(V)$ and $I_3(V)$ guarantee
$$
\left| \int_{\Nhd} \nabla u  \cdot \nabla V -
  \int_{\Nd} \nabla u  \cdot \nabla (V\circ\bM_h^{-1}) \right| \lesssim \delta^{3/2} |d|_{W^2_\infty(\mathcal N)} \| \wf \|_{L_2(\gamma)} \|  V \|_{H^1(\mathcal N_h(\delta))}.
$$

\noindent  
{\it Step 2: Forcing.}
Recalling \eqref{e:narrow:F}, we rewrite the forcing term in \eqref{e:narrow:Ih} as 
\begin{align*}
\int_\Nhd F \, V - \int_\Nd f \, V\circ\bM_h^{-1} = II_1(V) + II_2(V),
\end{align*}
where
\begin{align*}
II_1(V) & := \int_\Nd f\, V \circ \bM_h^{-1} \big(\det(D \bM_h)^{-1}-1 \big), 
\\  
II_2(V) & := -\frac{1}{|\mathcal{N}_h(\delta)|} \int_{\mathcal{N}_h(\delta)}  f \circ \bM_h \int_{\mathcal{N}_h(\delta)} V. 
\end{align*}
We make use of \eqref{e:narrow:lipschitz} and \eqref{e:narrow:det}, along with a change of variables, to compute
$$\begin{aligned}
II_1(V) &  \lesssim h|d|_{W_\infty^2(\mathcal{N})} \|f\|_{L_2(\Nd)} \|V \circ \bM_h^{-1} \|_{L_2(\Nd)}
  \\ & \lesssim h|d|_{W_\infty^2(\mathcal{N})} \|f\|_{L_2(\Nd)} \|V\|_{H^1(\Nd)}.
  \end{aligned}
  $$
Since $|\mathcal{N}_h(\delta)| \simeq |\Nd| \simeq \delta$, the first equivalence resulting from \eqref{e:narrow:lipschitz} and the second from the coarea formula, using \eqref{e:narrow:det} again we obtain
  $$
  \begin{aligned}
  II_2(V ) &\lesssim \delta^{-1/2} \|V\|_{L_2(\mathcal{N}_h(\delta))} \Big | \int_{\mathcal{N}_h(\delta)} f \circ \bM_h \big({\rm det} (D\bM_h)-1 \big) - \int_{\Nd} f \Big | 
  \\ & \lesssim \|V\|_{L_2(\mathcal{N}_h(\delta))} \Big ( h|d|_{W_\infty^2(\mathcal{N})} \|f\|_{L_2(\Nd)} + \delta^{-1/2}\Big |\int_{\Nd} f \Big| \Big) .
  \end{aligned}
  $$
To estimate the rightmost term we exploit the fact that $\wf$ has a vanishing
mean on $\gamma$. Using the coarea formula \eqref{e:coarea}, we see that
\[\begin{aligned}
\int_\Nd f = \int_{-\delta}^\delta \int_{\gamma_s} f &= \int_{-\delta}^\delta \int_\gamma \wf
\mu_s = \int_{-\delta}^\delta \int_\gamma \wf \big(\mu_s - 1  \big) 
\\ & \le 2 \delta \|\mu_s-1\|_{L_\infty(\gamma \times [-\delta, \delta])} \|\wf\|_{L_2(\gamma)},
\end{aligned}
\]
where $\mu_s = \det \big( \bI - d D^2 d \big)^{-1} \circ \bQs$ according to
Lemma \ref{L:area_ratio_distance} (relation between $q$ and $q_\Gamma$).
Remark \ref{r:narrow:mu} (estimate of $\mu$) in turn leads to
\[
\Big| \int_\Nd f \, \Big| \lesssim \delta^{2} |d|_{W^2_\infty(\mathcal{N})} \|\wf\|_{L_2(\gamma)}.  
\]
%
%
Consequently, collecting the previous estimates and using the structural assumption $C_1h \le \delta$ again readily gives
\[
\big| II_2(V) \big| \lesssim \delta^{\frac32} |d|_{W^2_\infty(\mathcal{N})}
\|\wf\|_{L_2(\gamma)} \|V\|_{L_2(\Nhd)}.
\]
%

Gathering the bounds for $II_i(V)$ for $i=1,2$, we discover
$$
\left| \int_\Nhd F \, V - \int_\Nd f \, V\circ\bM_h^{-1}\right| \lesssim  \delta^{\frac32}
|d|_{W^2_\infty(\mathcal{N})} \|\wf\|_{L_2(\gamma)} \|V\|_{H^1(\Nhd)}.
$$
{\it Step 3: Conclusion.}
The assertion follows from the bounds derived in Steps~1 and~2 together with the estimate \eqref{e:narrow:geom_consistency} of Lemma \ref{t:narrow:geom_consistency} (narrow band PDE consistency) with $\tv =V\circ \bM_h^{-1} \in H^1(\mathcal N(\delta))$. The proof is complete.
\end{proof}

\subsection{A Priori Error Estimates}

All of the ingredients for a-priori error analysis in the narrow band norm are now in place. We recall that the extension $u:\Nd\to\mathbb{R}$ constructed in Proposition \ref{P:H2-extension} ($H^2$ extension) is further extended to $\mathcal{N}$ and satisfies
\begin{equation}\label{extension-N}
\|u\|_{H^2(\mathcal{N})} \lesssim \|u\|_{H^2(\Nd)} \lesssim \delta^{\frac12}
|d|_{W^2_\infty(\mathcal{N})}  \|\wu\|_{H^2(\gamma)}.
\end{equation}
\begin{theorem}[a-priori error estimate]\label{t:narrow:aprior}
Let $\gamma$ be of class $C^2$ and $\delta$ and $h$ satisfy the structural condition
\eqref{delta-h} and be sufficiently small.
Let  $\wu \in H^1_\#(\gamma)$ be defined by~\eqref{e:weak} with $\wf \in L_{2,\#}(\gamma)$ and $u$ be its extension given by~\eqref{e:non_const_ext} with $C\delta \le \eps \le \frac{\delta}{2}$.
Let $U \in \mathbb V_\#(\mathcal T_\delta)$ be the finite element solution to~\eqref{e:discrete_nb} with $F$ given in~\eqref{e:narrow:F}.
Then, the following error estimate is valid
\begin{equation*}\label{e:narrow:best_apprx}
\| \nabla (u - U) \|_{L_2(\mathcal N_h(\delta))} \lesssim
\inf_{V \in \mathbb V(\mathcal T_\delta)} \| \nabla (u - V)\|_{L_2(\mathcal N_h(\delta))} + h^{\frac32} |d|_{W^2_\infty(\mathcal{N})} \| \wf \|_{L_2(\gamma)},
\end{equation*}
with hidden constant independent of $h$ and $\delta$.
\end{theorem}
\begin{proof}
The proof consists of a Strang type argument.
For any $V \in \mathbb V(\mathcal T_\delta)$ the equation~\eqref{e:discrete_nb} satisfied by $U$ and the definition~\eqref{e:narrow:Ih} of $E_h(.)$ give
$$
\| \nabla (V-U)\|_{L_2(\mathcal N_h(\delta))}^2 = \int_{\mathcal N_h(\delta)} \nabla (V-u) \cdot \nabla (V-U) + E_h(V-U).
$$
Invoking Lemma~\ref{l:narrow:consistency_Ih} (narrow band geometric consistency), together with the structural assumption \eqref{delta-h}, yields
$$
\| \nabla (V-U)\|_{L_2(\mathcal N_h(\delta))} \leq \|  \nabla (V-u ) \|_{L_2(\mathcal N_h(\delta))} + c h^{\frac32} |d|_{W^2_\infty(\mathcal{N})} \| \wf \|_{L_2(\gamma)}.
$$
The desired error estimate follows from a triangle inequality.
\end{proof}
\begin{corollary}[rate of convergence in $\Nhd$]\label{c:narrow:error} 
Under the assumptions of Theorem~\ref{t:narrow:aprior} we have
$$
\| \nabla (u  - U) \|_{L_2(\mathcal N_h(\delta))} \lesssim h^{\frac32}
|d|_{W^2_\infty(\mathcal{N})}  \| \wf \|_{L_2(\gamma)}.
$$
\end{corollary}
\begin{proof}
In view of \eqref{extension-N} standard polynomial interpolation theory gives
$$
\| \nabla (u  - I_h^{\textrm{sz}}u) \|_{L_2(\mathcal N_h(\delta))}
\lesssim h \| u \|_{H^2(\mathcal N)}
\lesssim h \| u \|_{H^2(\mathcal N(\delta))},
$$
where $I_h^{\textrm{sz}}u$ is the Scott-Zhang interpolant of $u$.
It remains to use Proposition \ref{P:H2-extension} ($H^2$ extension)
and Lemma \ref{L:regularity} (regularity) to arrive at
$$
\| \nabla (u  -I_h^{\textrm{sz}} u)\|_{L_2(\mathcal N_h(\delta))} \lesssim h^{\frac32}
|d|_{W^2_\infty(\mathcal{N})}  \| \wf \|_{L_2(\gamma)}.
$$
The asserted estimate follows from Theorem~\ref{t:narrow:aprior} (a-priori estimate).
\end{proof}

In addition, we follow \cite{MR3471100} to deduce a rate of convergence for
$\| \nabla_\gamma (\wu -U)\|_{L_2(\gamma)}$.
\begin{corollary}[rate of convergence on $\gamma$]\label{C:rate-gamma}
Under the assumptions of Theorem~\ref{t:narrow:aprior} we have  
$$
\| \nabla_\gamma ( \widetilde u - U ) \|_{L_2(\gamma)} \lesssim h \| \wf \|_{L_2(\gamma)}.
$$
\end{corollary}
\begin{proof}
We recall the scaled trace inequality \eqref{curved_trace_est}: for a bulk triangulation $\mathcal T$  there exists a constant $C$ only depending on the mesh shape regularity constant of $\mathcal T$ such that for $T \in \mathcal T_\delta$ and $\tv \in H^1(T)$, one has
\begin{equation*}\label{e:bulk_trace}
  \| \tv \|_{L_2(T \cap \gamma)}^2 \leq C \left( h_T^{-1} \| \tv \|_{L_2(T)}^2
  + h_T \| \nabla \tv \|_{L_2(T)}^2 \right),
\end{equation*}
where $h_T = \mathrm{diam}(T)$.
We apply this inequality with $\tv = \nabla(u-U)$, and $h_T\approx h$ sufficiently small,
to obtain
$$
\| \nabla (u -U) \|_{L_2(T\cap \gamma)}^2 \lesssim \left( h^{-1}
\| \nabla(u-U) \|_{L_2(T)}^2 + h | u |_{H^2(T)}^2 \right).
$$
Summing up over all $T \in \mathcal T_\delta$ with non-empty intersection with $\gamma$, Proposition~\ref{P:H2-extension} ($H^2$ extension) and Corollary~\ref{c:narrow:error} (rate of convergence in $\Nhd$) give
$$
\| \nabla_\gamma (\widetilde u-U) \|_{L_2(\gamma)} \leq \| \nabla (u-U) \|_{L_2(\gamma)}
\lesssim h |d|_{W^2_\infty(\mathcal{N})} \Big( \| \wf \|_{L_2(\gamma)} +
| \wu |_{H^2(\gamma)} \Big).
$$
The desired estimate follows from Lemma \ref{L:regularity} (regularity).
\end{proof}

\bibliographystyle{alpha}
\newcommand{\etalchar}[1]{$^{#1}$}

\end{document}